\documentclass[12pt,twosides]{amsart}
\usepackage{amssymb,amsmath,amsthm, amscd, enumerate, mathrsfs}
\usepackage{graphicx, hhline}
\usepackage[all]{xy}
\usepackage{braket}
\usepackage[usenames]{color}
\usepackage{hyperref}
\usepackage{fancyhdr}
\usepackage[top=30truemm,bottom=30truemm,left=30truemm,right=30truemm]{geometry}

\hypersetup{colorlinks=true}

\title{Iitaka fibrations for dlt pairs polarized by a nef and log big divisor}
\author{Kenta Hashizume}
\date{2022/09/04}
\address{Department of Mathematics, Graduate School of Science, Kyoto University, Kyoto 606-8502, Japan}
\email{hkenta@math.kyoto-u.ac.jp}

\newtheorem{thm}{Theorem}[section]

\newtheorem{lem}[thm]{Lemma}
\newtheorem{cor}[thm]{Corollary}

\theoremstyle{definition}
\newtheorem{defn}[thm]{Definition}
\newtheorem{rem}[thm]{Remark}

\newtheorem{exam}[thm]{Example}
\newtheorem*{ack}{Acknowledgments} 
 
\newtheorem*{b-divisor}{b-divisors} 
 
\newtheorem*{g-pair}{Generalized pairs} 
\newtheorem*{adj-g-pair}{Divisorial adjunction for generalized pairs} 
\newtheorem*{mmp-g-pair}{MMP for generalized pairs}

\newtheorem{step1}{Step}
\newtheorem{step2}{Step}
\newtheorem{step3}{Step}

\newtheorem*{claim*}{Claim}
\begin{document}

\maketitle

\begin{abstract}
We study lc pairs polarized by a nef and log big divisor. 
After proving the minimal model theory for projective lc pairs polarized by a nef and log big divisor, we prove the effectivity of the Iitaka fibrations and some boundedness results for dlt pairs polarized by a nef and log big divisor. 
\end{abstract}

\tableofcontents

\section{Introduction}\label{sec1}

Throughout this paper, we will work over the complex number field $\mathbb{C}$. 

In this paper, we study triples $(X,\Delta,M)$ such that $(X,\Delta)$ is a projective lc pair and $M$ is a nef $\mathbb{R}$-divisor on $X$ that is log big with respect to $(X,\Delta)$ (see Definition \ref{defn--abund}). 
By definition, the log bigness coincides with the bigness in the case of klt pairs, and 
 a nef $\mathbb{R}$-divisor $D$ on a normal projective variety $V$ is log big with respect to a projective lc pair $(V,B)$ if and only if $(D^{{\rm dim}V})>0$ and $(S\cdot D^{n})>0$ for any $n$-dimensional lc center $S$ of $(V,B)$. 
Especially, all ample divisors are nef and log big with respect to any projective lc pair. 
The triples $(X,\Delta,M)$ can be regarded as generalized pairs defined by Birkar--Zhang \cite{bz}. 
Generalized pairs are the main objects in the study of lc-trivial fibrations. 
Because the log bigness is one of the special cases of the property of being log abundant (see Definition \ref{defn--abund}), the recent progress of the canonical bundle formula for lc-trivial fibrations by Floris--Lazi\'c \cite{floris-lazic} (see also \cite[Theorem 1.2]{hu-lctrivial-b-div} by Hu) implies that the above triples $(X,\Delta,M)$ appear as the structures on the base varieties of special lc-trivial fibrations. 
The goal of this paper is to develop the theory of the above triples $(X,\Delta,M)$ from viewpoints of the minimal model theory. 

We start with the minimal model theory for the above triples $(X,\Delta,M)$. 
If $(X,\Delta)$ is a klt pair, then we can run a $(K_{X}+\Delta+M)$-MMP and we get a birational contraction $\phi \colon X\dashrightarrow X'$ such that $\phi_{*}(K_{X}+\Delta+M)$ is semi-ample or $X'$ has the structure of a Mori fiber space with respect to $\phi_{*}(K_{X}+\Delta+M)$. 
This fact is a consequence of the celebrated result by Birkar--Cascini--Hacon--M\textsuperscript{c}Kernan \cite{bchm}. 
Even if $(X,\Delta)$ is not klt, the same statement is known by the author and Hu \cite{hashizumehu} under the assumption of the ampleness of $M$. 
In the general case, the minimal model theory for $K_{X}+\Delta+M$ is not known. 
However, the abundance theorem (more strongly, the effective base point free theorem in the $\mathbb{Q}$-boundary case) was proved by Fujino \cite{fujino-eff-basepointfree} if $K_{X}+\Delta+M$ is nef and Cartier. 
Moreover, Fujino \cite{fujino-fund-slc} proved the abundance theorem for $K_{X}+\Delta+M$ in the case where $(X,\Delta)$ is slc, a more general situation  than the lc case. 

The first main result of this paper is the minimal model theory for $K_{X}+\Delta+M$. 

\begin{thm}\label{thm--mmp-neflogbig-main-intro}
Let $(X,\Delta)$ be a projective lc pair, and let $M$ be a nef $\mathbb{R}$-divisor on $X$ that is log big with respect to $(X,\Delta)$. 
Suppose that $M$ is a finite $\mathbb{R}_{>0}$-linear combination of nef $\mathbb{Q}$-Cartier divisors on $X$. 
Then there exists a birational contraction $\phi \colon X\dashrightarrow X'$, which is a sequence of steps of a $(K_{X}+\Delta+M)$-MMP, such that $X'$ satisfies one of the following conditions. 
\begin{itemize}
\item
$\phi_{*}(K_{X}+\Delta+M)$ is semi-ample, or 
\item
there is a contraction $X' \to Z$ to a projective variety $Z$ such that ${\rm dim}Z <{\rm dim}X'$, $-\phi_{*}(K_{X}+\Delta+M)$ is ample over $Z$, and the relative Picard number is one. 
\end{itemize}
Note that $K_{X'}+\phi_{*}\Delta$ is $\mathbb{R}$-Cartier and $(X',\phi_{*}\Delta)$ is an lc pair. 
\end{thm}

Theorem \ref{thm--mmp-neflogbig-main-intro} is proved by generalizing \cite[Theorem 3.5]{has-finite} to the context of generalized pairs. 
For the generalization, see Theorem \ref{thm--mmp-main-1}. 
To prove Theorem \ref{thm--mmp-neflogbig-main-intro}, we only need to prove the termination of the MMP. 
It is because the abundance theorem in the situation of Theorem \ref{thm--mmp-neflogbig-main-intro} was already proved by Fujino \cite{fujino-eff-basepointfree}. 
The proof in \cite{fujino-eff-basepointfree} heavily depends on a vanishing theorem for quasi-log schemes (see \cite[Theorem 5.7.3]{fujino-book}). 
On the other hand, for the termination of the MMP of Theorem \ref{thm--mmp-neflogbig-main-intro}, we need to discuss the MMP in the framework of generalized pairs (see, for example, \cite{hanli} by Han--Li) and we need the techniques developed in \cite{has-finite}.
By Theorem \ref{thm--mmp-main-1}, we also know the existence of a minimal model for generalized lc pairs with a polarization (Theorem \ref{thm--mmp-polarized-gpair}), which was mentioned in \cite{has-nonvan-gpair}.  
We note that we do not need the existence of flips for $\mathbb{Q}$-factorial generalized lc pairs by Hacon--Liu \cite{hacon-liu} in the proof of Theorem \ref{thm--mmp-neflogbig-main-intro} or Theorem \ref{thm--mmp-polarized-gpair}. 

We apply Theorem \ref{thm--mmp-neflogbig-main-intro} to study the effectivity of the Iitaka fibrations for the above triples $(X,\Delta,M)$ such that $(X,\Delta)$ is a dlt pair. 
The effectivity of the Iitaka fibrations for higher dimensional lc pairs was studied by Fujino--Mori \cite{fujino-mori}, Viehweg--Zhang \cite{viehweg-zhang}, Birkar--Zhang \cite{bz}, and Hacon--Xu \cite{haconxu-cy-bound}. 
Those results are based on the canonical bundle formula for the Iitaka fibrations \cite{fujino-mori}. 
If we know the effectivity of the Iitaka fibrations for lc pairs, then we know the existence of an integer $m$, which is independent of the variety, such that every projective lc pair $(Y,0)$ of fixed dimension with $\kappa(Y,K_{Y})\geq 0$ satisfies $H^{0}(Y, \mathcal{O}_{Y}( m K_{Y}))\neq 0$. 
Especially, if $K_{Y}$ is numerically trivial, then $mK_{Y}$ is Cartier. 
Hence, the effectivity of the Iitaka fibrations for lc pairs is stronger than the boundedness of the Cartier indices of numerically trivial lc varieties of fixed dimension. 

In the case of klt pairs whose boundary divisors are big, the effectivity of the Iitaka fibrations is known by Hacon--Xu \cite{haconxu-cy-bound}. 
In the non-klt case, the effectivity of the Iitaka fibrations for lc pairs of log general type was proved by Hacon--M\textsuperscript{c}Kernan--Xu \cite{hmx-acc}, and their result was  generalized to the context of generalized lc pairs (see Definition \ref{defn--gen-pair}) by Birkar--Zhang \cite{bz}. 
The theorem by Birkar--Zhang \cite{bz} implies the effectivity of the Iitaka fibrations for $K_{X}+\Delta+M$ of the triples $(X,\Delta,M)$ such that $K_{X}+\Delta+M$ are big and $M$ are finite $\mathbb{R}_{>0}$-linear combinations of nef $\mathbb{Q}$-Cartier divisors. 
In this paper, we study the case where $K_{X}+\Delta+M$ has an intermediate Iitaka dimension. 

\begin{thm}\label{thm--eff-iitaka-intro}
Let $d$ and $p$ be positive integers, and let $\Phi \subset \mathbb{Q}_{\geq0}$ be a DCC set. 
Then there exists a positive integer $m$, depending only on $d$, $p$ and $\Phi$, satisfying the following. 
Let $(X,\Delta)$ be a projective dlt pair and let $M$ be a nef $\mathbb{Q}$-divisor on $X$ such that 
\begin{itemize}
\item
${\rm dim}X=d$, 
\item
the coefficients of $\Delta$ belong to $\Phi$, 
\item
$pM$ is Cartier, and 
\item
$M$ is log big with respect to $(X,\Delta)$ and $K_{X}+\Delta+M$ is pseudo-effective. 
\end{itemize}
Then, the linear system $|\lfloor lm(K_{X}+\Delta+M)\rfloor |$ is not empty and it defines a map birational to the Iitaka fibration for every positive integer $l$. 
\end{thm} 

More precisely, we prove the existence of a generalized lc pair such that the coefficients of the boundary divisor and the b-Cartier index of the nef part are controlled. 

\begin{thm}\label{thm--base-iitaka-intro}
Let $d$ be a positive integer, and let $\Phi \subset \mathbb{Q}_{\geq0}$ be a DCC set. 
Then there exist positive integers $n$ and $q$ and a DCC set $\Omega \subset\mathbb{Q}_{\geq 0}$, depending only on $d$ and $\Phi$, satisfying the following. 
Let $(X,\Delta)$ be a projective dlt pair and let $M=\sum_{i}\mu_{i}M_{i}$ be a nef $\mathbb{R}$-divisor on $X$, which is the sum of nef and log big Cartier divisors $M_{i}$ with respect to $(X,\Delta)$, such that 
\begin{itemize}
\item
${\rm dim}X=d$, 
\item
the coefficients of $\Delta$ and $\mu_{i}$ belong to $\Phi$, and 
\item
$K_{X}+\Delta+M$ is pseudo-effective but not big. 
\end{itemize}
Then the variety $Z:={\boldsymbol{\rm Proj}}\bigoplus_{l \in \mathbb{Z}_{\geq 0}}H^{0}(X,\mathcal{O}_{X}(\lfloor l(K_{X}+\Delta+M)\rfloor))$, which is well defined by Theorem \ref{thm--mmp-neflogbig-main-intro}, has the structure of a generalized lc pair $(Z,\Delta_{Z},\boldsymbol{\rm N})$ such that 
\begin{itemize}
\item
$H^{0}(X,\mathcal{O}_{X}(\lfloor ln(K_{X}+\Delta+M)\rfloor)) \simeq H^{0}(Z,\mathcal{O}_{Z}(\lfloor ln(K_{Z}+\Delta_{Z}+\boldsymbol{\rm N}_{Z})\rfloor))$ for every positive integer $l$, 
\item
the coefficients of $\Delta_{Z}$ belong to $\Omega$, and 
\item
$q \boldsymbol{\rm N}$ is b-Cartier. 
\end{itemize}
\end{thm} 
Theorem \ref{thm--eff-iitaka-intro} directly follows from Theorem \ref{thm--base-iitaka-intro} and \cite[Theorem 1.3]{bz}. 

We further study the structures of the Iitaka fibrations. 
In \cite{li-bounded}, Li defined the Iitaka volumes of $\mathbb{Q}$-Cartier divisors on normal projective varieties (see Definition \ref{defn--iitakavol}). 
In our context, with notations as in Theorem \ref{thm--base-iitaka-intro}, the Iitaka volume of $K_{X}+\Delta+M$ coincides with the volume of $K_{Z}+\Delta_{Z}+\boldsymbol{\rm N}_{Z}$. 
The following theorem is the DCC for the Iitaka volumes of $K_{X}+\Delta+M$. 

\begin{thm}\label{thm--dcc-iitakavol-intro}
Let $d$ be a positive integer and $\Phi \subset \mathbb{Q}_{\geq0}$ a DCC set. 
Then there exists a DCC set $\Omega \subset \mathbb{Q}_{>0}$, depending only on $d$ and $\Phi$, satisfying the following. 
Let $(X,\Delta)$ be a projective dlt pair and let $M=\sum_{i}\mu_{i}M_{i}$ be a nef $\mathbb{R}$-divisor on $X$, which is the sum of nef and log big Cartier divisors $M_{i}$ with respect to $(X,\Delta)$, such that 
\begin{itemize}
\item
${\rm dim}X=d$, 
\item
the coefficients of $\Delta$ and $\mu_{i}$ belong to $\Phi$, and 
\item
$K_{X}+\Delta+M$ is pseudo-effective. 
\end{itemize}
Then the Iitaka volume ${\rm Ivol}(K_{X}+\Delta+M)$ is an element of $\Omega$. 
\end{thm} 

We give remarks on general fibers of the Iitaka fibrations. 
As shown in the example below, compared to \cite[Theorem 1.3]{haconxu-cy-bound}, we cannot prove a kind of the boundedness of general fibers of the Iitaka fibrations. 
\begin{exam}[see Example \ref{exam--unbounded}] 
We put $d=2$, $p=1$, and $\Phi=\{1\}$. 
We consider the category $\mathcal{C}$ whose objects are the triples $(X,\Delta,M)$, where $(X,\Delta)$ is a projective dlt pair and $M$ is a nef Cartier divisor on $X$ which is log big with respect to $(X,\Delta)$, such that ${\rm dim}X=2$, $\Delta$ is a Weil divisor, and $K_{X}+\Delta+M\sim_{\mathbb{Q}}0$. 
Then, for some set 
 $$\mathcal{D}\subset \{X\,|\,\text{$(X,\Delta,M)$ is a object of $\mathcal{C}$ for some $\Delta$ and $M$}\},$$ 
$\mathcal{D}$ is unbounded. 
Indeed, for each positive integer $n$, defining $X_{n}:=\mathbb{P}_{\mathbb{P}^{1}}(\mathcal{O}_{\mathbb{P}^{1}}\oplus \mathcal{O}_{\mathbb{P}^{1}}(-n))$ then there are divisors $\Delta_{n}$ and $M_{n}$ on $X_{n}$ such that $(X_{n},\Delta_{n},M_{n})$ is an object of $\mathcal{C}$. 
However, we see that the set $\{X_{n}\}_{n\geq 1}$ is unbounded. 
For details, see Example \ref{exam--unbounded}. 
\end{exam}

Despite the example, it is possible to prove the effective non-vanishing (and the boundedness of complements) for general fibers of the Iitaka fibrations for $K_{X}+\Delta+M$. 

\begin{thm}\label{thm--compl-dlt-intro} Let $d$ be a positive integer, and let $\Phi \subset \mathbb{Q}_{\geq0}$ be a DCC set. Then there exists a positive integer $n$, depending only on $d$ and $\Phi$, satisfying the following. Let $(X,\Delta)$ be a projective dlt pair and let $M=\sum_{i}\mu_{i}M_{i}$ be a nef $\mathbb{R}$-divisor on $X$, which is the sum of nef and log big Cartier divisors $M_{i}$ with respect to $(X,\Delta)$, such that \begin{itemize} \item ${\rm dim}X=d$, \item the coefficients of $\Delta$ and $\mu_{i}$ belong to $\Phi$, and \item $K_{X}+\Delta+M$ is pseudo-effective. \end{itemize} Let $\phi\colon X \dashrightarrow X'$ be finite steps of a $(K_{X}+\Delta+M)$-MMP such that $\phi_{*}(K_{X}+\Delta+M)$ is semi-ample (see Theorem \ref{thm--mmp-neflogbig-main-intro}), and let $F$ be a general fiber of the contraction induced by $\phi_{*}(K_{X}+\Delta+M)$. Put $\Delta'_{F}=\phi_{*}\Delta|_{F}$ and $M'_{F}=\phi_{*}M|_{F}$. 
Then 
\begin{itemize}
\item
$n(K_{F}+\Delta'_{F}+M'_{F})\sim 0$, and 
\item
there is $B'_{F} \in |nM'_{F}|$ such that $(F,\Delta'_{F}+\frac{1}{n}B'_{F})$ is an lc pair.
\end{itemize} \end{thm}

Theorem \ref{thm--compl-dlt-intro} is a consequence of the boundedness of complements for a special kind of dlt pairs (Theorem \ref{thm--compl-dlt}). 
In Example \ref{exam-eff-nonvan-counter}, we show that the log bigness of $M_{i}$ cannot be relaxed to the bigness in Theorem \ref{thm--eff-iitaka-intro} and Theorem \ref{thm--compl-dlt-intro}. 
It is not clear that the dlt property can be generalized to log canonicity. 

Finally, we discuss the effective finite generation of the generalized log canonical rings for the above $(X,\Delta,M)$. 
A kind of the effective finite generation for projective klt threefolds $(V,B)$ was proved by Cascini--Zhang \cite{cz} when $K_{V}+B$ is big or $B$ is nef and big. 
In this paper, we prove the following theorem. 

\begin{thm}\label{thm--base-boundedness-intro}
Let $d$ and $p$ be positive integers, and let $v$ be a positive real number. 
Then there exists a positive integer $m$, depending only on $d$, $p$, and $v$, satisfying the following. 
Let $(X,\Delta)$ be a projective dlt pair, let $M$ be a nef $\mathbb{Q}$-divisor on $X$ that is log big with respect to $(X,\Delta)$, and let $\phi \colon X \dashrightarrow X'$ be finite steps of a $(K_{X}+\Delta+M)$-MMP such that 
\begin{itemize}
\item
${\rm dim}X=d$, 
\item
$pM$ is Cartier and $p\phi_{*}(K_{X}+\Delta+M)$ is nef and Cartier, and 
\item
the Iitaka volume ${\rm Ivol}(K_{X}+\Delta+M)$ is less than or equal to $v$. 
\end{itemize}
Putting $R_{l}=H^{0}(X,\mathcal{O}_{X}(\lfloor l(K_{X}+\Delta+M)\rfloor))$ for every $l \in \mathbb{Z}_{\geq 0}$, then 
\begin{itemize}
\item
$\bigoplus_{l \in \mathbb{Z}_{\geq 0}} R_{lm}$ 
is generated by $R_{m}$ as a graded $\mathbb{C}$-algebra, and 
\item
the variety $Z:={\boldsymbol{\rm Proj}}\bigoplus_{l \in \mathbb{Z}_{\geq 0}}R_{l}$ belongs to a bounded family $\mathfrak{F}$ that depends only on $d$, $p$, and $v$. 
\end{itemize}
\end{thm} 

The proof of the theorem heavily depends on the effective base point free theorem for spacial generalized dlt pairs (Theorem \ref{thm--eff-basepoint-free-gen-dlt}) and the boundedness of generalized lc pairs in a special case (Lemma \ref{lem--bound-veryampleindex}), thus we will need some extra works to remove the upper bound of the Cartier index of $\phi_{*}(K_{X}+\Delta+M)$ in Theorem \ref{thm--base-boundedness-intro}. 

The contents of the paper are as follows: 
In Section \ref{sec2}, we collect notations and definitions.
In Section \ref{sec3}, we prove Theorem \ref{thm--mmp-neflogbig-main-intro}. 
In Section \ref{sec4}, we study the effectivity of the Iitaka fibrations for projective dlt pairs polarized by nef and log big divisors, and we prove Theorem \ref{thm--eff-iitaka-intro} and Theorem \ref{thm--base-iitaka-intro}. 
In Section \ref{sec5}, we show some boundedness results, and we prove Theorem \ref{thm--dcc-iitakavol-intro}, Theorem \ref{thm--compl-dlt-intro}, and Theorem \ref{thm--base-boundedness-intro}.
In Section \ref{sec6}, which is an appendix, we discuss the definition of generalized dlt pairs. 

\begin{ack}
The author was partially supported by JSPS KAKENHI Grant Numbers JP16J05875,  JP19J00046, JP22K13887. 
He thanks Professor Osamu Fujino for answering questions. 
He thanks the referee for carefully reading the manuscript and a lot of comments. 
\end{ack}

\section{Preliminaries}\label{sec2}
In this section, we collect notations and definitions. 
We will freely use the notations in \cite{kollar-mori} and \cite{bchm}. 

\subsection{Divisors, morphisms and singularities}
We collect notations and definitions on divisors, morphisms and singularities of  generalized pairs. 

We will use the standard definitions of nef $\mathbb{R}$-divisor, ample $\mathbb{R}$-divisor, semi-ample $\mathbb{R}$-divisor, and pseudo-effective $\mathbb{R}$-divisor. 
All big $\mathbb{R}$-divisors in this paper are $\mathbb{R}$-Cartier. 
For a morphism $f\colon X\to Y$ and an $\mathbb{R}$-Cartier divisor $D$ on $Y$, we sometimes denote $f^{*}D$ by $D|_{X}$. 
For a prime divisor $P$ over $X$, the image of $P$ on $X$ is denoted by $c_{X}(P)$. 
A projective morphism $f\colon X\to Y$ of varieties is called a {\em contraction} if $f_{*}\mathcal{O}_{X} \simeq \mathcal{O}_{Y}$. 
For a variety $X$ and an $\mathbb{R}$-divisor $D'$ on it, a {\em log resolution of} $(X,D')$ denotes a projective birational morphism $f\colon Y\to X$ from a smooth variety $Y$ such that the exceptional locus ${\rm Ex}(f)$ of $f$ is pure codimension one and ${\rm Ex}(f)\cup {\rm Supp}f_{*}^{-1}D'$ is an snc divisor.

\begin{defn}
Let $a$ be a real number. 
Then we define $\lfloor a \rfloor$ to be the unique integer satisfying $ \lfloor a \rfloor \leq a <\lfloor a \rfloor+1$, and we define $\{a\}$ by $\{a\}=a-\lfloor a \rfloor$. 

Let $X$ be a normal variety, and let $D$ be an $\mathbb{R}$-divisor on $X$. 
Let $D=\sum_{i} d_{i}D_{i}$ be the prime decomposition. 
Then we define 
$\lfloor D \rfloor:=\sum_{i} \lfloor d_{i} \rfloor D_{i}$, $\{D\}:=\sum_{i} \{d_{i}\}D_{i}$, and $\lceil D \rceil:=-\lfloor -D \rfloor$. 
By definition, we have $D=\lfloor D \rfloor + \{D\}$. 
\end{defn}

\begin{defn}[Hyperstandard set]\label{defn--hyperstandardset}
Let $\mathfrak{R}\subset \mathbb{R}$ be a subset. 
Throughout this paper, $\Phi(\mathfrak{R})$ is defined by
$$\Phi(\mathfrak{R})=\left\{1-\frac{a}{r}\;\middle| \;a\in \mathfrak{R}, \, r\in \mathbb{Z}_{>0}\right\},$$
and we call it a {\em hyperstandard set associated to $\mathfrak{R}$}. 
\end{defn}

\begin{defn}[Asymptotic vanishing order]\label{defn--asy-van-ord}
Let $X$ be a normal projective variety and $D$ a pseudo-effective $\mathbb{R}$-Cartier divisor on $X$. 
Let $P$ be a prime divisor over $X$. 
We define {\em asymptotic vanishing order} of $P$ with respect to $D$, which we denote $\sigma_{P}(D)$, as follows: 
We take a projective birational morphism $f\colon Y \to X$ such that $P$ appears as a prime divisor on $Y$. 
When $D$ is big, we define $\sigma_{P}(D)$ by
$$\sigma_{P}(D)={\rm inf}\!\set{{\rm coeff}_{P}(D')|f^{*}D\sim_{\mathbb{R}}D'\geq0}.$$
When $D$ is not necessarily big, $\sigma_{P}(D)$ is defined by
$$\sigma_{P}(D)=\underset{\epsilon\to+0}{\rm lim}\sigma_{P}(D+\epsilon A)$$
for an ample $\mathbb{R}$-divisor $A$ on $X$. 
It is easy to see that $\sigma_{P}(D)$ is independent of $f\colon Y\to X$. 
We have $\sigma_{P}(D)\geq0$, and $\sigma_{P}(D)$ is also independent of $A$ (\cite[III, 1.5 (2) Lemma]{nakayama}). 
We can easily check $\sigma_{P}(D)={\rm sup}\!\set{\sigma_{P}(D+H)|H \text{ is ample}}$. 
\end{defn}

\begin{defn}[Negative part of Nakayama--Zariski decomposition] 
For $X$ and $D$ as in Definition \ref{defn--asy-van-ord}, the {\em negative part of Nakayama--Zariski decomposition} of $D$, denoted by $N_{\sigma}(D)$, is defined by
$$N_{\sigma}(D)=\sum_{\substack {P:{\rm \,prime\,divisor}\\{\rm on\,}X}}\sigma_{P}(D)P.$$
Note that $N_{\sigma}(D)$ is not necessarily $\mathbb{R}$-Cartier in this paper.
\end{defn}

When $X$ is smooth, the definition of $N_{\sigma}(D)$ coincides with \cite[III, 1.12 Definition]{nakayama}. 

We refer to \cite[Remark 2.3]{has-finite} for basic properties of asymptotic vanishing order and the negative part of Nakayama--Zariski decomposition.

\begin{defn}[b-divisor]
Let $X$ be a normal variety. 
Then an {\em $\mathbb{R}$-b-divisor} $\boldsymbol{\rm D}$ on $X$ is a (possibly infinite) $\mathbb{R}$-linear combination of divisorial valuations $v_{P}$
$$\boldsymbol{\rm D}=\sum_{\substack {P:{\rm \,prime\,divisor}\\{\rm over \,}X}} r_{P}v_{P} \quad (r_{P}\in \mathbb{R})$$
such that for every projective birational morphism $Y\to X$, the set 
$$\{P\;|\; \text{$c_{Y}(P)$ is a divisor on $Y$ such that $r_{P}\neq 0$}\}$$
 is a finite set. 
When $r_{P}\in \mathbb{Q}$ for all $P$, we call $\boldsymbol{\rm D}$ a {\em $\mathbb{Q}$-b-divisor}. 
For a projective birational model $Y$ of $X$, the {\em trace} of an $\mathbb{R}$-b-divisor $\boldsymbol{\rm D}=\sum_{P} r_{P}v_{P}$ on $Y$, which we denote $\boldsymbol{\rm D}_{Y}$, is defined by
$$\boldsymbol{\rm D}_{Y}=\sum_{\substack {c_{Y}(P){\rm \,is \;a\;divisor}\\{\rm  on \,}Y}} r_{P}c_{Y}(P).$$
By definition, $\boldsymbol{\rm D}_{Y}$ is an $\mathbb{R}$-divisor on $Y$. 

Let $\boldsymbol{\rm D}$ be an $\mathbb{R}$-b-divisor on $X$. 
If there exists an $\mathbb{R}$-Cartier divisor $D$ on a projective birational model $Y$ of $X$ such that 
 $$\boldsymbol{\rm D}=\sum_{\substack {P:{\rm \,prime\,divisor}\\{\rm over \,}X}} {\rm ord}_{P}(D)\cdot v_{P},$$
 then, we say that $\boldsymbol{\rm D}$ is {\em $\mathbb{R}$-b-Cartier}, we say that {\em $\boldsymbol{\rm D}$ descends to $Y$}, and we denote $\boldsymbol{\rm D}=\overline{D}$. 
When $\boldsymbol{\rm D}=\overline{D}$ and $D$ is $\mathbb{Q}$-Cartier, we say that $\boldsymbol{\rm D}$ is {\em $\mathbb{Q}$-b-Cartier}. 
When $\boldsymbol{\rm D}=\overline{\boldsymbol{\rm D}_{Y'}}$ and $\boldsymbol{\rm D}_{Y'}$ is Cartier for some projective birational model $Y'$, we say that $\boldsymbol{\rm D}$ is {\em b-Cartier}. 

Suppose that $X$ has a projective morphism $X\to Z$ to a variety $Z$, and suppose in addition that $\boldsymbol{\rm D}=\overline{\boldsymbol{\rm D}_{Y'}}$ and $\boldsymbol{\rm D}_{Y'}$ is nef over $Z$ (resp.~big over $Z$) for some projective birational model $Y'$ of $X$. 
Then, we say that $\boldsymbol{\rm D}$ is {\em b-nef$/Z$} (resp.~{\em b-big$/Z$}).   

Let $X$ and $X'$ be normal varieties which are projective over a variety $Z$, and let $\phi \colon X \dashrightarrow X'$ be a birational map over $Z$. 
Let $\boldsymbol{\rm D}$ and $\boldsymbol{\rm D}'$ be $\mathbb{R}$-b-divisors on $X$ and $X'$, respectively. 
Then, we say that {\em $\boldsymbol{\rm D}=\boldsymbol{\rm D}'$ by $\phi$} if $\boldsymbol{\rm D}_{Y}=\boldsymbol{\rm D}'_{Y}$ for all projective  birational morphisms $f \colon Y \to X$ and $f' \colon Y \to X'$ such that $f'=\phi \circ f$.
\end{defn}

\begin{defn}[Singularities of generalized pairs]\label{defn--gen-pair}
A {\em generalized pair} $(X,\Delta,\boldsymbol{\rm M})/Z$ consists of 
\begin{itemize}
\item
a projective morphism $X\to Z$ from a normal variety to a variety, 
\item
an effective $\mathbb{R}$-divisor $\Delta$ on $X$, and
\item
a b-nef$/Z$ $\mathbb{R}$-b-Cartier $\mathbb{R}$-b-divisor $\boldsymbol{\rm M}$ on $X$ 
\end{itemize}
such that $K_{X}+\Delta+\boldsymbol{\rm M}_{X}$ is $\mathbb{R}$-Cartier. 
When $\boldsymbol{\rm M}=0$, the generalized pair $(X,\Delta,\boldsymbol{\rm M})/Z$ is a usual pair $(X,\Delta)$ with $X\to Z$. 
When $Z$ is a point, we simply denote $(X,\Delta,\boldsymbol{\rm M})$. 

Let $(X,\Delta,\boldsymbol{\rm M})/Z$ be a generalized pair and $P$ a prime divisor over $X$. 
Let $f \colon Y \to X$ be a projective birational morphism such that $\boldsymbol{\rm M}$ descends to $Y$ and $P$ appears as a divisor on $Y$. 
Then there is an $\mathbb{R}$-divisor $\Gamma$ on $Y$ such that
$$K_{Y}+\Gamma+\boldsymbol{\rm M}_{Y}=f^{*}(K_{X}+\Delta+\boldsymbol{\rm M}_{X}).$$
Then the {\em generalized log discrepancy $a(P,X,\Delta+\boldsymbol{\rm M}_{X})$} of $P$ with respect to $(X,\Delta,\boldsymbol{\rm M})/Z$ is defined to be $1-{\rm coeff}_{P}(\Gamma)$. 
When $\boldsymbol{\rm M}=0$, the generalized log discrepancy $a(P,X,\Delta)$ coincides with the {\em log discrepancy} of $P$ with respect to the usual pair $(X,\Delta)$. 

A generalized pair $(X,\Delta,\boldsymbol{\rm M})/Z$ is called a {\em generalized klt} (resp.~{\em generalized lc}) {\em pair} if $a(P,X,\Delta+\boldsymbol{\rm M}_{X})>0$ (resp.~$a(P,X,\Delta+\boldsymbol{\rm M}_{X})\geq 0$) for all prime divisors $P$ over $X$. 
A {\em generalized lc center} of $(X,\Delta,\boldsymbol{\rm M})/Z$ is the image on $X$ of a prime divisor $P$ over $X$ satisfying $a(P,X,\Delta+\boldsymbol{\rm M}_{X})=0$. 

A generalized pair $(X,\Delta,\boldsymbol{\rm M})/Z$ is a {\em generalized dlt pair} if it is generalized lc and for any generic point $\eta$ of any generalized lc center of $(X,\Delta,\boldsymbol{\rm M})/Z$, $(X,\Delta)$ is log smooth near $\eta$ and $\boldsymbol{\rm M}$ descends to $X$ on a neighborhood of $\eta$. 
\end{defn}

\begin{lem}
Let $(X,\Delta,\boldsymbol{\rm M})/Z$ be a generalized dlt pair such that $Z$ is quasi-projective, and let $S$ be a component of $\lfloor \Delta \rfloor$. 
Then there is a plt pair $(X,B)$ such that $S= \lfloor B \rfloor$. 
In particular, $S$ is normal, and $(X,B')$ is a klt pair for some $B'$. 
\end{lem}

\begin{proof}
In this proof, we will use the notion of {\em generalized sub-lc pair}, whose definition is the same as that of generalized lc pairs except that the boundary part is not necessarily effective. 
This is an analogous definition of sub-lc pairs. 

First, we reduce the lemma to the case where $\lfloor \Delta \rfloor =S$. 
We can take an open subset $U\subset X$ such that $(U,\Delta|_{U})$ is log smooth, $U$ contains the generic points of all generalized lc centers of $(X,\Delta,\boldsymbol{\rm M})/Z$, and $\boldsymbol{\rm M}|_{U}$ descends to $U$. 
We take a log resolution $f \colon Y \to X$ of $(X,\Delta)$ such that 
\begin{itemize}
\item
$f$ is an isomorphism over $U$, and 
\item
$-E$ is $f$-ample for some effective $f$-exceptional divisor $E$ on $Y$. 
\end{itemize}
We note that we do not assume $\boldsymbol{\rm M}_{Y}$ to be nef over $Z$. 
We may write 
$$K_{Y}+\Gamma+\boldsymbol{\rm M}_{Y}=f^{*}(K_{X}+\Delta+\boldsymbol{\rm M}_{X}). $$
Since $U$ contains the generic points of all generalized lc centers of $(X,\Delta,\boldsymbol{\rm M})/Z$, there is a real number $t>0$ such that $(Y, \Gamma+tE,\boldsymbol{\rm M})/Z$ is generalized sub-lc. 
Take an ample divisor $A$ on $X$ such that $-\frac{t}{2}E+f^{*}A$ is ample. 
By perturbing the coefficients of $\Gamma$ with $-\frac{t}{2}E+f^{*}A$ and pushing down, we get an $\mathbb{R}$-divisor $G \geq0$ on $X$ such that 
\begin{itemize}
\item
$K_{X}+\Delta+\boldsymbol{\rm M}_{X}+A\sim_{\mathbb{R},Z}K_{X}+G+\boldsymbol{\rm M}_{X}$, 
\item
$(U,G|_{U})$ is log smooth, 
\item
$\lfloor G \rfloor =S$, and 
\item
writing
$K_{Y}+G_{Y}+\boldsymbol{\rm M}_{Y}=f^{*}(K_{X}+G+\boldsymbol{\rm M}_{X})$, then $(Y, G_{Y}, \boldsymbol{\rm M})/Z$ is a generalized sub-lc pair whose generalized lc centers intersect $f^{-1}(U)$. 
\end{itemize}
In this way, we get a generalized dlt pair $(X, G, \boldsymbol{\rm M})/Z$  such that $\lfloor G \rfloor =S$. 

In the case where $\lfloor \Delta \rfloor =S$, we see that $(U,\Delta|_{U})$ is plt. 
Hence $a(P,X,\Delta+\boldsymbol{\rm M}_{X})>0$ for all prime divisors $P$ over $X$ except $P=S$. 
Let $g \colon Y' \to X$ be a log resolution of $(X,\Delta)$ such that 
\begin{itemize}
\item
$\boldsymbol{\rm M}$ descends to $Y'$, and 
\item
$-F'$ is $g$-ample for some $g$-exceptional divisor $F'$. 
\end{itemize}
Let $H$ be an ample divisor on $X$ such that $-F'+g^{*}H$ is ample. 
Since $\boldsymbol{\rm M}_{Y'}$ is nef over $Z$ and $a(F_{i},X,\Delta+\boldsymbol{\rm M}_{X})>0$ for every component $F_{i}$ of $F'$, we can apply the  argument of perturbation of coefficients. 
We get a plt pair $(X,B)$ such that $\lfloor B \rfloor =S$. 
\end{proof}

We introduce properties of divisorial adjunction for generalized pairs. 

\begin{rem}[Divisorial adjunction]\label{rem--coeff-adj}
Let $(X,\Delta,\boldsymbol{\rm M})/Z$ be a generalized lc pair, let $S$ be a component of $\lfloor \Delta \rfloor$ with the normalization $S^{\nu}$, and let $(S^{\nu},\Delta_{S^{\nu}},\boldsymbol{\rm N})/Z$ be a generalized lc pair defined with  divisorial adjunction for generalized pairs. 
\begin{itemize} 
\item 
If $p\boldsymbol{\rm M}$ is b-Cartier, then $p \boldsymbol{\rm N}$ is b-Cartier. 
\item 
We fix a DCC set $\Lambda \subset \mathbb{R}_{\geq 0}$.  By \cite[Proposition 4.9]{bz}, there is a DCC set $\Omega \subset \mathbb{R}_{\geq 0}$, depending only on $\Lambda$, such that if the coefficients of $\Delta$ belong to $\Lambda$ and $\boldsymbol{\rm M}$ is the sum of finitely many b-nef$/Z$ b-Cartier b-divisors with the coefficients in $\Lambda$, then the coefficients of $\Delta_{S^{\nu}}$ belong to $\Omega$. 
\item 
We fix a finite set of rational numbers $\mathfrak{R}\subset [0,1]$ and a positive integer $p$. By \cite[Lemma 3.3]{birkar-compl}, 
there is a finite set of rational numbers $\mathfrak{S}\subset [0,1]$, depending only on $\mathfrak{R}$ and $p$, such that if the coefficients of $\Delta$ belong to $\Phi(\mathfrak{R})$ and $p \boldsymbol{\rm M}$ is b-Cartier then the coefficients of $\Delta_{S^{\nu}}$ belong to $\Phi(\mathfrak{S})$. \end{itemize}
\end{rem}

\begin{defn}[Restriction morphism]\label{defn--rest-mor} 
Let $(X,\Delta,\boldsymbol{\rm M})/Z$ be a generalized lc pair, and let $S$ be a component of $\lfloor \Delta \rfloor$ with the normalization $S^{\nu}$. 
With divisorial adjunction for generalized pairs we define a generalized lc pair $(S^{\nu},\Delta_{S^{\nu}},\boldsymbol{\rm N})/Z$. 
Suppose that $p \boldsymbol{\rm M}$ is b-Cartier. 
Let $\tau\colon S^{\nu}\to S \hookrightarrow X$ be the natural morphism. 
With the idea of \cite[3.1]{birkar-compl}, we will define a morphism 
$$\mathcal{O}_{X}(\lfloor p(K_{X}+\Delta+\boldsymbol{\rm M}_{X})\rfloor) 
\longrightarrow 
\tau_{*}\mathcal{O}_{S^{\nu}}(\lfloor p(K_{S^{\nu}}+\Delta_{S^{\nu}}+\boldsymbol{\rm N}_{S^{\nu}})\rfloor).$$

We take a log resolution $f\colon Y \to X$ of $(X,\Delta)$ such that $\boldsymbol{\rm M}$ descends to $Y$, and we put $T = f^{-1}_{*}S$. 
Note that $p \boldsymbol{\rm M}_{Y}$ is Cartier since $p \boldsymbol{\rm M}$ is b-Cartier. 
Hence, we may replace $p \boldsymbol{\rm M}_{Y}$ by a linear equivalence class so that ${\rm Supp}\boldsymbol{\rm M}_{Y}$ does not contain $T$. 
We can write 
$$p(K_{Y}+T+\Gamma+\boldsymbol{\rm M}_{Y})=pf^{*}(K_{X}+\Delta+\boldsymbol{\rm M}_{X}),$$ 
and we have 
$$\lfloor p(K_{Y}+T+\Gamma+\boldsymbol{\rm M}_{Y})\rfloor|_{T}\sim pK_{T}+\lfloor p\Gamma \rfloor|_{T}+p \boldsymbol{\rm M}_{Y}|_{T}.$$ 
Thus, we obtain a natural morphism
\begin{equation*}\tag{$*$}\label{defn--rest-mor-*}
f_{*}\mathcal{O}_{Y}\bigl(\lfloor p(K_{Y}+T+\Gamma+\boldsymbol{\rm M}_{Y})\rfloor\bigr) \longrightarrow \tau_{*}f_{T*}\mathcal{O}_{T}(pK_{T}+\lfloor p\Gamma \rfloor|_{T}+p \boldsymbol{\rm M}_{Y}|_{T}),
\end{equation*}
where $f_{T}\colon T \to S^{\nu}$ and $\tau \colon S^{\nu} \to S \to X$ are the natural morphisms. 
With the relation 
$$K_{Y}+T+\Gamma+\boldsymbol{\rm M}_{Y}=f^{*}(K_{X}+\Delta+\boldsymbol{\rm M}_{X}),$$
we have
\begin{equation*}\tag{$**$}\label{defn--rest-mor-**}
f_{*}\mathcal{O}_{Y}\bigl(\lfloor p(K_{Y}+T+\Gamma+\boldsymbol{\rm M}_{Y})\rfloor\bigr) \simeq \mathcal{O}_{X}\bigl(\lfloor p(K_{X}+\Delta+\boldsymbol{\rm M}_{X})\rfloor\bigr)
\end{equation*}
By construction of divisorial adjunction for generalized pairs, we have
$$f_{T}^{*}(K_{S^{\nu}}+\Delta_{S^{\nu}}+\boldsymbol{\rm N}_{S^{\nu}})=K_{T}+\Gamma|_{T} +\boldsymbol{\rm M}_{Y}|_{T}$$
as $\mathbb{R}$-divisors. 
Since $p \boldsymbol{\rm M}_{Y}|_{T}$ is a Weil divisor, we see that
$$\lfloor f_{T}^{*}\bigl(p(K_{S^{\nu}}+\Delta_{S^{\nu}}+\boldsymbol{\rm N}_{S^{\nu}})\bigr) \rfloor =pK_{T}+\lfloor p\Gamma\rfloor|_{T}+p \boldsymbol{\rm M}_{Y}|_{T}.$$
From this fact, we see that 
\begin{equation*}\tag{$*\!*\!*$}\label{defn--rest-mor-***}
f_{T*}\mathcal{O}_{T}(pK_{T}+\lfloor p\Gamma\rfloor|_{T}+p \boldsymbol{\rm M}_{Y}|_{T}) \simeq \mathcal{O}_{S^{\nu}}\bigl(\lfloor p(K_{S^{\nu}}+\Delta_{S^{\nu}}+\boldsymbol{\rm N}_{S^{\nu}}) \rfloor\bigr). 
\end{equation*}
By (\ref{defn--rest-mor-*}), (\ref{defn--rest-mor-**}), and (\ref{defn--rest-mor-***}), 
we can define 
the desired morphism.  
\end{defn}

\begin{defn}[Models, cf.~{\cite{hanli}}]\label{defn--models}
Let $(X,\Delta,\boldsymbol{\rm M})/Z$ be a generalized lc pair, and let $(X',\Delta',\boldsymbol{\rm M}')/Z$ be a generalized pair. 
Let $\phi\colon X\dashrightarrow X'$ be a birational map over $Z$.  

We say that $(X',\Delta',\boldsymbol{\rm M}')/Z$ is a {\em generalized log birational model} of $(X,\Delta,\boldsymbol{\rm M})/Z$ if $\boldsymbol{\rm M}=\boldsymbol{\rm M}'$ by $\phi$ and $\Delta'=\phi_{*}\Delta+\sum_{i}E_{i}$, where $E_{i}$ runs over $\phi^{-1}$-exceptional prime divisors. 

We say that a generalized log birational model $(X',\Delta',\boldsymbol{\rm M}')/Z$ of $(X,\Delta,\boldsymbol{\rm M})/Z$ is a {\em weak generalized log canonical model} ({\em weak generalized lc model}, for short)  {\em over $Z$} if 
\begin{itemize}
\item
$K_{X'}+\Delta'+\boldsymbol{\rm M}'_{X'}$ is nef over $Z$, and 
\item
$a(D, X, \Delta+\boldsymbol{\rm M}_{X}) \leq a(D, X', \Delta'+\boldsymbol{\rm M}'_{X'})$ for every prime divisor $D$ on $X$ which is exceptional over $X'$.  
\end{itemize}
We say that a weak generalized lc model $(X',\Delta',\boldsymbol{\rm M}')/Z$ of $(X,\Delta,\boldsymbol{\rm M})/Z$ is a {\it minimal model over $Z$} if the inequality on generalized log discrepancies is strict. 

We say that a minimal model $(X',\Delta',\boldsymbol{\rm M}')/Z$ of $(X,\Delta,\boldsymbol{\rm M})/Z$ is a {\it good minimal model over $Z$} if $K_{X'}+\Delta'+\boldsymbol{\rm M}'_{X'}$ is semi-ample over $Z$. 
\end{defn}

\subsection{Abundant divisors}\label{subsec2.2}

We define the invariant Iitaka dimension and the numerical dimension (\cite{nakayama}) for $\mathbb{R}$-Cartier divisors on normal projective varieties, then we define abundant divisors, log abundant divisors, and log big divisors. 

\begin{defn}[Invariant Iitaka dimension]\label{defn--inv-iitaka-dim}Let $X$ be a normal projective variety, and let $D$ be an $\mathbb{R}$-Cartier  divisor on $X$. We define the {\em invariant  Iitaka dimension} of $D$, denoted by $\kappa_{\iota}(X,D)$, as follows (\cite[Definition 2.2.1]{choi}, see also \cite[Definition 2.5.5]{fujino-book}):  If there is an $\mathbb{R}$-divisor $E\geq 0$ such that $D\sim_{\mathbb{R}}E$, set $\kappa_{\iota}(X,D)=\kappa(X,E)$. Here, the right hand side is the usual Iitaka dimension of $E$. Otherwise, we set $\kappa_{\iota}(X,D)=-\infty$. 

Let $X\to Z$ be a projective morphism from a normal variety to a variety and $D$ an $\mathbb{R}$-Cartier divisor on $X$. Then the {\em relative invariant Iitaka dimension} of $D$, denoted by $\kappa_{\iota}(X/Z,D)$, is similarly defined: If there is an $\mathbb{R}$-divisor $E\geq 0$ such that $D\sim_{\mathbb{R},Z}E$ then we set $\kappa_{\iota}(X/Z,D)=\kappa_{\iota}(F,D|_{F})$, where $F$ is a sufficiently general fiber of the Stein factorization of $X\to Z$, and otherwise we set $\kappa_{\iota}(X/Z,D)=-\infty$. 
\end{defn}

\begin{defn}[Numerical dimension]\label{defn--num-dim}
Let $X$ be a normal projective variety, and let $D$ be an $\mathbb{R}$-Cartier divisor on $X$. 
We define the {\em numerical dimension} of $D$, denoted by $\kappa_{\sigma}(X,D)$, as follows (\cite[V, 2.5 Definition]{nakayama}): 
For any Cartier divisor $A$ on $X$, we set
\begin{equation*}
\sigma(D;A)={\rm max}\!\Set{\!k\in \mathbb{Z}_{\geq0} | \underset{m\to \infty}{\rm lim\,sup}\frac{{\rm dim}H^{0}(X,\mathcal{O}_{X}(\lfloor mD \rfloor+A))}{m^{k}}>0\!}
\end{equation*}
if ${\rm dim}H^{0}(X,\mathcal{O}_{X}(\lfloor mD \rfloor+A))>0$ for infinitely many $m\in \mathbb{Z}_{>0}$, and otherwise we set $\sigma(D;A):=-\infty$. 
Then, we define 
\begin{equation*}
\kappa_{\sigma}(X,D):={\rm max}\!\set{\sigma(D;A) | \text{$A$ is a Cartier divisor on $X$}\!}.
\end{equation*}

Let $X\to Z$ be a projective morphism from a normal variety to a variety, and let $D$ be an $\mathbb{R}$-Cartier divisor on $X$. 
Then, the {\em relative numerical dimension} of $D$ over $Z$ is defined by $\kappa_{\sigma}(F,D|_{F})$, where $F$ is a sufficiently general fiber of the Stein factorization of $X\to Z$.  
Then $\kappa_{\sigma}(F,D|_{F})$ does not depend on the choice of $F$, so the relative numerical dimension is well-defined. 
In this paper, we denote $\kappa_{\sigma}(F,D|_{F})$ by $\kappa_{\sigma}(X/Z,D)$. 
\end{defn}

We refer to \cite[Remark 2.14]{has-finite} (see also \cite[V, 2.7 Proposition]{nakayama}, \cite[Remark 2.8]{hashizumehu}) for basic properties of the invariant Iitaka dimension and the numerical dimension. 

\begin{defn}[Abundant divisor, log abundant divisor, log big divisor]\label{defn--abund}
Let $X \to Z$ be a projective morphism from a normal variety to a variety, and let $D$ be an $\mathbb{R}$-Cartier divisor on $X$. 
We say that $D$ is {\em abundant over $Z$} if $\kappa_{\iota}(X/Z,D)=\kappa_{\sigma}(X/Z,D)$. 

Let $X\to Z$ and $D$ be as above, and let $(X,\Delta,\boldsymbol{\rm M})/Z$ be a generalized lc pair. 
We say that $D$ is {\em log abundant over $Z$ with respect to $(X,\Delta,\boldsymbol{\rm M})/Z$} if $D$ is abundant over $Z$ and for any generalized lc center $S$ of $(X,\Delta,\boldsymbol{\rm M})/Z$ with the normalization $S^{\nu}\to S$, the pullback $D|_{S^{\nu}}$ is abundant over $Z$. 

We say that $D$ is {\em log big over $Z$ with respect to $(X,\Delta,\boldsymbol{\rm M})/Z$} if $D$ is big over $Z$ and for any generalized lc center $S$ of $(X,\Delta,\boldsymbol{\rm M})/Z$ with the normalization $S^{\nu}\to S$, the pullback $D|_{S^{\nu}}$ is big over $Z$. 

When $Z$ is a point in the above definitions, we remove ``over $Z$'' in each terminology.   
\end{defn}

\begin{lem}[cf.~{\cite[Lemma 2.11]{hashizumehu}, \cite{nakayama}, \cite{fujino-subadd-correct}}]
\label{lem--iitakafib} 
Let $(X,\Delta,\boldsymbol{\rm M})$ be a generalized lc pair such that $K_{X}+\Delta+\boldsymbol{\rm M}_{X}$ is abundant and there is  an effective $\mathbb{R}$-divisor $D$ on $X$ such that $D\sim_{\mathbb{R}}K_{X}+\Delta+\boldsymbol{\rm M}_{X}$. 
Let $X\dashrightarrow V$ be the Iitaka fibration associated to $D$. 
We take a log resolution $f\colon Y\to X$ of $(X,\Delta)$ such that $\boldsymbol{\rm M}$ descends to $Y$ and the induced map $Y\dashrightarrow V$ is a morphism. 
Let $(Y,\Gamma,\boldsymbol{\rm M})$ be a  generalized lc pair such that $$K_{Y}+\Gamma+\boldsymbol{\rm M}_{Y}=f^{*}(K_{X}+\Delta+\boldsymbol{\rm M}_{X})+E$$ for some effective $f$-exceptional $\mathbb{R}$-divisor $E$. 
Then $\kappa_{\sigma}(Y/V,K_{Y}+\Gamma+\boldsymbol{\rm M}_{Y})=0$. 
\end{lem}

\begin{proof}
The argument in \cite[Proof of Lemma 2.11]{hashizumehu} works with no changes because we may apply the discussion in \cite[Section 3]{fujino-subadd-correct}. 
\end{proof}

\begin{lem}\label{lem--lcpair-neflogbig}
Let $(X,\Delta)$ be a projective lc pair and $M=\sum_{i}\mu_{i}M_{i}$ a finite $\mathbb{R}_{>0}$-linear combination of nef Cartier divisors $M_{i}$ on $X$ such that $M$ is log big with respect to $(X,\Delta)$ and $\mu_{i} > 2\cdot{\rm dim}X$ for all $i$. 
Then $K_{X}+\Delta+M$ is log big with respect to $(X,\Delta)$. 
\end{lem}

\begin{proof}
We need to prove that if $S$ is $X$ or an lc center of $(X,\Delta)$ with the normalization $S^{\nu}$ then $(K_{X}+\Delta+M)|_{S^{\nu}}$ is big.
Suppose by contradiction that there is an lc center $S$ of $(X,\Delta)$ or $S=X$ such that $(K_{X}+\Delta+M)|_{S^{\nu}}$ is not big. 
By replacing $M$ with $(1-t)M$ for some $0<t \ll 1$, we may assume that $(K_{X}+\Delta+M)|_{S^{\nu}}$ is not pseudo-effective. 

Taking a dlt blow-up of $(X,\Delta)$ and applying Ambro's canonical bundle formula as  in \cite[Corollary 3.2]{fg-bundle}, we get a generalized lc pair $(S^{\nu},\Delta_{S^{\nu}},\boldsymbol{\rm N})$ such that 
$$K_{S^{\nu}}+\Delta_{S^{\nu}}+\boldsymbol{\rm N}_{S^{\nu}}\sim_{\mathbb{R}}(K_{X}+\Delta)|_{S^{\nu}}.$$ 
Let $(Y,\Gamma, \boldsymbol{\rm N})$ be a $\mathbb{Q}$-factorial generalized dlt model of $(S^{\nu},\Delta_{S^{\nu}},\boldsymbol{\rm N})$, and let $M_{Y}$ be the pullback of $M|_{S^{\nu}}$ to $Y$. 
Then $M_{Y}$ is big and 
$K_{Y}+\Gamma+\boldsymbol{\rm N}_{Y}+M_{Y}$
is not pseudo-effective. 

By running a $(K_{Y}+\Gamma+\boldsymbol{\rm N}_{Y}+M_{Y})$-MMP with scaling of an ample divisor, we get a birational contraction $\phi \colon Y \dashrightarrow Y'$ to a normal projective variety $Y'$ which has the structure of a Mori fiber space $Y' \to V$ with respect to $\phi_{*}(K_{Y}+\Gamma+\boldsymbol{\rm N}_{Y}+M_{Y})$. 
By the length of extremal rays (cf.~\cite[Proposition 3.17]{hanli}), the birational transform of $M_{Y}$ is numerically trivial with respect to the extremal contraction in each step of the MMP.  
Thus, $\phi_{*}M_{Y}$ is big and $\phi_{*}M_{Y} \sim_{\mathbb{R}, V}0$, a contradiction because ${\rm dim}Y'>{\rm dim}V$. 

In this way, we see that $(K_{X}+\Delta+M)|_{S^{\nu}}$ is big, and therefore 
$K_{X}+\Delta+M$ is log big with respect to $(X,\Delta)$. 
\end{proof}

The following lemma plays a crucial role in Section \ref{sec4} and Section \ref{sec5}.

\begin{lem}\label{lem--extension-lc-center}
Let $(X,\Delta,\boldsymbol{\rm M})/Z$ be a generalized lc pair with a morphism $\pi\colon X\to Z$ such that 
\begin{itemize}
\item
$\Delta$ is a $\mathbb{Q}$-divisor and $\boldsymbol{\rm M}$ is a $\mathbb{Q}$-b-Cartier $\mathbb{Q}$-b-divisor, 
\item
$K_{X}+\Delta+\boldsymbol{\rm M}_{X}$ is nef over $Z$, and
\item 
there is a log resolution $f\colon Y \to X$ of $(X,\Delta)$ such that 
\begin{itemize}
\item
$\boldsymbol{\rm M}$ descends to $Y$, and 
\item
writing $K_{Y}+\Gamma+\boldsymbol{\rm M}_{Y}=f^{*}(K_{X}+\Delta+\boldsymbol{\rm M}_{X})+E$,
where $\Gamma \geq 0$ and $E \geq 0$ have no common components, then $\boldsymbol{\rm M}_{Y}$ is log big over $Z$ with respect to $(Y,\Gamma)$. 
\end{itemize}
\end{itemize}
Let $S$ be a component of $\lfloor \Delta \rfloor$, let $S^{\nu}$ be the normalization of $S$, and let $(S^{\nu},\Delta_{S^{\nu}},\boldsymbol{\rm N})/Z$ be a generalized lc pair with a morphism $\pi_{S^{\nu}}\colon S^{\nu}\to Z$ defined with divisorial adjunction for $(X,\Delta,\boldsymbol{\rm M})/Z$ and $S$. 
Let $p$ be a positive integer such that $p\Delta$ is a Weil divisor and $p\boldsymbol{\rm M}_{Y}$ is Cartier. 
Then, the morphism
\begin{equation*}
\begin{split}
\pi_{*}\mathcal{O}_{X}(lp(K_{X}+\Delta+\boldsymbol{\rm M}_{X})) 
\longrightarrow 
\pi_{S^{\nu}*}\mathcal{O}_{S^{\nu}}(\lfloor lp(K_{S^{\nu}}+\Delta_{S^{\nu}}+\boldsymbol{\rm N}_{S^{\nu}})\rfloor)
\end{split}
\end{equation*}
induced by Definition \ref{defn--rest-mor} is surjective for every positive integer $l$. 
\end{lem}

\begin{proof}
The idea is very similar to \cite{fujino-abund-logbig}. 
We fix $l$. 

We put 
$$L=\boldsymbol{\rm M}_{Y}+(lp-1) f^{*}(K_{X}+\Delta+\boldsymbol{\rm M}_{X}).$$ 
Since $K_{X}+\Delta+\boldsymbol{\rm M}_{X}$ is nef over $Z$ and $\boldsymbol{\rm M}_{Y}$ is nef and log big over $Z$ with respect to $(Y,\Gamma)$, it follows that $L$ is nef and log big over $Z$ with respect to $(Y,\Gamma)$. 
We have
\begin{equation*}\tag{$1$} \label{surj-proof-(1)}
K_{Y}+\Gamma+L=lp f^{*}(K_{X}+\Delta+\boldsymbol{\rm M}_{X})+E.
\end{equation*}
Since the coefficients of $p\Delta$ are integers and $p\boldsymbol{\rm M}_{Y}$ is Cartier, we can write
\begin{equation*}\tag{$2$}\label{surj-proof-(2)}
lp f^{*}(K_{X}+\Delta+\boldsymbol{\rm M}_{X})=\lfloor lp f^{*}(K_{X}+\Delta+\boldsymbol{\rm M}_{X}) \rfloor + E'
\end{equation*}
with an effective $f$-exceptional $\mathbb{Q}$-divisor $E'$. 
Put $T=f^{-1}_{*}S$. 
By restricting both hand sides of (\ref{surj-proof-(2)}) to $T$ and taking the round down, we have
\begin{equation*}\tag{$3$}\label{surj-proof-(4)}
\lfloor \bigl(lp f^{*}(K_{X}+\Delta+\boldsymbol{\rm M}_{X})\bigr)\big{|}_{T}\rfloor=\lfloor lp f^{*}(K_{X}+\Delta+\boldsymbol{\rm M}_{X}) \rfloor|_{T}. 
\end{equation*}
We define $\Gamma':=\Gamma-T.$
By the definition, $(Y,T+\Gamma')$ is a log smooth lc pair and $L$ is nef and log big over $Z$ with respect to $(Y,T+\Gamma')$. 
The relations (\ref{surj-proof-(1)}) and (\ref{surj-proof-(2)}) show 
\begin{equation*}\tag{$4$}\label{surj-proof-(5)}
K_{Y}+T+\Gamma'+L=\lfloor lp f^{*}(K_{X}+\Delta+\boldsymbol{\rm M}_{X}) \rfloor  + E'+E.
\end{equation*}
We define a Weil divisor $F$ on $Y$ so that 
\begin{equation*}
{\rm coeff}_{P}(F)=\left\{  \begin{array}{l}{0 \qquad ({\rm coeff}_{P}(\Gamma'-\{(E+E')\})\geq 0)} \\ {1 \qquad  ({\rm coeff}_{P}(\Gamma'-\{(E+E')\})< 0)} \end{array}\right. \end{equation*} 
for every prime divisor $P$. 
We put 
$$B=\Gamma'-\{(E+E')\}+F \qquad {\rm and} \qquad G=\lfloor( E + E')\rfloor + F.$$
By the definition, we can check that $B$ is a boundary divisor, $\lfloor B \rfloor \leq \lfloor \Gamma' \rfloor$, and $G$ is an effective $f$-exceptional Weil divisor. 
Since $\lfloor B \rfloor \leq \lfloor \Gamma' \rfloor$, the divisor $L$ is nef and log big over $Z$ with respect to $(Y,T+B)$. 
Furthermore, the relation (\ref{surj-proof-(5)}) implies
\begin{equation*}\tag{$5$}\label{surj-proof-(6)}
K_{Y}+T+B+L=\lfloor lp f^{*}(K_{X}+\Delta+\boldsymbol{\rm M}_{X}) \rfloor  +G.
\end{equation*}
We put $D=lp f^{*}(K_{X}+\Delta+\boldsymbol{\rm M}_{X})$. 
Since $G$ is $f$-exceptional, we have
$$\mathcal{O}_{X}( lp(K_{X}+\Delta+\boldsymbol{\rm M}_{X}) )\simeq f_{*}\mathcal{O}_{Y}(\lfloor D \rfloor  +G).$$
By construction, $(Y,B)$ is a log smooth lc pair and $L$ is nef and log big over $Z$ with respect to $(Y,B)$. 
Thus, the Kodaira type vanishing theorem \cite[Lemma 1.5]{fujino-abund-logbig} implies
 $$R^{1}(\pi \circ f)_{*}\mathcal{O}_{Y}(\lfloor D \rfloor  +G-T)=R^{1}(\pi \circ f)_{*}\mathcal{O}_{Y}(K_{Y}+B+L)=0.$$ 
From these facts, the morphism
$$(\pi \circ f)_{*}\mathcal{O}_{Y}(\lfloor D \rfloor  +G) \longrightarrow (\pi_{S^{\nu}} \circ f_{T})_{*}\mathcal{O}_{T}\bigl((\lfloor D \rfloor  +G)|_{T}\bigr)$$
is surjective, where $f_{T}\colon T \to S^{\nu}$ is the natural morphism. 
Since $\lfloor D \rfloor|_{T}=\lfloor D|_{T} \rfloor$ by (\ref{surj-proof-(4)}),
a natural morphism 
$$\mathcal{O}_{S^{\nu}}(\lfloor lp(K_{S^{\nu}}+\Delta_{S^{\nu}}+\boldsymbol{\rm N}_{S^{\nu}})\rfloor) \longrightarrow f_{T*}\mathcal{O}_{T}(\lfloor D|_{T} \rfloor) \hookrightarrow f_{T*}\mathcal{O}_{T}(\lfloor D \rfloor|_{T}  + G|_{T})$$
is defined. 
From these facts, we obtain the diagram
 $$
\xymatrix{
\pi_{*}f_{*}\mathcal{O}_{Y}(\lfloor D \rfloor  + G) \ar@{->>}[r]&\pi_{S^{\nu}*}f_{T*}\mathcal{O}_{T}(\lfloor D \rfloor|_{T}  + G|_{T})\\
\pi_{*}\mathcal{O}_{X}( lp(K_{X}+\Delta+\boldsymbol{\rm M}_{X}))\ar[u]^{\simeq}\ar[r]&\pi_{S^{\nu}*}\mathcal{O}_{S^{\nu}}\bigl(\lfloor lp(K_{S^{\nu}}+\Delta_{S^{\nu}}+\boldsymbol{\rm N}_{S^{\nu}})\rfloor\bigr),\ar@{^{(}->}[u]
}
$$
and therefore the morphism
\begin{equation*}
\begin{split}
\pi_{*}\mathcal{O}_{X}(lp(K_{X}+\Delta+\boldsymbol{\rm M}_{X})) 
\longrightarrow 
\pi_{S^{\nu}*}\mathcal{O}_{S^{\nu}}\bigl(\lfloor lp(K_{S^{\nu}}+\Delta_{S^{\nu}}+\boldsymbol{\rm N}_{S^{\nu}})\rfloor\bigr)
\end{split}
\end{equation*}
is surjective. 
\end{proof}

\section{Minimal model theory for generalized pairs}\label{sec3}

The goal of this section is to prove Theorem \ref{thm--mmp-neflogbig-main-intro}. 
All the arguments in this section are very similar to those in \cite{hashizumehu}, \cite{has-mmp}, \cite{has-finite}, and \cite{has-nonvan-gpair}. 

The following results will be used without any mention. 

\begin{lem}\label{lem--equiv-log-birat}
Let $(X,\Delta,\boldsymbol{\rm M})/Z$ be a generalized lc pair and $(X',\Delta',\boldsymbol{\rm M})/Z$ a generalized lc pair with a birational morphism $X' \to X$ such that $(X',\Delta',\boldsymbol{\rm M})/Z$ is a generalized log birational model of $(X,\Delta,\boldsymbol{\rm M})/Z$ as in Definition \ref{defn--models}. 
Then $(X,\Delta,\boldsymbol{\rm M})/Z$ has a minimal model (resp.~a good minimal model) if and only if $(X',\Delta',\boldsymbol{\rm M})/Z$ has a minimal model (resp.~a good minimal model).  
\end{lem}

\begin{lem}
Let $(X,\Delta,\boldsymbol{\rm M})/Z$ be a generalized lc pair and $(X,\Delta,\boldsymbol{\rm M})\dashrightarrow (X',\Delta',\boldsymbol{\rm M})$ finite steps of a $(K_{X}+\Delta+\boldsymbol{\rm M}_{X})$-MMP over $Z$. 
Then $(X',\Delta',\boldsymbol{\rm M})/Z$ is a generalized lc pair, and $(X,\Delta,\boldsymbol{\rm M})/Z$ has a minimal model (resp.~a good minimal model) if and only if $(X',\Delta',\boldsymbol{\rm M})/Z$ has a minimal model (resp.~a good minimal model).  
\end{lem}

\begin{lem}[{cf.~\cite[Theorem 4.1]{hanli}}]
Let $(X,\Delta,\boldsymbol{\rm M})/Z$ be a generalized lc pair such that $Z$ is quasi-projective, $(X,0)$ is a $\mathbb{Q}$-factorial klt pair, and $\boldsymbol{\rm M}$ is a finite $\mathbb{R}_{>0}$-linear combination of b-nef$/Z$ $\mathbb{Q}$-b-Cartier $\mathbb{Q}$-b-divisors. 
If $(X,\Delta,\boldsymbol{\rm M})/Z$ has a minimal model, then any sequence of steps of a $(K_{X}+\Delta+\boldsymbol{\rm M}_{X})$-MMP over $Z$ with scaling of an ample divisor terminates with a minimal model. 
\end{lem}

\subsection{Auxiliary results}
We collect results used in this section. 
Note that all results in this subsection are known in the case of usual pairs.

\begin{lem}[cf.~{\cite[Lemma 2.16]{hashizumehu}}]\label{lem--extraction-gpair}
Let $(X,\Delta,\boldsymbol{\rm M})/Z$ be a generalized lc pair such that $Z$ is quasi-projective and $(X,B)$ is a dlt pair for some $B$.
Let $\mathcal{T}$ be an empty set or a finite set of exceptional prime divisors over $X$ such that $0< a(P,X,\Delta+\boldsymbol{\rm M}_{X})< 1$ for any $P\in \mathcal{T}$. 
Then there is a $\mathbb{Q}$-factorial variety $\widetilde{X}$ and a projective birational morphism $f\colon \widetilde{X} \to X$ such that $f$-exceptional prime divisors are exactly elements of $\mathcal{T}$. 
\end{lem}

\begin{proof}
Replacing $B$, we may assume that $(X,B)$ is klt. 
We pick a real number $0<t\ll 1$ so that the generalized klt pair $\bigl(X,(1-t)\Delta+tB,(1-t)\boldsymbol{\rm M}\bigr)/Z$ satisfies
$$0< a(P,X,((1-t)\Delta+tB)+(1-t)\boldsymbol{\rm M}_{X})< 1$$ for any $P\in \mathcal{T}$. 
Replacing $(X,\Delta,\boldsymbol{\rm M})/Z$ with $\bigl(X,(1-t)\Delta+tB,(1-t)\boldsymbol{\rm M}\bigr)/Z$, we may assume that $(X,\Delta,\boldsymbol{\rm M})/Z$ is generalized klt. 
By the perturbation of coefficients, we can find $\Delta'$ such that $(X,\Delta')$ is klt and $0<a(P,X,\Delta')<1$  for any $P\in \mathcal{T}$, where $a(\,\cdot\,,X,\Delta')$ is the log discrepancy. 
Then the lemma follows from \cite[Lemma 2.16]{hashizumehu}. 
\end{proof}

\begin{lem}[cf.~{\cite[Lemma 2.6]{has-finite}}] \label{lem--discre-relation} Let $(X,\Delta,\boldsymbol{\rm M})$ and $(X',\Delta',\boldsymbol{\rm M}')$ be 
generalized dlt pairs with a birational map $\phi \colon X\dashrightarrow X'$ by which $\boldsymbol{\rm M}=\boldsymbol{\rm M}'$. Let $S$ and $S'$ be generalized lc centers of $(X,\Delta,\boldsymbol{\rm M})$ and $(X',\Delta',\boldsymbol{\rm M}')$ respectively such that $\phi$ is an isomorphism near the generic point of $S$ and the restriction $\phi|_{S}$ defines a birational map $\phi|_{S}\colon S\dashrightarrow S'$. Suppose that $K_{X}+\Delta+\boldsymbol{\rm M}_{X}$ is pseudo-effective. Suppose in addition that  
\begin{itemize} 
\item 
$a(D',X',\Delta'+\boldsymbol{\rm M}'_{X'})\leq a(D',X,\Delta+\boldsymbol{\rm M}_{X})$ for every prime divisor $D'$ on $X'$, and 
\item 
$\sigma_{P}(K_{X}+\Delta+\boldsymbol{\rm M}_{X})=0$  
 for every prime divisor $P$ over $X$ such that $c_{X}(P)\cap S \neq \emptyset$ and $a(P,X,\Delta+\boldsymbol{\rm M}_{X})< 1$, where $\sigma_{P}(\;\cdot \;)$ is as in Definition \ref{defn--asy-van-ord}. 
\end{itemize} Let $(S,\Delta_{S},\boldsymbol{\rm N})$ and $(S',\Delta_{S'}, \boldsymbol{\rm N})$ be generalized dlt pairs 
which are constructed by applying divisorial adjunctions for generalized pairs to $(X,\Delta,\boldsymbol{\rm M})$ and $(X',\Delta',\boldsymbol{\rm M}')$ repeatedly. 
Then 
$$a(Q,S',\Delta_{S'}+\boldsymbol{\rm N}_{S'})\leq a(Q,S,\Delta_{S}+\boldsymbol{\rm N}_{S})$$
 for all prime divisors $Q$ on $S'$. \end{lem}

\begin{proof}
We closely follow \cite[Proof of Lemma 2.6]{has-finite}. 

Taking an appropriate common log resolution $f \colon Y \to X$ and $f'\colon Y \to X'$ of the birational map $(X,\Delta)\dashrightarrow (X',\Delta')$,  we may find a subvariety $T$ of $Y$ which is birational to $S$ and $S'$ such that the induced morphisms $T \to S$ and $T \to S'$ form a common resolution of $\phi|_{S}$. 
Note that $\boldsymbol{\rm M}$ does not necessarily descend to $Y$. 
We may write
$$f^{*}(K_{X}+\Delta+\boldsymbol{\rm M}_{X})=f'^{*}(K_{X'}+\Delta'+\boldsymbol{\rm M}'_{X'})+G_{+}-G_{-}$$
such that $G_{+}\geq0$ and $G_{-}\geq0$ have no common components. 
Then $G_{+}$ is $f'$-exceptional by the first condition of Lemma \ref{lem--discre-relation}. 
Since $K_{X}+\Delta+\boldsymbol{\rm M}_{X}$ is pseudo-effective, we see that $K_{X'}+\Delta'+\boldsymbol{\rm M}'_{X'}$ is pseudo-effective. 
We have ${\rm Supp}G_{+} \not\supset T$ and ${\rm Supp}G_{-}\not\supset T$ since $S$ and $S'$ are generalized lc centers of $(X,\Delta,\boldsymbol{\rm M})$ and $(X',\Delta',\boldsymbol{\rm M}')$ respectively. 

We can write 
$$G_{+}=G_{0}+G_{1},$$
 where all components of $G_{0}$ intersect $T$ and $G_{1}|_{T} = 0$. 
Pick any component $E$ of $G_{0}$. 
Then ${\rm ord}_{E}(G_{-})=0$. 
We have
\begin{equation*}
\begin{split}
&\sigma_{E}(K_{X}+\Delta+\boldsymbol{\rm M}_{X})\\
=&\sigma_{E}\bigl(f^{*}(K_{X}+\Delta+\boldsymbol{\rm M}_{X})\bigr)+{\rm ord}_{E}(G_{-}) \qquad \;\;\;\, ({\rm ord}_{E}(G_{-})=0)  \\
\geq &\sigma_{E}\bigl(f^{*}(K_{X}+\Delta+\boldsymbol{\rm M}_{X})+G_{-}\bigr) \qquad\qquad\quad\;\;\, (\text{\cite[Remark 2.3 (1)]{has-finite}})\\
=&\sigma_{E}\bigl(f'^{*}(K_{X'}+\Delta'+\boldsymbol{\rm M}'_{X'})+G_{+}\bigr) \\
=& \sigma_{E}\bigl(f'^{*}(K_{X'}+\Delta'+\boldsymbol{\rm M}'_{X'})\bigr)+{\rm ord}_{E}(G_{+}) \qquad (\text{\cite[Remark 2.3 (3)]{has-finite}}),
\end{split}
\end{equation*}
therefore $\sigma_{P}(K_{X}+\Delta+\boldsymbol{\rm M}_{X})>0$. 
By the second condition of Lemma \ref{lem--discre-relation}, we see that 
$a(E,X,\Delta+\boldsymbol{\rm M}_{X})\geq 1$, hence we have 
$$a(E,X',\Delta'+\boldsymbol{\rm M}'_{X'})>a(E,X,\Delta+\boldsymbol{\rm M}_{X})\geq 1.$$ 
By the same argument as in \cite[Proof of Lemma 2.5]{has-finite}, we see that $E|_{T}$ is exceptional over $S'$. 
Therefore, $G_{0}|_{T}$ is exceptional over $S'$. 
This implies that 
$${\rm coeff}_{Q}(G_{0}|_{T})=0$$
 for every prime divisor $Q$ on $S'$. 
Since we have $G_{1}|_{T}=0$ and
$$f|_{T}^{*}\bigl((K_{X}+\Delta+\boldsymbol{\rm M}_{X})|_{S}\bigr)-f'|_{T}^{*}\bigl((K_{X'}+\Delta'+\boldsymbol{\rm M}'_{X'})|_{S'}\bigr)=(G_{0}+G_{1})|_{T}-G_{-}|_{T},$$
 the inequality 
 $$a(Q,S',\Delta_{S'}+\boldsymbol{\rm N}_{S'})- a(Q,S,\Delta_{S}+\boldsymbol{\rm N}_{S})\leq 0$$ holds for every prime divisor $Q$ on $S'$. Therefore, Lemma \ref{lem--discre-relation} holds. 
\end{proof}

\begin{lem}[{cf.~\cite[Lemma 5.3]{hmx-boundgentype}, \cite[Lemma 2.25]{has-finite}}]\label{lem--minmodel-zariskidecom}
Let $(X,\Delta,\boldsymbol{\rm M})$ be a generalized lc pair such that $\boldsymbol{\rm M}$ is a finite $\mathbb{R}_{>0}$-linear combination of b-nef $\mathbb{Q}$-b-Cartier $\mathbb{Q}$-b-divisors. 
Let $(Y,\Gamma,\boldsymbol{\rm M})$ be a generalized lc pair with a projective birational morphism $f\colon Y\to X$. 
Suppose that $K_{X}+\Delta+\boldsymbol{\rm M}_{X}$ is pseudo-effective and all prime divisors $D$ on $Y$ satisfy
$$0\leq a(D,Y,\Gamma+\boldsymbol{\rm M}_{Y})-a(D,X,\Delta+\boldsymbol{\rm M}_{X})\leq \sigma_{D}(K_{X}+\Delta+\boldsymbol{\rm M}_{X}).$$  
Then $(X,\Delta,\boldsymbol{\rm M})$ has a minimal model (resp.~a good minimal model) if and only if $(Y,\Gamma,\boldsymbol{\rm M})$ has a minimal model (resp.~a good minimal model).  
\end{lem}

\begin{proof}
Let $g \colon W \to Y$ be a log resolution of $(Y,\Gamma)$ such that $f\circ g\colon W \to X$ is a log resolution of $(X,\Delta)$. 
Let $(W, \Delta_{W},\boldsymbol{\rm M})$ and $(W, \Gamma_{W},\boldsymbol{\rm M})$ be generalized log birational models of $(X,\Delta,\boldsymbol{\rm M})$ and $(Y,\Gamma,\boldsymbol{\rm M})$ as in Definition \ref{defn--models}, respectively. 
We may write 
$$K_{W}+\Delta_{W}+\boldsymbol{\rm M}_{W}=g^{*}f^{*}(K_{X}+\Delta+\boldsymbol{\rm M}_{X})+E_{W}$$
 for some $E_{W}\geq 0$ which is exceptional over $X$. 
 
By \cite[Remark 2.3 (3)]{has-finite}, all prime divisors $P$ on $W$ satisfy
$$\sigma_{P}(K_{W}+\Delta_{W}+\boldsymbol{\rm M}_{W})=\sigma_{P}(K_{X}+\Delta+\boldsymbol{\rm M}_{X})+{\rm coeff}_{P}(E_{W}).$$
If $P$ is not exceptional over $X$, then we have
$${\rm coeff}_{P}(E_{W})=0, \quad {\rm coeff}_{P}(\Delta_{W})={\rm coeff}_{f(g(P))}(\Delta),\quad {\rm and} \quad {\rm coeff}_{P}(\Gamma_{W})={\rm coeff}_{g(P)}(\Gamma).$$ 
If $P$ is exceptional over $X$ but not exceptional over $Y$, then we have
$${\rm coeff}_{P}(E_{W})=a(P,X,\Delta+\boldsymbol{\rm M}_{X}), \quad {\rm coeff}_{P}(\Delta_{W})=1, \quad {\rm and} \quad {\rm coeff}_{P}(\Gamma_{W})={\rm coeff}_{g(P)}(\Gamma).$$ 
If $P$ is exceptional over $Y$, then we have
$${\rm coeff}_{P}(E_{W})=a(P,X,\Delta+\boldsymbol{\rm M}_{X}), \quad {\rm coeff}_{P}(\Delta_{W})=1, \quad {\rm and} \quad {\rm coeff}_{P}(\Gamma_{W})=1.$$ 
In any case, by simple computations using the hypothesis about the relations between generalized log discrepancies, we see that 
$$0\leq a(P,W,\Gamma_{W}+\boldsymbol{\rm M}_{W})-a(P,W,\Delta_{W}+\boldsymbol{\rm M}_{W})\leq \sigma_{P}(K_{W}+\Delta_{W}+\boldsymbol{\rm M}_{W}).$$
Thus, we see that $(W, \Delta_{W},\boldsymbol{\rm M})$ and $(W, \Gamma_{W},\boldsymbol{\rm M})$ satisfy the hypothesis of Lemma \ref{lem--minmodel-zariskidecom}. 
By Lemma \ref{lem--equiv-log-birat}, we may replace $(X,\Delta,\boldsymbol{\rm M})$ and $(Y,\Gamma,\boldsymbol{\rm M})$ with $(W, \Delta_{W},\boldsymbol{\rm M})$ and $(W, \Gamma_{W},\boldsymbol{\rm M})$ respectively.
By the replacement, we may assume that $X=Y$ and $(X,0)$ is a $\mathbb{Q}$-factorial klt pair. 

We may carry out \cite[Proof of Lemma 5.3]{hmx-boundgentype} in the framework of generalized pairs. 
It is because we may use the argument of the length of extremal rays 
\cite[Proposition 3.17]{hanli} and the result of the termination of MMP \cite[Theorem 4.1]{hanli}. 
By the same argument as in \cite[Proof of Lemma 5.3]{hmx-boundgentype}, we see that Lemma \ref{lem--minmodel-zariskidecom} holds. 
\end{proof}

\begin{lem}[cf.~{\cite[Lemma 2.15]{has-mmp}}]\label{lem--equiv-model-birat}
Let $(X,\Delta,\boldsymbol{\rm M})$ be a generalized lc pair such that $\boldsymbol{\rm M}$ is a finite $\mathbb{R}_{>0}$-linear combination of b-nef $\mathbb{Q}$-b-Cartier $\mathbb{Q}$-b-divisors. 
Let $(Y,\Gamma,\boldsymbol{\rm M})$ be a generalized lc pair with a projective birational morphism $f\colon Y\to X$. 
Suppose that we can write 
$$K_{Y}+\Gamma+\boldsymbol{\rm M}_{Y}=f^{*}(K_{X}+\Delta+\boldsymbol{\rm M}_{X})+E$$
such that $E$ is effective and $f$-exceptional. 
Then $(X,\Delta,\boldsymbol{\rm M})$ has a minimal model (resp.~a good minimal model) if and only if $(Y,\Gamma,\boldsymbol{\rm M})$ has a minimal model (resp.~a good minimal model).  
\end{lem}

\begin{proof}
Replacing $(Y,\Gamma,\boldsymbol{\rm M})$ by a $\mathbb{Q}$-factorial generalized dlt model, we may assume that $(Y,\Gamma,\boldsymbol{\rm M})$ is a $\mathbb{Q}$-factorial generalized dlt pair. 
Running a $(K_{Y}+\Gamma+\boldsymbol{\rm M}_{Y})$-MMP over $X$ and replacing 
$(Y,\Gamma,\boldsymbol{\rm M})$ by a crepant generalized dlt model of $(X,\Delta,\boldsymbol{\rm M})$, we may assume $E=0$. 
Then Lemma \ref{lem--equiv-model-birat} follows from Lemma \ref{lem--minmodel-zariskidecom}. 
\end{proof}

\begin{lem}[cf.~{\cite[Lemma 2.22]{has-finite}}]\label{lem--mmpscaling-nonneflocus}
Let $(X,\Delta,\boldsymbol{\rm M})/Z$ be a generalized lc pair such that $Z$ is quasi-projective.
Let $S$ be a subvariety of $X$. 
We denote the morphism $X \to Z$ by $\pi$.
Let
$$(X,\Delta,\boldsymbol{\rm M})=:(X_{0},\Delta_{0},\boldsymbol{\rm M}) \dashrightarrow (X_{1},\Delta_{1},\boldsymbol{\rm M}) \dashrightarrow\cdots \dashrightarrow (X_{i},\Delta_{i},\boldsymbol{\rm M})\dashrightarrow \cdots$$ 
be a sequence of steps of a $(K_{X}+\Delta+\boldsymbol{\rm M}_{X})$-MMP over $Z$ with scaling of some $\mathbb{R}$-divisor $A\geq0$. 
We define 
$$\lambda_{i}={\rm inf}\set{\mu \in \mathbb{R}_{\geq0} \!|\! K_{X_{i}}+\Delta_{i}+\boldsymbol{\rm M}_{X_{i}}+\mu A_{i}\text{\rm \, is nef over }Z}.$$ 
Suppose that each step of the $(K_{X}+\Delta+\boldsymbol{\rm M}_{X})$-MMP is an isomorphism on a neighborhood of $S$ and ${\rm lim}_{i\to\infty}\lambda_{i}=0$.  
Then, for any $\pi$-ample $\mathbb{R}$-divisor $H$ on $X$ and any closed point $x\in S$, there exists $E\geq 0$ such that $E\sim_{\mathbb{R},Z}K_{X}+\Delta+\boldsymbol{\rm M}_{X}+H$ and ${\rm Supp}E \not\ni x$. 
In particular, if $Z$ is a point, then $\sigma_{P}(K_{X}+\Delta+\boldsymbol{\rm M}_{X})=0$ for every prime divisor $P$ over $X$ such that $c_{X}(P)$ intersects $S$. 
\end{lem}

\begin{proof}
The argument in \cite[Proof of Lemma 2.22]{has-finite} works with no changes. 
Note that \cite[Proof of Lemma 2.22]{has-finite} does not use any result about the abundance conjecture or the existence of flips for lc pairs. 
\end{proof}

\begin{lem}[cf.~{\cite[Lemma 2.26]{has-finite}}]\label{lem--bir-relation} Let $(X,\Delta,\boldsymbol{\rm M})$ and $(X',\Delta',\boldsymbol{\rm M}')$ be generalized lc pairs with a birational map $X\dashrightarrow X'$ such that $\boldsymbol{\rm M}$ is a finite $\mathbb{R}_{>0}$-linear combination of b-nef $\mathbb{Q}$-b-Cartier $\mathbb{Q}$-b-divisors and $\boldsymbol{\rm M}=\boldsymbol{\rm M}'$ by $X\dashrightarrow X'$. 
 Suppose in addition that \begin{itemize} 
 \item $a(P,X,\Delta+\boldsymbol{\rm M}_{X})\leq a(P,X',\Delta'+\boldsymbol{\rm M}'_{X'})$ for all prime divisors $P$ on $X$, and \item $a(P',X',\Delta'+\boldsymbol{\rm M}'_{X'})\leq a(P',X,\Delta+\boldsymbol{\rm M}_{X})$ for all prime divisors $P'$ on $X'$. \end{itemize} Then $K_{X}+\Delta+\boldsymbol{\rm M}_{X}$ is abundant if and only if $K_{X'}+\Delta'+\boldsymbol{\rm M}'_{X'}$ is abundant. 
Furthermore, $(X,\Delta,\boldsymbol{\rm M})$ has a minimal model (resp.~a good minimal model) if and only if $(X',\Delta',\boldsymbol{\rm M}')$ has a  minimal model (resp.~a good minimal model). \end{lem}

\begin{proof}
We closely follow \cite[Proof of Lemma 2.26]{has-finite}. 
Let $f\colon Y\to X$ and $f'\colon Y \to X'$ be a common log resolution of $(X,\Delta)\dashrightarrow (X',\Delta')$ such that $\boldsymbol{\rm M}$ descends to $Y$. 
We define an $\mathbb{R}$-divisor $\Gamma$ on $Y$ by 
$$\Gamma:=-\sum_{D}{\rm min}\{a(D,X,\Delta+\boldsymbol{\rm M}_{X})-1,a(D,X',\Delta'+\boldsymbol{\rm M}'_{X'})-1,0\}D,$$
 where $D$ runs over prime divisors on $Y$. 
Then $\Gamma$ is an effective snc $\mathbb{R}$-divisor, $(Y,\Gamma, \boldsymbol{\rm M})$ is generalized lc, and there exist an $f$-exceptional $\mathbb{R}$-divisor $E\geq 0$ and  an $f'$-exceptional $\mathbb{R}$-divisor $E'\geq 0$ such that 
$$E+f^{*}(K_{X}+\Delta+\boldsymbol{\rm M}_{X})=K_{Y}+\Gamma+\boldsymbol{\rm M}_{Y}=f'^{*}(K_{X'}+\Delta'+\boldsymbol{\rm M}'_{X'})+E'.$$
Lemma \ref{lem--bir-relation} follows from the relation and Lemma \ref{lem--equiv-model-birat}. 
\end{proof}

\begin{lem}[cf.~{\cite[Proposition 3.3]{has-mmp}}]\label{lem--mmp-abundant-fibration}
Let $(X,\Delta,\boldsymbol{\rm M})$ be a generalized lc pair such that $\boldsymbol{\rm M}$ is a finite $\mathbb{R}_{>0}$-linear combination of b-nef $\mathbb{Q}$-b-Cartier $\mathbb{Q}$-b-divisors. 
Let $\pi \colon X \to Z$ be a contraction to a normal projective variety $Z$. 
Suppose that 
\begin{itemize}
\item
$\kappa_{\iota}(X/Z,K_{X}+\Delta+\boldsymbol{\rm M}_{X})=\kappa_{\sigma}(X/Z,K_{X}+\Delta+\boldsymbol{\rm M}_{X})=0$, 
\item
any generalized lc center of $(X,\Delta,\boldsymbol{\rm M})$ dominates $Z$, and
\item
$\kappa_{\sigma}(X,K_{X}+\Delta+\boldsymbol{\rm M}_{X})={\rm dim}Z$. 
\end{itemize}
Then $(X,\Delta,\boldsymbol{\rm M})$ has a good minimal model.
\end{lem}

\begin{proof}
The argument in \cite[Proof of Proposition 3.3]{has-mmp} works with no changes because we can use the generalized canonical bundle formula \cite{hanliu} instead of the canonical bundle formula \cite{fg-bundle}. 
So we only outline the proof. 

By taking a log resolution of $(X,\Delta)$ and applying weak semistable reduction (\cite{ak}), we may assume that $(X,0)$ is $\mathbb{Q}$-factorial klt and all fibers of $\pi$ have the same dimensions. 

We run a $(K_{X}+\Delta+\boldsymbol{\rm M}_{X})$-MMP over $Z$ with scaling of an ample divisor. 
Applying the negativity lemma for very exceptional divisors (\cite[Section 3]{birkar-flip}) and replacing $(X,\Delta,\boldsymbol{\rm M})$, we may assume $K_{X}+\Delta+\boldsymbol{\rm M}_{X}\sim_{\mathbb{R},Z}0$ (but we lose the property of being equi-dimensional of $X \to Z$). 
We used the first condition of Lemma \ref{lem--mmp-abundant-fibration} for the reduction. 

By generalized canonical bundle formula \cite[Theorem 1.2]{hanliu} and the second condition of Lemma \ref{lem--mmp-abundant-fibration}, there is a generalized klt pair $(Z,\Delta_{Z},\boldsymbol{\rm N})$ such that 
$$K_{X}+\Delta+\boldsymbol{\rm M}_{X}\sim_{\mathbb{R}}\pi^{*}(K_{Z}+\Delta_{Z}+\boldsymbol{\rm N}_{Z}).$$ 
By the third condition of Lemma \ref{lem--mmp-abundant-fibration}, we see that $K_{Z}+\Delta_{Z}+\boldsymbol{\rm N}_{Z}$ is big, so $(Z,\Delta_{Z},\boldsymbol{\rm N})$ has a good minimal model $(Z', \Delta_{Z'}, \boldsymbol{\rm N})$ by \cite{bchm}. 
Since $(Z,\Delta_{Z},\boldsymbol{\rm N})$ is generalized klt, the birational map $Z \dashrightarrow Z'$ is a birational contraction. 
Hence we can find an open subset $U' \subset Z'$ such that ${\rm codim}_{Z'}(Z'\setminus U')\geq 2$ and $Z'\dashrightarrow Z$ is an isomorphism on $U'$.

Let $f\colon Y \to X$ be a log resolution of $(X,\Delta)$ such that $\boldsymbol{\rm M}$ descends to $Y$ and the map $\pi_{Y} \colon Y\dashrightarrow Z'$ is a morphism. 
Let $(Y,\Gamma,\boldsymbol{\rm M})$ be a generalized log birational model of $(X,\Delta,\boldsymbol{\rm M})$ as in Definition \ref{defn--models}. 
Then 
$$K_{Y}+\Gamma+\boldsymbol{\rm M}_{Y}\sim_{\mathbb{R}} \pi_{Y}^{*}(K_{Z'}+\Delta_{Z'}+\boldsymbol{\rm N}_{Z'})+E+F$$ 
for some $E \geq 0$ and $F\geq0$ such that $E$ is exceptional over $X$ and $\pi_{Y}({\rm Supp}F) \subset Z'\setminus U'$.
Then $(Y,\Gamma,\boldsymbol{\rm M})$ has a good minimal model over $U'$. 

Running a $(K_{Y}+\Gamma+\boldsymbol{\rm M}_{Y})$-MMP over $Z'$ with scaling of an ample divisor, we get a birational contraction $\phi \colon Y \dashrightarrow Y'$ over $Z'$ with the morphism $\pi' \colon Y' \to Z'$ such that $$\phi_{*}(K_{Y}+\Gamma+\boldsymbol{\rm M}_{Y}) \sim_{\mathbb{R}}\pi'^{*}(K_{Z'}+\Delta_{Z'}+\boldsymbol{\rm N}_{Z'})+\phi_{*}E+\phi_{*}F$$ is the limit of movable divisors over $Z'$ and $\phi_{*}E|_{\pi'^{-1}(U')}=0$. 
Since ${\rm codim}_{Z'}(Z'\setminus U')\geq 2$, applying the negativity lemma for very exceptional divisors (\cite[Section 3]{birkar-flip}) to $\phi_{*}E+\phi_{*}F$, we have $\phi_{*}E+\phi_{*}F=0$. 
Thus we have
$$\phi_{*}(K_{Y}+\Gamma+\boldsymbol{\rm M}_{Y})\sim_{\mathbb{R}}\pi'^{*}(K_{Z'}+\Delta_{Z'}+\boldsymbol{\rm N}_{Z'}),$$ 
and the right hand side is semi-ample. 
From the fact, we see that $(Y,\Gamma,\boldsymbol{\rm M})$ has a good minimal model. 
So $(X,\Delta,\boldsymbol{\rm M})$ has a good minimal model. 
\end{proof}

\subsection{Generalized abundance}

In this subsection, we study the property of being log abundant for special generalized lc pairs under MMP. 

\begin{thm}\label{thm--logbig-boundary}
Let $(X,\Delta)$ be a projective lc pair, and let $\pi \colon X\to Z$ be a projective surjective morphism to a normal projective variety $Z$. 
Suppose that there is an effective $\mathbb{R}$-Cartier divisor $C$ on $X$ such that
\begin{itemize}
\item 
the pair $(X,\Delta+tC)$ is lc for some $t>0$, and  
\item 
$K_{X}+\Delta+C \sim_{\mathbb{R},Z}0$. 
\end{itemize}
Let $A_{Z}$ be a big and semi-ample $\mathbb{R}$-divisor on $Z$ such that the pullback of $A_{Z}$ to $\pi(S)^{\nu}$ is big for any lc center $S$ of $(X,\Delta)$, where $\pi(S)^{\nu}$ is the normalization of $\pi(S)$. 
Then, $K_{X}+\Delta+\pi^{*}A_{Z}$ is abundant. 
\end{thm}

We will use Lemma \ref{lem--abund-birat} below to prove Theorem \ref{thm--logbig-boundary}. 

\begin{lem}\label{lem--abund-birat}
Assume Theorem \ref{thm--logbig-boundary} for all projective lc pairs of dimension at most $n-1$. 
Let $(X,\Delta)$ be a projective lc pair, let $\pi\colon X\to Z$ be a morphism, and let $C$ and $A_{Z}$ be $\mathbb{R}$-Cartier divisors as in Theorem \ref{thm--logbig-boundary} such that ${\rm dim}X\leq n- 1$. 
Let $f\colon Y \to X$ be a log resolution of $(X,\Delta)$, and let $\Gamma \geq 0$ be an $\mathbb{R}$-divisor on $Y$ such that $(Y,\Gamma)$ is an lc pair and the effective part of the divisor $K_{Y}+\Gamma-f^{*}(K_{X}+\Delta)$ is $f$-exceptional. 
Then, $K_{Y}+\Gamma+f^{*}\pi^{*}A_{Z}$ is abundant. 
\end{lem}

\begin{proof}
The argument in \cite[Proof of Lemma 5.2]{hashizumehu} or \cite[Proof of Lemma 3.6]{has-nonvan-gpair} works with no changes. 
\end{proof}

\begin{proof}[Proof of Theorem \ref{thm--logbig-boundary}]
The argument of \cite[Proof of Theorem 5.4]{hashizumehu} works with no changes. 
We only outline the proof. 
We prove Theorem \ref{thm--logbig-boundary} by induction on ${\rm dim}X$. 

We put $A= \pi^{*}A_{Z}$. 
We may assume that $K_{X}+\Delta+A$ is pseudo-effective and $\pi$ is a contraction. 
By replacing $A$ with a general one, we may also assume that $(X,\Delta+A)$ is lc, $\Delta$ and $A$ have no common components, and all lc centers of $(X,\Delta+A)$ are lc centers of $(X,\Delta)$. 
By taking a dlt blow-up of $(X,\Delta+A)$, we may assume that $(X,\Delta+A)$ is $\mathbb{Q}$-factorial dlt. 

If $(X,\Delta+A)$ is a klt pair, then we follow \cite[Proof of Lemma 5.3]{hashizumehu}, and we see that $K_{X}+\Delta+A$ is abundant. 
More generally, if $K_{X}+\Delta-\epsilon \lfloor \Delta \rfloor+A$ is pseudo-effective for some $\epsilon >0$, then the argument in \cite[Step 1 in the proof of Theorem 5.4]{hashizumehu} enables us to reduce to the klt case, hence we see that $K_{X}+\Delta+A$ is abundant. 
From these facts, we may assume that $K_{X}+\Delta-\epsilon \lfloor \Delta \rfloor+A$ is not pseudo-effective for every $\epsilon>0$. 
Then there is a component $S$ of $\lfloor \Delta\rfloor$ such that $K_{X}+\Delta-\epsilon S+A$ is not pseudo-effective for every $\epsilon>0$. 

For all $\epsilon'>0$, we run a $(K_{X}+\Delta-\epsilon' S+A)$-MMP and get a birational contraction $X\dashrightarrow X'$ to a Mori fiber space $X'\to \overline{Z}$. 
Let $\Delta'$, $S'$ and $A'$ be the birational transforms of $\Delta$, $S$ and $A$ on $X'$, respectively. 
By taking $\epsilon'>0$ sufficiently small with the aid of the ACC for lc thresholds (\cite[Theorem 1.1]{hmx-acc}) and the ACC for numerically trivial pairs (\cite[Theorem 1.5]{hmx-acc}), we see that $(X',\Delta'+A')$ is lc and we have the relation $K_{X'}+\Delta'+A'\sim_{\mathbb{R},\overline{Z}}0$.

Let $\phi \colon Y \to X$ and $ Y \to X'$ be a common log resolution of the birational map $(X,\Delta)\dashrightarrow (X',\Delta')$. 
Putting $A_{Y}=\phi^{*}A$, we can write 
\begin{equation*}\tag{$\dagger$}\label{proof--thm--logbig-boundary-dagger}
K_{Y}+\Gamma+A_{Y}=\phi^{*}(K_{X}+\Delta+A)+E,
\end{equation*}
where $\Gamma\geq0$ and $E\geq0$ have no common components. 
We run a $(K_{Y}+\Gamma+A_{Y})$-MMP over $\overline{Z}$ with scaling of an ample divisor. 
As in \cite[Step 3 in the proof of Theorem 5.4]{hashizumehu}, an appropriate $(K_{Y}+\Gamma+A_{Y})$-MMP terminates with a good minimal model $(Y',\Gamma'+A_{Y'})$ over $\overline{Z}$. 
Let $Y'\to Z'$ be the contraction over $\overline{Z}$ induced by $K_{Y'}+\Gamma'+A_{Y'}$. 
Then the induced morphism $Z' \to \overline{Z}$ is birational. 
We have the following diagram.
\begin{equation*}
\xymatrix@R=15pt
{
Y\ar[d]_{\phi}\ar@{-->}[r]&Y'\ar[d]\\
X\ar[d]_{\pi}&Z'\\
Z
}
\end{equation*}

We put $T=\phi^{-1}_{*}S$, and let $T'$ be the birational transform of $T$ on $Y'$. 
By construction, $T'$ dominates $Z'$. 
We define projective dlt pairs 
$$(S,\Delta_{S}), \quad (T,\Gamma_{T}),\quad {\rm and} \quad (T',\Gamma_{T'})$$ with the divisorial adjunctions $K_{S}+\Delta_{S}=(K_{X}+\Delta)|_{S}$, $K_{T}+\Gamma_{T}=(K_{Y}+\Gamma)|_{T}$, and $K_{T'}+\Gamma_{T'}=(K_{Y'}+\Gamma')|_{T'}$, respectively. 
We put 
$$A_{S}=A|_{S}, \quad A_{T}=A_{Y}|_{T}, \quad {\rm and} \quad A_{T'}=A_{Y'}|_{T'}.$$ 

By construction, it is sufficient to prove that $K_{Y'}+\Gamma'+A_{Y'}$ is abundant. 
Since we have $K_{Y'}+\Gamma'+A_{Y'}\sim_{\mathbb{R},Z'}0$ and $T'$ dominates $Z'$, it is sufficient to prove that 
$$K_{T'}+\Gamma_{T'}+A_{T'} \sim_{\mathbb{R}}(K_{Y'}+\Gamma'+A_{Y'})|_{T'}$$
 is abundant. 
We take a common log resolution $\tau\colon T'' \to T$ and $\tau'\colon T''\to T'$ of the birational map $(T,\Gamma_{T}) \dashrightarrow (T',\Gamma_{T'})$. 
Replacing $A$ with a general member of $|A|_{\mathbb{R}}$, we may assume that
\begin{itemize}
\item $A_{T}\geq0$ and $A_{T'}\geq0$, 
\item
$\tau^{*}A_{T}\leq \tau_{*}'^{-1}A_{T'}$ (see \cite[Lemma 2.1]{hashizumehu}), and 
\item
$\tau'\colon T''\to T'$ is a log resolution of $(T',\Gamma_{T'}+A_{T'})$ (see  \cite[Lemma 2.2]{hashizumehu}). 
\end{itemize}

Let $\pi_{S}\colon S\to Z$ be the restriction of $\pi\colon X\to Z$ to $S$, and let $\phi_{T} \colon T \to S$ be the birational morphism induced by $\phi \colon Y \to X$. 
Now we have the following diagram.
\begin{equation*}
\xymatrix@R=15pt{
&
T''\ar[dl]_{\tau}\ar[dr]^{\tau'}&\\
T \ar[d]_{\phi_{T}}\ar@{-->}[rr]&&T'\ar[d]\\
S\ar[d]_{\pi_{S}}&&Z' \\
Z&&
}
\end{equation*}
By restricting (\ref{proof--thm--logbig-boundary-dagger}) to $T$, we get 
\begin{equation*}
K_{T}+\Gamma_{T}+A_{T}=\phi_{T}^{*}(K_{S}+\Delta_{S}+A_{S})+E|_{T}
\end{equation*}
such that $\Gamma_{T}$ and $E|_{T}$ are effective. 
By \cite[Lemma 2.4]{hashizumehu}, we see that  $\Gamma_{T}+A_{T}$ and $E|_{T}$ have no common components and 
$E|_{T}$ is $\phi_{T}$-exceptional. 
We put $A_{T''}=\tau^{*}A_{T}$. 
Then $A_{T''} \leq \tau'^{-1}_{*}A_{T'}$. 
We have
\begin{equation*}
K_{T''}+\Psi_{T''}+A_{T''}=\tau'^{*}(K_{T'}+\Gamma_{T'}+A_{T'})+E_{T''}
\end{equation*}
for some $\Psi_{T''}\geq0$ and $E_{T''}\geq0$ such that 
$(T'', \Psi_{T''}+A_{T''})$ is a log smooth lc pair and $\Psi_{T''}+A_{T''}$ and $E_{T''}$ have no common components.

We may write
\begin{equation*}
K_{T''}+\Psi_{T''}+A_{T''}=\tau^{*}\phi_{T}^{*}(K_{S}+\Delta_{S}+A_{S})+M-N,
\end{equation*}
with $M\geq0$ and $N\geq0$ having no common components. 
Pick any component $Q$ of $M$. 
If $Q$ is not exceptional over $S$, then we have $a(Q, T'', \Psi_{T''}+A_{T''}) <a(Q,S, \Delta_{S}+A_{S})\leq 1$.  
We apply \cite[Step 6 in the proof of Theorem 5.4]{hashizumehu}, and we get the following contradiction
\begin{equation*} 
\begin{split}
a(Q,S, \Delta_{S}+A_{S})=&{\rm min}\{1, a(Q,S, \Delta_{S}+A_{S}) \} \qquad \;\;\, (a(Q,S, \Delta_{S}+A_{S})\leq1)\\ 
=&{\rm min}\{1, a(Q, T, \Gamma_{T}+A_{T}) \} \qquad \;\;\,(\text{$E|_{T}$ is $\phi_{T}$-exceptional})\\
\leq &{\rm min}\{1, a(Q, T', \Gamma_{T'}+A_{T'}) \} \qquad (\text{\cite[Lemma 4.2.10]{fujino-sp-ter}})\\
=&a(Q, T'', \Psi_{T''}+A_{T''}) \qquad \qquad  \quad (\text{\cite[Lemma 2.3]{hashizumehu}})\\
<&a(Q,S, \Delta_{S}+A_{S}).
\end{split} 
\end{equation*}
From this, we see that $M$ is exceptional over $S$. 

By applying the induction hypothesis of Theorem \ref{thm--logbig-boundary} to $(S,\Delta_{S}) \to \pi(S)^{\nu}$, $C|_{S}$, and $A_{Z}|_{\pi(S)^{\nu}}$, we see that $K_{S}+\Delta_{S}+A_{S}$ is abundant. 
By applying Lemma \ref{lem--abund-birat} to $(S,\Delta_{S}) \to \pi(S)^{\nu}$ and $(T'',\Psi_{T''}) \to S$, we see that $K_{T''}+\Psi_{T''}+A_{T''}$ is abundant. 
Then $K_{T'}+\Gamma_{T'}+A_{T'}$ is abundant, hence $K_{X}+\Delta+A$ is abundant. 
\end{proof}

\begin{cor}[{cf.~\cite[Theorem 4.1]{has-nonvan-gpair}}]\label{cor--abundant-mmp}
Let $(X,\Delta)$ be a projective lc pair and $M\geq 0$ a semi-ample $\mathbb{R}$-divisor on $X$ such that $M$ is log big with respect to $(X,\Delta)$ and $(X,\Delta+M)$ is an lc pair whose lc centers coincide with those of $(X,\Delta)$. Then, for any sequence of steps of a $(K_{X}+\Delta+M)$-MMP starting the lc pair $(X,\Delta+M)$
$$(X,\Delta+M)\dashrightarrow (X',\Delta'+M'),$$ the divisor $K_{X'}+\Delta'+M'$ is log abundant with respect to $(X',\Delta'+M')$.   
\end{cor}

\begin{proof}
We can apply Lemma \ref{lem--abund-birat}, and the argument in \cite[Proof of Theorem 4.1]{has-nonvan-gpair} works with no changes. 
\end{proof}

\subsection{Minimal model program}

In this subsection, we study the termination of MMP for generalized lc pairs that preserves the property of being log abundant. 

\begin{thm}[{cf.~\cite[Theorem 3.4]{has-finite}}]\label{thm--mmp-induction-2}
Let $(X,\Delta, \boldsymbol{\rm M})$ be a generalized dlt pair such that $\boldsymbol{\rm M}$ is a finite $\mathbb{R}_{>0}$-linear combination of b-nef $\mathbb{Q}$-b-Cartier $\mathbb{Q}$-b-divisors. 
Suppose that
\begin{itemize}
\item
$K_{X}+\Delta+\boldsymbol{\rm M}_{X}$ is pseudo-effective and abundant, 
\item for any generalized lc center $S$ of $(X,\Delta, \boldsymbol{\rm M})$, the restriction $(K_{X}+\Delta+\boldsymbol{\rm M}_{X})|_{S}$ is nef, and
\item
for any prime divisor $P$ over $X$ such that $c_{X}(P)$ intersects a generalized lc center of $(X,\Delta, \boldsymbol{\rm M})$ and $a(P,X,\Delta+\boldsymbol{\rm M}_{X})< 1$, we have $\sigma_{P}(K_{X}+\Delta+\boldsymbol{\rm M}_{X})=0$. 
\end{itemize}
Then $(X,\Delta, \boldsymbol{\rm M})$ has a minimal model. 
\end{thm}

\begin{proof}
We closely follow \cite[Subsection 3.1]{has-finite}. 
We prove Theorem \ref{thm--mmp-induction-2} in several steps. 

\begin{step1}\label{step1-mmp-induction-2}
In this step, we replace $(X,\Delta, \boldsymbol{\rm M})$ with a crepant $\mathbb{Q}$-factorial generalized dlt model which has good properties. 
We follow \cite[Proof of Proposition 3.2]{has-finite}. 

Since $K_{X}+\Delta+\boldsymbol{\rm M}_{X}$ is pseudo-effective and abundant, there is an effective $\mathbb{R}$-divisor $D$ on $X$ such that $K_{X}+\Delta+\boldsymbol{\rm M}_{X}\sim_{\mathbb{R}}D$. 
We take the Iitaka fibration 
$X\dashrightarrow V$ associated to $D$. 
Then we have 
$${\rm dim}V=\kappa_{\sigma}(X, K_{X}+\Delta+\boldsymbol{\rm M}_{X}).$$ 

We take a log resolution $\overline{f}\colon \overline{X}\to X$ of $(X,\Delta)$ such that $\boldsymbol{\rm M}$ descends to $\overline{X}$ and the induced map $\overline{X}\dashrightarrow V$ is a morphism. 
Then we have
$$K_{\overline{X}}+\overline{\Delta}+\boldsymbol{\rm M}_{\overline{X}}=\overline{f}^{*}(K_{X}+\Delta+\boldsymbol{\rm M}_{X})+\overline{E}$$ 
for some $\overline{\Delta} \geq 0$ and $\overline{E}\geq0$ which have no common components. 
By construction and Lemma \ref{lem--iitakafib}, we have
\begin{enumerate}[(i)]
\item \label{proof-prop--creppantmmp-(i)}
$\kappa_{\sigma}(\overline{X},K_{\overline{X}}+\overline{\Delta}+\boldsymbol{\rm M}_{\overline{X}})={\rm dim}V$ and $\kappa_{\sigma}(\overline{X}/V,K_{\overline{X}}+\overline{\Delta}+\boldsymbol{\rm M}_{\overline{X}})=0$. 
\end{enumerate}
Moreover, $K_{\overline{X}}+\overline{\Delta}+\boldsymbol{\rm M}_{\overline{X}}$ is $\mathbb{R}$-linearly equivalent to the sum of an effective $\mathbb{R}$-divisor and the pullback of an ample divisor on $V$. 
Therefore, we can find an effective $\mathbb{R}$-divisor 
$$\overline{D}\sim_{\mathbb{R}}K_{\overline{X}}+\overline{\Delta}+\boldsymbol{\rm M}_{\overline{X}}$$
 such that ${\rm Supp}\overline{D}$ contains all generalized  lc centers of $(\overline{X},\overline{\Delta},\boldsymbol{\rm M})$ which are vertical over $V$. 
By applying \cite[Lemma 2.10]{has-trivial} to the morphism $(\overline{X},\overline{\Delta})\to V$ and replacing $(\overline{X},\overline{\Delta})$, we may assume that
\begin{enumerate}[(i)]\setcounter{enumi}{1}
\item \label{proof-prop--creppantmmp-(ii)}
$\overline{\Delta}=\overline{\Delta}'+\overline{\Delta}''$ such that $\overline{\Delta}'$ is effective, $\Delta''=0$ or $\overline{\Delta}''$ is reduced and vertical over $V$, and all generalized lc centers of $(\overline{X},\overline{\Delta}-s\overline{\Delta}'', \boldsymbol{\rm M})$ dominate $V$ for all $s \in (0,1]$. 
\end{enumerate}

By taking a log resolution of $(\overline{X},\overline{\Delta}+\overline{D})$ and replacing $(\overline{X},\overline{\Delta})$ and $\overline{D}$, we may assume that $(\overline{X},\overline{\Delta}+\overline{D})$ is log smooth.
Note that (\ref{proof-prop--creppantmmp-(ii)}) is preserved after this replacement. 
Since $ {\rm Supp}\overline{\Delta}''$ is a union of generalized lc centers of $(\overline{X},\overline{\Delta},\boldsymbol{\rm M})$ which are vertical over $V$, we see that ${\rm Supp}\overline{D}\supset {\rm Supp}\overline{\Delta}''$. 
By decomposing $\overline{D}$ appropriately, we obtain effective $\mathbb{R}$-divisors $\overline{G}$ and $\overline{H}$ such that 
\begin{enumerate}[(i)]\setcounter{enumi}{2}
\item \label{proof-prop--creppantmmp-(iii)}
$K_{\overline{X}}+\overline{\Delta}+\boldsymbol{\rm M}_{\overline{X}}\sim_{\mathbb{R}}\overline{G}+\overline{H}$, 
\item \label{proof-prop--creppantmmp-(iv)}
${\rm Supp}\overline{\Delta}''\subset {\rm Supp}\overline{G}\subset {\rm Supp}\lfloor \overline{\Delta} \rfloor$, and 
\item \label{proof-prop--creppantmmp-(v)}
no component of $\overline{H}$ is a component of $\lfloor \overline{\Delta} \rfloor$ and $(\overline{X},\overline{\Delta}+\overline{H})$ is log smooth.   
\end{enumerate}

We fix a real number $t_{0}\in(0,1)$ such that $\overline{\Delta}-t_{0}\overline{G}\geq 0$. 
We pick $t\in(0,t_{0}]$, and we consider $(\overline{X},\overline{\Delta}-t\overline{G}, \boldsymbol{\rm M})$.
Then the morphism 
$$(\overline{X},\overline{\Delta}-t\overline{G}, \boldsymbol{\rm M})\to V$$
 satisfies the following three conditions. 
\begin{itemize}
\item
Any generalized lc center of $(\overline{X},\overline{\Delta}-t\overline{G}, \boldsymbol{\rm M})$ dominates $V$, 
\item
$\kappa_{\sigma}(\overline{X},K_{\overline{X}}+\overline{\Delta}-t\overline{G}+\boldsymbol{\rm M}_{\overline{X}})={\rm dim}V$, and 
\item
$\kappa_{\iota}(\overline{X}/V,K_{\overline{X}}+\overline{\Delta}-t\overline{G}+\boldsymbol{\rm M}_{\overline{X}}) = \kappa_{\sigma}(\overline{X}/V,K_{\overline{X}}+\overline{\Delta}-t\overline{G}+\boldsymbol{\rm M}_{\overline{X}})=0$.
\end{itemize}
Indeed, the first condition follows from (\ref{proof-prop--creppantmmp-(iv)}) and (\ref{proof-prop--creppantmmp-(ii)}), and the other two conditions follow from (\ref{proof-prop--creppantmmp-(iii)}), (\ref{proof-prop--creppantmmp-(i)}), and \cite[Remark 2.8 (1)]{hashizumehu}. 
By Lemma \ref{lem--mmp-abundant-fibration}, we see that $(\overline{X},\overline{\Delta}-t\overline{G}, \boldsymbol{\rm M})$ has a good minimal model for all $t\in(0,t_{0}]$. 

Running a $(K_{\overline{X}}+\overline{\Delta}+\boldsymbol{\rm M}_{\overline{X}})$-MMP over $X$, we get a crepant generalized dlt model 
$$\widetilde{f}\colon(\widetilde{X},\widetilde{\Delta},\boldsymbol{\rm M})\to (X,\Delta, \boldsymbol{\rm M}).$$  
Let $\widetilde{G}$ (resp.~$\widetilde{H}$) be the birational transform of $\overline{G}$ (resp.~$\overline{H}$) on $\widetilde{X}$. 
Then 
$$K_{\widetilde{X}}+\widetilde{\Delta}+\boldsymbol{\rm M}_{\widetilde{X}}\sim_{\mathbb{R}}\widetilde{G}+\widetilde{H}$$
 by (\ref{proof-prop--creppantmmp-(iii)}). 
By (\ref{proof-prop--creppantmmp-(v)}) and replacing $t_{0}$, we may assume that $(\overline{X},\overline{\Delta}+t_{0}\overline{H}, \boldsymbol{\rm M})$ is a $\mathbb{Q}$-factorial generalized dlt pair, and we may further assume that $\overline{X}\dashrightarrow \widetilde{X}$ is a sequence of steps of a $(K_{\overline{X}}+\overline{\Delta}+t\overline{H}+\boldsymbol{\rm M}_{\overline{X}})$-MMP and a sequence of steps of a $(K_{\overline{X}}+\overline{\Delta}-t\overline{G}+\boldsymbol{\rm M}_{\overline{X}})$-MMP for all $t\in(0,t_{0}]$. 
Then it follows that $(\widetilde{X},\widetilde{\Delta}+t_{0}\widetilde{H},\boldsymbol{\rm M})$ is a $\mathbb{Q}$-factorial generalized dlt pair, and $(\widetilde{X},\widetilde{\Delta}-t\widetilde{G},\boldsymbol{\rm M})$ has a good minimal model for all $t\in(0,t_{0}]$ since $(\overline{X},\overline{\Delta}-t\overline{G}, \boldsymbol{\rm M})$ has a good minimal model. 

Applying \cite[Lemma 2.4]{has-finite} to $K_{\widetilde{X}}+\widetilde{\Delta}+\boldsymbol{\rm M}_{\widetilde{X}}$ and $\widetilde{H}$ and replacing $t_{0}$, we may assume that 
$${\rm Supp}N_{\sigma}(K_{\widetilde{X}}+\widetilde{\Delta}+t\widetilde{H}+\boldsymbol{\rm M}_{\widetilde{X}})$$
is independent of the choice of $t\in(0,t_{0}]$. 

Since $K_{\widetilde{X}}+\widetilde{\Delta}+\boldsymbol{\rm M}_{\widetilde{X}}=\widetilde{f}^{*}(K_{X}+\Delta+\boldsymbol{\rm M}_{X})$, it is easy to check that $(\widetilde{X},\widetilde{\Delta},\boldsymbol{\rm M})$ satisfies all the conditions of Theorem \ref{thm--mmp-induction-2}. 
Replacing $ (X,\Delta, \boldsymbol{\rm M})$ by $(\widetilde{X},\widetilde{\Delta},\boldsymbol{\rm M})$, we may assume that $X$ is $\mathbb{Q}$-factorial and there exist effective $\mathbb{R}$-divisors $G$ and $H$ on $X$ such that 
\begin{itemize}
\item 
$K_{X}+\Delta+\boldsymbol{\rm M}_{X}\sim_{\mathbb{R}}G+H$, 
\item 
${\rm Supp}G\subset {\rm Supp}\lfloor \Delta \rfloor$, and
\item 
there exists a real number $t_{0}>0$ such that the following properties hold for all $t\in(0,t_{0}]$:  
\begin{itemize}
\item 
$(X,\Delta+tH, \boldsymbol{\rm M})$ is generalized dlt and ${\rm Supp}N_{\sigma}(K_{X}+\Delta+tH+\boldsymbol{\rm M}_{X})$ does not depend on $t$, and 
\item 
$(X,\Delta-tG, \boldsymbol{\rm M})$ has a good minimal model. 
\end{itemize}
\end{itemize}
\end{step1}

\begin{step1}\label{step2-mmp-induction-2}
We follow \cite[Proof of Proposition 3.3]{has-finite}. 

Pick any real number $t \in (0,t_{0}]$. 
Since $(X,\Delta-\tfrac{t}{1+t}G, \boldsymbol{\rm M})$ has a good minimal model and 
$$K_{X}+\Delta+tH+\boldsymbol{\rm M}_{X}\sim_{\mathbb{R}} (1+t)\left(K_{X}+\Delta-\frac{t}{1+t}G+\boldsymbol{\rm M}_{X}\right),$$
there is a sequence of steps of a $(K_{X}+\Delta+tH+\boldsymbol{\rm M}_{X})$-MMP to a good minimal model
$$\phi_{t}\colon (X,\Delta+tH, \boldsymbol{\rm M}) \dashrightarrow (X_{t},\Delta_{t}+tH_{t}, \boldsymbol{\rm M}).$$
Because ${\rm Supp}N_{\sigma}(K_{X}+\Delta+tH+\boldsymbol{\rm M}_{X})$ is independent of $t \in (0,t_{0}]$, prime divisors contracted by $\phi_{t}$ are independent of $t \in  (0,t_{0}]$. 
Therefore, putting
$$X_{0}=X_{t_{0}}, \quad \Delta_{0}=\Delta_{t_{0}}, \quad {\rm and} \quad H_{0}=H_{t_{0}},$$ 
then $X_{0}$ and $X_{t}$ are isomorphic in codimension one for all $t \in (0,t_{0}]$.  
From the fact, we see that 
$(X_{t},\Delta_{t}+tH_{t}, \boldsymbol{\rm M})$ is a good minimal model of $(X_{0},\Delta_{0}+tH_{0},\boldsymbol{\rm M})$. 

By applying \cite[Proof of Lemma 2.14]{has-mmp} and \cite[Proof of Proposition 3.3]{has-finite}, we get a sequence of steps of a $(K_{X_{0}}+\Delta_{0}+\boldsymbol{\rm M}_{X_{0}})$-MMP with scaling of $t_{0}H_{0}$
$$(X_{0},\Delta_{0},\boldsymbol{\rm M}) \dashrightarrow (X_{1}, \Delta_{1},\boldsymbol{\rm M}) \dashrightarrow \cdots \dashrightarrow (X_{i}, \Delta_{i},\boldsymbol{\rm M}) \dashrightarrow \cdots$$
such that 
\begin{itemize}
\item
if we define $t_{i}={\rm inf}\set{\mu \in \mathbb{R}_{\geq0}|\text{$K_{X_{i}}+\Delta_{i}+\boldsymbol{\rm M}_{X_{i}}+\mu H_{i}$ is nef}}$ for each $i$, then ${\rm lim}_{i \to \infty}t_{i}=0$,
\item
for all $t \in [t_{i},t_{i-1}] \cap \mathbb{R}_{>0}$, the generalized pair $(X_{i}, \Delta_{i}+t H_{i},\boldsymbol{\rm M})$ is a good minimal model of both $(X, \Delta+t H,\boldsymbol{\rm M})$ and $(X_{0}, \Delta_{0}+t H_{0},\boldsymbol{\rm M})$, and
\item
the MMP occurs only in ${\rm Supp}\lfloor \Delta_{0}\rfloor$. 
\end{itemize}
By the argument as in \cite[Step 1 in the proof of Theorem 3.4]{has-finite}, it is sufficient to prove the termination of the $(K_{X_{0}}+\Delta_{0}+\boldsymbol{\rm M}_{X_{0}})$-MMP. 
\end{step1}

\begin{step1}\label{step3-mmp-induction-2}
We follow \cite[Step 1 in the proof of Theorem 3.4]{has-finite}. 

For every $i$ and generalized lc center $S_{i}$ of $(X_{i},\Delta_{i}, \boldsymbol{\rm M})$, we define a generalized dlt pair 
$(S_{i}, \Delta_{S_{i}}, \boldsymbol{\rm N})$ by applying divisorial adjunction for generalized pairs repeatedly. 
We put $H_{S_{i}}=H_{i}|_{S_{i}}$. 
Similarly, for every generalized lc center $S$ of $(X,\Delta, \boldsymbol{\rm M})$, we define a generalized dlt pair $(S, \Delta_{S}, \boldsymbol{\rm N})$ by applying divisorial adjunction for generalized pairs repeatedly. 
We put $H_{S}=H|_{S}$. 

There exists $m_{0}\gg 0 $ such that for every $i \geq m_{0}$, the birational map $X_{m_{0}}\dashrightarrow X_{i}$ is an isomorphism on a neighborhood of the generic points of all generalized lc centers of $(X_{m_{0}},\Delta_{m_{0}},\boldsymbol{\rm M})$. 

We will prove the termination of the $(K_{X_{0}}+\Delta_{0}+\boldsymbol{\rm M}_{X_{0}})$-MMP by applying the special termination for MMP for generalized dlt pairs (cf.~\cite[Subsection 4.2]{hanli}). 
More precisely, by induction on $l \in \mathbb{Z}_{\geq 0}$, we will prove  that the MMP terminates on a neighborhood of all $l$-dimensional generalized lc centers of $(X_{m_{0}},\Delta_{m_{0}},\boldsymbol{\rm M})$ for all $l$. 
Pick $l \in \mathbb{Z}_{\geq 0}$. 
By the induction hypothesis and the standard argument of the special termination, we can find $m\gg m_{0}$ such that for every $i \geq m$ and every generalized lc center $S_{m}$ of $(X_{m},\Delta_{m}, \boldsymbol{\rm M})$ with ${\rm dim}S_{m} \leq l$, the induced birational map $S_{m} \dashrightarrow S_{i}$ to the corresponding generalized lc center $S_{i}$ of $(X_{i},\Delta_{i}, \boldsymbol{\rm M})$ is small and the birational transform of $\Delta_{S_{m}}$ (resp.~$H_{S_{m}}$, $\boldsymbol{\rm N}_{S_{m}}$) to $S_{i}$ is equal to $\Delta_{S_{i}}$ (resp.~$H_{S_{i}}$, $\boldsymbol{\rm N}_{S_{i}}$). 

To prove the special termination, it is sufficient to prove the existence of a minimal model of $(S_{m}, \Delta_{S_{m}},\boldsymbol{\rm N})$ for every generalized lc center $S_{m}$ of $(X_{m},\Delta_{m},\boldsymbol{\rm M})$. 
\end{step1}

\begin{step1}\label{step4-mmp-induction-2}
We follow \cite[Step 2 in the proof of Theorem 3.4]{has-finite}. 

From now on, we fix a generalized lc center $S_{m}$ of $(X_{m},\Delta_{m},\boldsymbol{\rm M})$. 
Unless otherwise stated, we assume that all integers $i$ appearing in the rest of the proof satisfy $i \geq m$.  
For all $i$, let $S_{i}$ be the generalized lc center of $(X_{i},\Delta_{i}, \boldsymbol{\rm M})$ such that $X_{m}\dashrightarrow X_{i}$ induces a birational map $S_{m} \dashrightarrow S_{i}$.  
By construction of the birational map $X\dashrightarrow X_{i}$, we can find a generalized lc center $S$ of $(X,\Delta, \boldsymbol{\rm M})$ such that $X\dashrightarrow X_{i}$ induces a birational map $S\dashrightarrow S_{i}$. 
Using the $S$, we will define the following variety, divisor and inequalities. 

\begin{enumerate}[(a)]
\item \label{proof-thm--ind-1-step3-(a)}
A birational morphism $\psi\colon T\to S_{m}$ from a projective $\mathbb{Q}$-factorial variety $T$ such that for every prime divisor $\bar{D}$ on $S$, if we have $a(\bar{D},S_{m}, \Delta_{S_{m}}+\boldsymbol{\rm N}_{S_{m}})<a(\bar{D},S,\Delta_{S}+\boldsymbol{\rm N}_{S})$ 
then $\bar{D}$ is a $\psi$-exceptional divisor on $T$, 
\item \label{proof-thm--ind-1-step3-(b)}
an effective $\mathbb{R}$-divisor $\Psi$ on $T$ defined by $\Psi=-\sum_{\substack {D}}\bigl(a(D,S,\Delta_{S}+\boldsymbol{\rm N}_{S})-1\bigr) D$, where $D$ runs over all prime divisors on $T$, 
\item \label{proof-thm--ind-1-step3-(c)}
the inequality $a(Q,S,(\Delta_{S}+t_{i}H_{S})+\boldsymbol{\rm N}_{S})\leq a(Q,S_{i},(\Delta_{S_{i}}+t_{i}H_{S_{i}})+\boldsymbol{\rm N}_{S_{i}})$ for all $i$ and all prime divisors $Q$ over $S$, and 
\item \label{proof-thm--ind-1-step3-(d)}
the inequality $a(Q',S_{m},\Delta_{S_{m}}+\boldsymbol{\rm N}_{S_{m}})\leq a(Q',T,\Psi+\boldsymbol{\rm N}_{T})$ for all prime divisors $Q'$ over $S_{m}$, in particular, the generalized pair $(T,\Psi, \boldsymbol{\rm N})$ is generalized lc. 
\end{enumerate}

The generalized pair $(X_{i}, \Delta_{i}+t_{i} H_{i},\boldsymbol{\rm M})$ is a good minimal model of $(X, \Delta+t_{i} H, \boldsymbol{\rm M})$, hence the negativity lemma (see also \cite[Proof of Lemma 4.2.10]{fujino-sp-ter}) implies the inequality
\begin{equation*}
\begin{split}
 a(Q,S,(\Delta_{S}+t_{i}H_{S})+\boldsymbol{\rm N}_{S}) \leq a(Q,S_{i},(\Delta_{S_{i}}+t_{i}H_{S_{i}})+\boldsymbol{\rm N}_{S_{i}})  
\end{split}
\end{equation*}
for all prime divisors $Q$ over $S$. 
We have (\ref{proof-thm--ind-1-step3-(c)}). 

In this paragraph, we prove that the equality
\begin{equation*}\tag{$\star$}\label{proof-thm--ind-1-(starstar)}
a(\widetilde{D},S_{m},\Delta_{S_{m}}+\boldsymbol{\rm N}_{S_{m}})= a(\widetilde{D},S,\Delta_{S}+\boldsymbol{\rm N}_{S})
\end{equation*} 
holds for all prime divisors $\widetilde{D}$ on $S_{m}$. 
Note that $\widetilde{D}$ is a prime divisor on $S_{i}$ for every $i$ since the birational map $S_{m}\dashrightarrow S_{i}$ is small. 
For every $i$, 
we may apply Lemma \ref{lem--discre-relation} to the birational map $(X,\Delta,\boldsymbol{\rm M})\dashrightarrow(X_{i},\Delta_{i},\boldsymbol{\rm M})$ and the generalized pairs $(S,\Delta_{S},\boldsymbol{\rm N})$ and $(S_{i},\Delta_{S_{i}},\boldsymbol{\rm N})$ because $(X,\Delta,\boldsymbol{\rm M})\dashrightarrow(X_{i},\Delta_{i},\boldsymbol{\rm M})$ obviously satisfies the first condition of Lemma \ref{lem--discre-relation} and the second condition of Lemma \ref{lem--discre-relation} follows from the third condition of Theorem \ref{thm--mmp-induction-2}. 
Thus, we have
$$a(\widetilde{D},S_{i},\Delta_{S_{i}}+\boldsymbol{\rm N}_{S_{i}}) \leq a(\widetilde{D},S,\Delta_{S}+\boldsymbol{\rm N}_{S}),$$ and therefore we have 
$$a(\widetilde{D},S_{i},(\Delta_{S_{i}}+t_{i}H_{S_{i}})+\boldsymbol{\rm N}_{S_{i}}) \leq a(\widetilde{D},S,\Delta_{S}+\boldsymbol{\rm N}_{S})$$
 for all prime divisors $\widetilde{D}$ on $S_{m}$.  
By this inequality and (\ref{proof-thm--ind-1-step3-(c)}), we have
\begin{equation*}a(\widetilde{D},S,(\Delta_{S}+t_{i}H_{S})+\boldsymbol{\rm N}_{S})\leq a(\widetilde{D},S_{i},(\Delta_{S_{i}}+t_{i}H_{S_{i}})+\boldsymbol{\rm N}_{S_{i}})  \leq a(\widetilde{D},S,\Delta_{S}+\boldsymbol{\rm N}_{S}).\end{equation*} 
We have $a(\widetilde{D},S_{i},(\Delta_{S_{i}}+t_{i}H_{S_{i}})+\boldsymbol{\rm N}_{S_{i}})=a(\widetilde{D},S_{m},(\Delta_{S_{m}}+t_{i}H_{S_{m}})+\boldsymbol{\rm N}_{S_{m}})$ for all $i$ since $\widetilde{D}$ appears as a prime divisor on $S_{m}$ and $S_{i}$. 
Since ${\rm lim}_{i \to \infty}t_{i}=0$, we get (\ref{proof-thm--ind-1-(starstar)})  for all prime divisors $\widetilde{D}$ on $S_{m}$  by taking the limit $i\to \infty$. 

We put
\begin{equation*}
\mathcal{C}=\Set{ \bar{D} | \begin{array}{l}\!\!\text{$\bar{D}$ is a prime divisor on $S$ such that}\\ \text{$a(\bar{D},S_{m}, \Delta_{S_{m}}+\boldsymbol{\rm N}_{S_{m}})<a(\bar{D},S,\Delta_{S}+\boldsymbol{\rm N}_{S})$} \end{array}\!\!}. \end{equation*} 
By (\ref{proof-thm--ind-1-(starstar)}), all elements of $\mathcal{C}$ are exceptional over $S_{m}$. 
By a basic property of generalized log discrepancies and (\ref{proof-thm--ind-1-step3-(c)}), we have $$a(\bar{D},S, (\Delta_{S}+t_{m}H_{S})+\boldsymbol{\rm N}_{S}) \leq a(\bar{D},S_{m}, \Delta_{S_{m}}+\boldsymbol{\rm N}_{S_{m}}).$$ 
This implies
$$a(\bar{D},S,(\Delta_{S}+t_{m}H_{S})+\boldsymbol{\rm N}_{S}) \leq a(\bar{D},S_{m}, \Delta_{S_{m}}+\boldsymbol{\rm N}_{S_{m}})< a(\bar{D},S,\Delta_{S}+\boldsymbol{\rm N}_{S})$$
for all $\bar{D}\in \mathcal{C}$. 
Since every element of $\mathcal{C}$ is a prime divisor on $S$, we see that all elements of $\mathcal{C}$ are components of $H_{S}$. 
Thus, $ \mathcal{C}$ is a finite set, and $a(\bar{D},S_{m}, \Delta_{S_{m}}+\boldsymbol{\rm N}_{S_{m}})<1$ for all $\bar{D}\in \mathcal{C}$ because
$a(\bar{D},S,\Delta_{S}+\boldsymbol{\rm N}_{S})\leq 1$. 
From these facts, for every $\bar{D}\in \mathcal{C}$ we get the relation
\begin{equation*}
\begin{split}
0\leq &a(\bar{D},S,(\Delta_{S}+t_{0}H_{S})+\boldsymbol{\rm N}_{S})\\
<&a(\bar{D},S,(\Delta_{S}+t_{m}H_{S})+\boldsymbol{\rm N}_{S}) \qquad\;\; (\text{$\bar{D}$ is a component of $H_{S}$ and $m\gg 0$})\\
\leq &a(\bar{D},S_{m},(\Delta_{S_{m}}+t_{m}H_{S_{m}})+\boldsymbol{\rm N}_{S_{m}}) \qquad (\text{the inequality (\ref{proof-thm--ind-1-step3-(c)})})\\
\leq & a(\bar{D},S_{m}, \Delta_{S_{m}}+\boldsymbol{\rm N}_{S_{m}}) \\
<& a(\bar{D},S,\Delta_{S}+\boldsymbol{\rm N}_{S}) \leq 1\qquad\qquad\qquad\;\;\; (\text{definition of $\mathcal{C}$}).
\end{split}
\end{equation*} 
In this way, we have $0<a(\bar{D},S_{m}, \Delta_{S_{m}}+\boldsymbol{\rm N}_{S_{m}})<1$.  
By Lemma \ref{lem--extraction-gpair}, there exists a projective birational morphism $\psi\colon T\to S_{m}$ such that $T$ is $\mathbb{Q}$-factorial and $\psi^{-1}$ exactly extracts elements of $\mathcal{C}$. 
We have constructed the desired birational morphism as in (\ref{proof-thm--ind-1-step3-(a)}). 

Let $D$ be a prime divisor on $T$. 
When $D$ is $\psi$-exceptional, by the definitions of $\mathcal{C}$ and $\psi$ we have 
$$a(D,S_{m}, \Delta_{S_{m}}+\boldsymbol{\rm N}_{S_{m}})<a(D,S,\Delta_{S}+\boldsymbol{\rm N}_{S})\leq 1.$$ 
When $D$ is not $\psi$-exceptional, from (\ref{proof-thm--ind-1-(starstar)}) we see that 
$$a(D,S,\Delta_{S}+\boldsymbol{\rm N}_{S})= a(D,S_{m},\Delta_{S_{m}}+\boldsymbol{\rm N}_{S_{m}})\leq 1.$$  
Thus, the relation
\begin{equation*}\tag{$\star$$\star$}\label{proof-thm--ind-1-(starstarstar)}
a(D,S_{m}, \Delta_{S_{m}}+\boldsymbol{\rm N}_{S_{m}})\leq a(D,S,\Delta_{S}+\boldsymbol{\rm N}_{S})\leq 1.
\end{equation*} 
holds for all prime divisors $D$ on $T$. 
Note that only finitely many prime divisors $D$ on $T$ satisfy $a(D,S,\Delta_{S}+\boldsymbol{\rm N}_{S})<1$. 
Therefore, we may define an $\mathbb{R}$-divisor $\Psi\geq0$ on $T$ by
$$\Psi=-\sum_{\substack {D}}\bigl(a(D,S,\Delta_{S}+\boldsymbol{\rm N}_{S})-1\bigr) D,$$
where $D$ runs over all prime divisors on $T$. 
This is the divisor stated in (\ref{proof-thm--ind-1-step3-(b)}). 

Finally, we prove (\ref{proof-thm--ind-1-step3-(d)}). 
Since $T$ is $\mathbb{Q}$-factorial, $K_{T}+\Psi+\boldsymbol{\rm N}_{T}$ is $\mathbb{R}$-Cartier. 
By (\ref{proof-thm--ind-1-(starstarstar)}), we obtain 
$$K_{T}+\Psi +\boldsymbol{\rm N}_{T}\leq \psi^{*}(K_{S_{m}}+\Delta_{S_{m}}+\boldsymbol{\rm N}_{S_{m}}).$$ 
From this, we have
\begin{equation*}
0 \leq a(Q',S_{m},\Delta_{S_{m}}+\boldsymbol{\rm N}_{S_{m}})\leq a(Q',T,\Psi+\boldsymbol{\rm N}_{T}) 
\end{equation*}
for any prime divisor $Q'$ over $S_{m}$. 
This shows (\ref{proof-thm--ind-1-step3-(d)}). 
\end{step1}

\begin{step1}\label{step5-mmp-induction-2}
The goal of this step is to show that $(T, \Psi,\boldsymbol{\rm N})$ has a minimal model. 
We follow \cite[Step 3 in the proof of Theorem 3.4]{has-finite}. 

By the second condition of Theorem \ref{thm--mmp-induction-2}, the generalized pair $(S,\Delta_{S},\boldsymbol{\rm N})$ is a minimal model of $(S,\Delta_{S},\boldsymbol{\rm N})$ itself. 
By applying \cite[Step 3 in the proof of Theorem 3.4]{has-finite} with minor changes, we see that the following property holds.
\begin{itemize}
\item Let $\widetilde{Q}$ be a prime divisor over $S$. Then the following two statements hold:
\begin{itemize} 
\item 
If $\widetilde{Q}$ is a divisor on $S$, then $a(\widetilde{Q},S,\Delta_{S}+\boldsymbol{\rm N}_{S})\leq a(\widetilde{Q},T,\Psi+\boldsymbol{\rm N}_{T})$, and 
\item 
if $\widetilde{Q}$ is a divisor on $T$, then $a(\widetilde{Q},T,\Psi+\boldsymbol{\rm N}_{T})\leq a(\widetilde{Q},S,\Delta_{S}+\boldsymbol{\rm N}_{S})$. 
\end{itemize} 
\end{itemize}
By applying Lemma \ref{lem--bir-relation} to the map $(T, \Psi,\boldsymbol{\rm N})\dashrightarrow (S,\Delta_{S},\boldsymbol{\rm N})$, we see that $(T, \Psi,\boldsymbol{\rm N})$ has a minimal model. 
\end{step1}

\begin{step1}\label{step6-mmp-induction-2}
We follow \cite[Step 4 in the proof of Theorem 3.4]{has-finite}. 
With this step we complete the proof of the special termination of the $(K_{X_{0}}+\Delta_{0}+\boldsymbol{\rm M}_{X_{0}})$-MMP in Step \ref{step3-mmp-induction-2}. 

We recall that $S_{m}\dashrightarrow S_{i}$ is small, $\Delta_{S_{i}}+t_{i}H_{S_{i}}$ is equal to the birational transform of $\Delta_{S_{m}}+t_{i}H_{S_{m}}$ on $S_{i}$, and the divisor 
$$K_{S_{i}}+\Delta_{S_{i}}+t_{i}H_{S_{i}}+\boldsymbol{\rm N}_{S_{i}} \sim_{\mathbb{R}} (K_{X_{i}}+\Delta_{i}+t_{i}H_{i}+\boldsymbol{\rm M}_{X_{i}})|_{S_{i}}$$
 is nef. 
So $(S_{i},\Delta_{S_{i}}+t_{i}H_{S_{i}}, \boldsymbol{\rm N})$ is a weak generalized lc model of $(S_{m},\Delta_{S_{m}}+t_{i}H_{S_{m}}, \boldsymbol{\rm N})$ for every $i$. 

We pick an arbitrary prime divisor $D$ on $T$. 
By using the generalized pair analogue of \cite[Remark 2.9 (1)]{has-finite}, we see that 
\begin{equation*}\begin{split}
&\sigma_{D}(K_{S_{m}}+\Delta_{S_{m}}+t_{i}H_{S_{m}}+\boldsymbol{\rm N}_{S_{m}})\\
=&a(D,S_{i},(\Delta_{S_{i}}+t_{i}H_{S_{i}})+\boldsymbol{\rm N}_{S_{i}})-a(D, S_{m},( \Delta_{S_{m}}+t_{i}H_{S_{m}})+\boldsymbol{\rm N}_{S_{m}}).\end{split}\end{equation*}
From the relation, we have
\begin{equation*}\begin{split}
&\sigma_{D}(K_{S_{m}}+\Delta_{S_{m}}+t_{i}H_{S_{m}}+\boldsymbol{\rm N}_{S_{m}})\\
=&a(D,S_{i},(\Delta_{S_{i}}+t_{i}H_{S_{i}})+\boldsymbol{\rm N}_{S_{i}})-a(D, S_{m}, (\Delta_{S_{m}}+t_{i}H_{S_{m}})+\boldsymbol{\rm N}_{S_{m}})\\
\geq &a(D,S,(\Delta_{S}+t_{i}H_{S})+\boldsymbol{\rm N}_{S})-a(D, S_{m}, (\Delta_{S_{m}}+t_{i}H_{S_{m}})+\boldsymbol{\rm N}_{S_{m}}) \qquad (\text{(\ref{proof-thm--ind-1-step3-(c)}) in Step \ref{step4-mmp-induction-2}}).
\end{split}\end{equation*}
By (\ref{proof-thm--ind-1-step3-(b)}) in Step \ref{step4-mmp-induction-2}, the equality $a(D,S,\Delta_{S}+\boldsymbol{\rm N}_{S})=a(D,T,\Psi+\boldsymbol{\rm N}_{T})$ holds. 
From these relations and the fact that ${\rm lim}_{i \to \infty}t_{i}=0$, we obtain  
\begin{equation*}
\begin{split}
&\sigma_{D}(K_{S_{m}}+\Delta_{S_{m}}+\boldsymbol{\rm N}_{S_{m}})\\
=&\underset{i\to \infty}{\rm lim} \sigma_{D}(K_{S_{m}}+\Delta_{S_{m}}+t_{i}H_{S_{m}}+\boldsymbol{\rm N}_{S_{m}}) \qquad \qquad \qquad (\text{\cite[Remark 2.3 (2)]{has-finite}})\\
\geq &\underset{i\to \infty}{\rm lim}\bigl(a(D,S,(\Delta_{S}+t_{i}H_{S})+\boldsymbol{\rm N}_{S})-a(D, S_{m}, (\Delta_{S_{m}}+t_{i}H_{S_{m}})+\boldsymbol{\rm N}_{S_{m}})\bigr)\\
=&a(D,S,\Delta_{S}+\boldsymbol{\rm N}_{S})-a(D, S_{m}, \Delta_{S_{m}}+\boldsymbol{\rm N}_{S_{m}})\\
=&a(D,T,\Psi+\boldsymbol{\rm N}_{T})-a(D, S_{m}, \Delta_{S_{m}}+\boldsymbol{\rm N}_{S_{m}})
\end{split}
\end{equation*}
for all prime divisors $D$ on $T$. 
By this relation and (\ref{proof-thm--ind-1-step3-(d)}) in Step \ref{step4-mmp-induction-2},  we obtain 
 \begin{equation*}
\begin{split}
&0 \leq  a(D,T,\Psi+\boldsymbol{\rm N}_{T}) - a(D, S_{m}, \Delta_{S_{m}}+\boldsymbol{\rm N}_{S_{m}}) \leq \sigma_{D}(K_{S_{m}}+\Delta_{S_{m}}+\boldsymbol{\rm N}_{S_{m}}).
\end{split}
\end{equation*}
Therefore, we may apply Lemma \ref{lem--minmodel-zariskidecom} to $(S_{m},\Delta_{S_{m}},\boldsymbol{\rm N})$ and $(T, \Psi,\boldsymbol{\rm N})$. 
Because $(T, \Psi,\boldsymbol{\rm N})$ has a minimal model, $(S_{m},\Delta_{S_{m}},\boldsymbol{\rm N})$ has a  minimal model. 

By arguments in the last paragraph of Step \ref{step3-mmp-induction-2}, we complete the special termination. 
\end{step1}
Since the $(K_{X_{0}}+\Delta_{0}+\boldsymbol{\rm M}_{X_{0}})$-MMP occurs only in ${\rm Supp}\lfloor \Delta_{0}\rfloor$ (see Step \ref{step2-mmp-induction-2}), we see that the $(K_{X_{0}}+\Delta_{0}+\boldsymbol{\rm M}_{X_{0}})$-MMP terminates, so $(X,\Delta,\boldsymbol{\rm M})$ has a  minimal model. 
\end{proof}

\begin{thm}[cf.~{\cite[Theorem 3.5]{has-finite}}]\label{thm--mmp-main-1}
Let $(X,\Delta, \boldsymbol{\rm M})$ be a generalized lc pair such that $\boldsymbol{\rm M}$ is a finite $\mathbb{R}_{>0}$-linear combination of b-nef $\mathbb{Q}$-b-Cartier $\mathbb{Q}$-b-divisors. 
Let $A$ be an $\mathbb{R}$-divisor on $X$ such that $K_{X}+\Delta+\boldsymbol{\rm M}_{X}+A$ is nef. 

Then, no infinite sequence of steps of a $(K_{X}+\Delta+\boldsymbol{\rm M}_{X})$-MMP with scaling of $A$
$$(X,\Delta, \boldsymbol{\rm M})=:(X_{0},\Delta_{0},\boldsymbol{\rm M}) \dashrightarrow (X_{1}, \Delta_{1},\boldsymbol{\rm M}) \dashrightarrow \cdots \dashrightarrow (X_{i}, \Delta_{i},\boldsymbol{\rm M}) \dashrightarrow \cdots$$
satisfies the following properties. 
 \begin{itemize}
\item
If we define $\lambda_{i}={\rm inf}\set{\mu \in \mathbb{R}_{\geq0}|\text{$K_{X_{i}}+\Delta_{i}+\boldsymbol{\rm M}_{X_{i}}+\mu A_{i}$ is nef}}$ for each $i\geq 0$ then $\underset{i \to \infty}{\rm lim}\lambda_{i}=0$, and
\item
there are infinitely many $i$ such that $K_{X_{i}}+\Delta_{i}+\boldsymbol{\rm M}_{X_{i}}$ is log abundant with respect to $(X_{i},\Delta_{i}, \boldsymbol{\rm M})$.  
\end{itemize}
\end{thm}

\begin{proof}
The argument of \cite[Proof of Theorem 3.5]{has-finite} works without any change by using Lemma \ref{lem--mmpscaling-nonneflocus}, the special termination for generalized dlt pairs (\cite[Subsection 4.2]{hanli}), and Theorem \ref{thm--mmp-induction-2} instead of \cite[Lemma 2.22]{has-finite}, \cite{fujino-sp-ter}, and \cite[Theorem 3.5]{has-finite} respectively. 
\end{proof}

\begin{thm}\label{thm--mmp-neflogbig}
Let $(X,\Delta)$ be a projective lc pair, and let $M$ be a finite $\mathbb{R}_{>0}$-linear combination of nef $\mathbb{Q}$-divisors on $X$.  Suppose that $K_{X}+\Delta+M$ is pseudo-effective and $M$ is log big with respect to $(X,\Delta)$. 
We put $\boldsymbol{\rm M}=\overline{M}$. 
Then there is a sequence of steps of a $(K_{X}+\Delta+M)$-MMP starting with the generalized lc pair $(X,\Delta, \boldsymbol{\rm M})$
$$(X,\Delta, \boldsymbol{\rm M}) \dashrightarrow \cdots \dashrightarrow (X_{i},\Delta_{i},\boldsymbol{\rm M}) \dashrightarrow \cdots$$
that terminates with a good minimal model. 
\end{thm}

\begin{proof}
We first reduce the theorem to the case where $K_{X}+\Delta+lM$ is nef for some positive real number $l >2$. 
We may write $M=\sum_{j} r_{j}M^{(j)}$, where $M^{(j)}$ are nef Cartier divisors on $X$ and $r_{j}>0$. 
We fix a real number $l>2$ such that $lr_{j}> 2\cdot {\rm dim}X$ for all $j$. 
Pick an ample $\mathbb{R}$-divisor $A$ on $X$ such that $K_{X}+\Delta+lM+A$ is nef. 
Since $M$ is nef, by \cite[Theorem 1.5]{hashizumehu}, the generalized lc pair $(X,\Delta, l\boldsymbol{\rm M}+t\overline{A})$ has a good minimal model for every real number $t>0$. 
Applying the argument in \cite[Proof of Lemma 2.14]{has-mmp} (see \cite[Proof of Theorem 1.7]{hashizumehu}) with the aid of the length of
extremal rays (\cite[Section 18]{fujino-fund} or \cite[Theorem 4.6.2]{fujino-book}), we may construct a $(K_{X}+\Delta+lM)$-MMP with scaling of $A$
$$(X,\Delta,l\boldsymbol{\rm M}) \dashrightarrow \cdots \dashrightarrow (X_{i},\Delta_{i},l\boldsymbol{\rm M}) \dashrightarrow \cdots$$
such that if we define 
$$\lambda_{i}={\rm inf}\set{\mu \in \mathbb{R}_{\geq0}|\text{$K_{X_{i}}+\Delta_{i}+l\boldsymbol{\rm M}_{X_{i}}+\mu A_{i}$ is nef}}$$
 for each $i$, then ${\rm lim}_{i \to \infty}\lambda_{i}=0$. 
By the choice of $l$ and the length of extremal rays (see \cite[Section 18]{fujino-fund} or \cite[Theorem 4.6.2]{fujino-book}), the birational transforms $M^{(j)}_{i}$ of $M^{(j)}$ on $X_{i}$ are numerically trivial with respect to the extremal contraction of the MMP. 
This fact implies the following properties:
\begin{itemize}
\item
$M^{(j)}_{i}$ is nef and Cartier for every $i$ and $j$, and 
\item
the sequence of birational maps $X \dashrightarrow \cdots \dashrightarrow X_{i} \dashrightarrow \cdots$ is a sequence of steps of a $(K_{X}+\Delta+M)$-MMP. 
\end{itemize}
By the first property, $\boldsymbol{\rm M}_{X_{i}}=\sum_{j} r_{j}M^{(j)}_{i}$ is a finite $\mathbb{R}_{>0}$-linear combination of nef Cartier divisors on $X_{i}$, and the log bigness of $M$ shows that $\boldsymbol{\rm M}_{X_{i}}$ is log big with respect to $(X_{i},\Delta_{i})$ for every $i$. 
By Lemma \ref{lem--lcpair-neflogbig}, we see that $K_{X_{i}}+\Delta_{i}+l\boldsymbol{\rm M}_{X_{i}}$ is log big with respect to $(X_{i},\Delta_{i})$ for every $i$. 
By Theorem \ref{thm--mmp-main-1}, the $(K_{X}+\Delta+lM)$-MMP terminates with a minimal model $(X_{m},\Delta_{m},l\boldsymbol{\rm M})$. 
Then the birational map 
$$X \dashrightarrow X_{m}$$
 is a sequence of steps of a $(K_{X}+\Delta+M)$-MMP. 
To prove Theorem \ref{thm--mmp-neflogbig}, it is sufficient to prove the existence of a $(K_{X_{m}}+\Delta_{m}+\boldsymbol{\rm M}_{X_{m}})$-MMP that terminates with a good minimal model. 
In this way, replacing $(X,\Delta,\boldsymbol{\rm M})$ with $(X_{m},\Delta_{m},\boldsymbol{\rm M})$, we may assume that $K_{X}+\Delta+lM$ is nef for some $l >2$. 

Replacing $l$ by $2l$, we may assume that $K_{X}+\Delta+lM$ is log big with respect to $(X,\Delta)$. 
By the argument of Shokurov polytope (cf.~\cite[Subsection 3.3]{hanli}) and \cite{fujino-abund-logbig}, we see that $K_{X}+\Delta+lM$ is semi-ample. 
We fix an effective $\mathbb{R}$-divisor 
$$B\sim_{\mathbb{R}}K_{X}+\Delta+lM$$
 such that $(X,\Delta+B)$ is lc. 
Then 
$$l(K_{X}+\Delta+M)=(l-1)(K_{X}+\Delta)+(K_{X}+\Delta+lM)\sim_{\mathbb{R}}(l-1)\left(K_{X}+\Delta+\frac{1}{l-1}B\right).$$

By Corollary \ref{cor--abundant-mmp}, $K_{X}+\Delta+\tfrac{1}{l-1}B$ is log abundant with respect to $(X,\Delta+\tfrac{1}{l-1}B)$ and every $(K_{X}+\Delta+\tfrac{1}{l-1}B)$-MMP preserves the property of being log abundant. 
Then the argument in \cite[Proof of Corollary 3.9]{has-finite} implies the existence of a sequence of steps of a $(K_{X}+\Delta+\tfrac{1}{l-1}B)$-MMP that terminates with a good minimal model. 

Since all $(K_{X}+\Delta+\tfrac{1}{l-1}B)$-MMP are $(K_{X}+\Delta+M)$-MMP, there exists a sequence of steps of a $(K_{X}+\Delta+M)$-MMP that terminates with a good minimal model. 
In this way, we see that Theorem \ref{thm--mmp-neflogbig} holds. 
\end{proof}

\begin{proof}[Proof of Theorem \ref{thm--mmp-neflogbig-main-intro}]
If $K_{X}+\Delta+M$ is pseudo-effective, then Theorem \ref{thm--mmp-neflogbig-main-intro} follows from Theorem \ref{thm--mmp-neflogbig}.
If $K_{X}+\Delta+M$ is not pseudo-effective, then take an ample $\mathbb{R}$-divisor $A$ on $X$ and a real number $\epsilon>0$ such that $K_{X}+\Delta+M+A$ is nef and $K_{X}+\Delta+M+\epsilon A$ is not pseudo-effective. 
Since $M+\epsilon A$ is ample, by \cite[Theorem 1.7]{hashizumehu} there is $X\dashrightarrow X'$ a sequence of steps of a $(K_{X}+\Delta+\epsilon A+M)$-MMP with scaling of $A$ that terminates with a Mori fiber space. 
Then $X\dashrightarrow X'$ is a sequence of steps of a $(K_{X}+\Delta+M)$-MMP with scaling of $A$ terminating with a Mori fiber space. 
\end{proof}

The following result was mentioned in \cite{has-nonvan-gpair}. 

\begin{thm}\label{thm--mmp-polarized-gpair}
Let $(X,B,\boldsymbol{\rm M})$ be a generalized lc pair such that $\boldsymbol{\rm M}$ is a finite $\mathbb{R}_{>0}$-linear combination of b-nef $\mathbb{Q}$-b-Cartier $\mathbb{Q}$-b-divisors. 
Let $A$ be an effective ample $\mathbb{R}$-divisor on $X$ such that $(X,B+A, \boldsymbol{\rm M})$ is a generalized lc pair and generalized lc centers of $(X,B+A, \boldsymbol{\rm M})$ are generalized lc centers of $(X,B, \boldsymbol{\rm M})$. 
Let $(Y,\Gamma,\boldsymbol{\rm M})$ be a $\mathbb{Q}$-factorial generalized dlt model of $(X,B+A, \boldsymbol{\rm M})$. 
Then, every $(K_{Y}+\Gamma+\boldsymbol{\rm M}_{Y})$-MMP with scaling of an ample divisor
$$(Y,\Gamma,\boldsymbol{\rm M})=:(Y_{0},\Gamma_{0},\boldsymbol{\rm M}) \dashrightarrow \cdots \dashrightarrow (Y_{i},\Gamma_{i},\boldsymbol{\rm M}) \dashrightarrow \cdots$$
terminates with a minimal model. 
\end{thm}

\begin{proof}
By \cite[Theorem 1.3]{has-nonvan-gpair}, we see that $K_{Y_{i}}+\Gamma_{i}+\boldsymbol{\rm M}_{Y_{i}}$ is log abundant with respect to $(Y_{i},\Gamma_{i},\boldsymbol{\rm M})$ for all $i$. 
Hence the theorem follows from Theorem \ref{thm--mmp-main-1}. 
\end{proof}

\section{Effectivity of Iitaka fibration}\label{sec4}

In this section, we prove Theorem \ref{thm--eff-iitaka-intro} and Theorem \ref{thm--base-iitaka-intro}.

\begin{thm}\label{thm--eff-dlt-trivial}
Let $d$ and $p$ be positive integers, and let $\Phi \subset \mathbb{Q}$ be a DCC set.  
Then, there exists a positive integer $m$, depending only on $d$, $p$, and $\Phi$, satisfying the following. 
Let $(X,\Delta,\boldsymbol{\rm M})$ be a generalized dlt pair such that
\begin{itemize}
\item
${\rm dim}X=d$, 
\item
the coefficients of $\Delta$ belong to $\Phi$, 
\item
$p \boldsymbol{\rm M}$ is b-Cartier, 
\item
$K_{X}+\Delta+\boldsymbol{\rm M}_{X}\equiv 0$, and
\item
there is a log resolution $f\colon \tilde{X}\to X$ of $(X,\Delta)$ such that 
\begin{itemize}
\item
$f$ is an isomorphism over an open subset $U \subset X$ containing all the generic points of generalized lc centers of $(X,\Delta,\boldsymbol{\rm M})$, 
\item
$\boldsymbol{\rm M}$ descends to $\tilde{X}$, and 
\item
writing
$K_{\tilde{X}}+\tilde{\Delta}+\boldsymbol{\rm M}_{\tilde{X}}=f^{*}(K_{X}+\Delta+\boldsymbol{\rm M}_{X})+\tilde{E}$,
where $\tilde{\Delta}\geq0$ and $\tilde{E}\geq0$ have no common components, then $\boldsymbol{\rm M}_{\tilde{X}}$ is log big with respect to $(\tilde{X},\tilde{\Delta})$. 
\end{itemize}
\end{itemize}
Then $H^{0}(X,\mathcal{O}_{X}(\lfloor m (K_{X}+\Delta+\boldsymbol{\rm M}_{X})\rfloor))\neq 0$. 
In particular, $m (K_{X}+\Delta+\boldsymbol{\rm M}_{X}) \sim 0$. 
\end{thm}

\begin{proof}
We prove the theorem by induction on the dimension of $X$. 

By the global ACC \cite[Theorem 1.6]{bz}, we can find a finite set $\Phi' \subset \Phi$, depending only on $d$, $p$, and $\Phi$, such that the coefficients of $\Delta$ belong to $\Phi'$. 
By \cite[Theorem 1.10]{birkar-compl}, there exists a positive integer $m'$, depending only on $d$, $p$, and $\Phi'$, such that if $(X,\Delta,\boldsymbol{\rm M})$ is a generalized klt pair then $H^{0}(X,\mathcal{O}_{X}(\lfloor m' (K_{X}+\Delta+\boldsymbol{\rm M}_{X})\rfloor))\neq 0$. 
Let $s$ be the minimum positive integer such that $sa \in \mathbb{Z}$ for every $a \in \Phi'$. 
If $(X,\Delta,\boldsymbol{\rm M})$ is not generalized klt, then we may find a component $S$ of $\lfloor \Delta \rfloor$. 
Since $sp \boldsymbol{\rm M}_{\tilde{X}}$ is Cartier and $spa \in \mathbb{Z}$ for every $a \in \Phi'$, by the projective case of Lemma \ref{lem--extension-lc-center} the morphism
\begin{equation*}
\begin{split}
H^{0}(X,\mathcal{O}_{X}(\lfloor lsp(K_{X}+\Delta+\boldsymbol{\rm M}_{X})\rfloor)) \longrightarrow H^{0}(S,\mathcal{O}_{S}(\lfloor lsp(K_{S}+\Delta_{S}+\boldsymbol{\rm N}_{S})\rfloor))
\end{split}
\end{equation*}
is surjective for every positive integer $l$, where  the divisor $K_{S}+\Delta_{S}+\boldsymbol{\rm N}_{S}$ is the generalized log canonical divisor of the generalized pair $(S,\Delta_{S},\boldsymbol{\rm N})$ defined with divisorial adjunction for $(X,\Delta,\boldsymbol{\rm M})$. 
By Remark \ref{rem--coeff-adj}, there exists a DCC set $\Omega \subset \mathbb{Q}$,  depending only on $d$, $p$, and $\Phi'$, such that all the coefficients of $\Delta_{S}$ belong to $\Omega$. 
By the induction hypothesis of Theorem \ref{thm--eff-dlt-trivial}, there is a positive integer $m''$, depending only on $d-1$, $p$, and $\Omega$, such that $H^{0}(S,\mathcal{O}_{S}(\lfloor m'' (K_{S}+\Delta_{S}+\boldsymbol{\rm N}_{S})\rfloor))\neq 0$. 
Then $H^{0}(X,\mathcal{O}_{X}(\lfloor m''sp(K_{X}+\Delta+\boldsymbol{\rm M}_{X})\rfloor)) \neq 0$ by the above surjection. 
Hence $m:=m'm''sp$ is the desired positive integer.
\end{proof}

We will use the following result about the generalized canonical bundle formula by Filipazzi--Moraga \cite{filipazzimoraga}.

\begin{lem}[cf.~{\cite[Theorem 1.5]{filipazzimoraga}}, see also {\cite[Proposition 6.3]{birkar-compl}}]\label{lem--klt-canbundleformula}
Let $d$ and $p$ be positive integers. 
Let $\mathfrak{R} \subset [0,1]$ be a finite set of rational numbers and $\Phi(\mathfrak{R})$ the hyperstandard set associated to $\mathfrak{R}$.  
Then, there exist positive integers $n$ and $q$ and a DCC set $\Omega \subset \mathbb{Q}$, depending only on $d$, $p$, and $\mathfrak{R}$, satisfying the following. 
Let $\pi \colon X \to Z$ be a contraction of normal projective varieties and let $(X,\Delta,\boldsymbol{\rm M})$ be a generalized klt pair such that
\begin{itemize}
\item
${\rm dim}X=d$, 
\item
the coefficients of $\Delta$ belong to $\Phi(\mathfrak{R})$, 
\item 
$p\boldsymbol{\rm M}$ is b-Cartier, 
\item
$K_{X}+\Delta+\boldsymbol{\rm M}_{X}\sim_{\mathbb{Q}, Z}0$, and
\item
$X$ is of Fano type over an open subset $U$ of $Z$,  
\end{itemize} 
Then $n(K_{X}+\Delta+\boldsymbol{\rm M}_{X})\sim \pi^{*}D$ for some $\mathbb{Q}$-Cartier divisor $D$ on $Z$ and there exists a generalized lc pair $(Z,\Delta_{Z},\boldsymbol{\rm N})$ constructed with generalized canonical bundle formula such that 
\begin{itemize}
\item
$K_{Z}+\Delta_{Z}+\boldsymbol{\rm N}_{Z}=\tfrac{1}{n}D$, 
\item
the coefficients of $\Delta_{Z}$ belong to $\Omega$, and 
\item
$q \boldsymbol{\rm N}$ is b-Cartier. 
\end{itemize}
\end{lem}

\begin{rem}
The existence of $n$ satisfying $n(K_{X}+\Delta+\boldsymbol{\rm M}_{X})\sim n \pi^{*}(K_{Z}+\Delta_{Z}+\boldsymbol{\rm N}_{Z})$ in Lemma \ref{lem--klt-canbundleformula} was proved in \cite[Proof of Theorem 1.5]{filipazzimoraga} (see \cite[Proof of Lemma 5.4]{filipazzimoraga} or \cite[Proof of Proposition 6.3]{birkar-compl}). 
If ${\rm dim}Z=0$, then the existence of $n$ is equivalent to the existence of $n$-complements, which was proved by Birkar \cite[Theorem 1.10]{birkar-compl}. 
\end{rem}

\begin{thm}\label{thm--eff-vertical}
Let $d$ and $p$ be positive integers, and let $\Phi \subset \mathbb{Q}$ be a DCC set.  
Then, there exist positive integers $n$ and $q$ and a DCC set $\Omega \subset \mathbb{Q}$, depending only on $d$, $p$, and $\Phi$, satisfying the following. 
Let $\pi \colon X \to Z$ be a contraction of normal projective varieties and let $(X,\Delta,\boldsymbol{\rm M})$ be a generalized lc pair such that
\begin{itemize}
\item
${\rm dim}X=d$, 
\item
the coefficients of $\Delta$ belong to $\Phi$, 
\item 
$\boldsymbol{\rm M}$ is b-big$/Z$ and $p\boldsymbol{\rm M}$ is b-Cartier, 
\item
$K_{X}+\Delta+\boldsymbol{\rm M}_{X}\sim_{\mathbb{Q}, Z}0$, and
\item
all generalized lc centers of $(X,\Delta,\boldsymbol{\rm M})$ are vertical over $Z$. 
\end{itemize} 
Then $n(K_{X}+\Delta+\boldsymbol{\rm M}_{X})\sim \pi^{*}D$ for some $\mathbb{Q}$-Cartier divisor $D$ on $Z$ and there exists a generalized lc pair $(Z,\Delta_{Z},\boldsymbol{\rm N})$ constructed with generalized canonical bundle formula such that 
\begin{itemize}
\item
$K_{Z}+\Delta_{Z}+\boldsymbol{\rm N}_{Z}=\tfrac{1}{n}D$, 
\item
the coefficients of $\Delta_{Z}$ belong to $\Omega$, and 
\item
$q \boldsymbol{\rm N}$ is b-Cartier. 
\end{itemize}
\end{thm}

\begin{proof}
We first show the existence of $\Omega$. 
Let $\mathcal{J}(d,\Phi,p) \subset \mathbb{R}$ be the set of the generalized lc thresholds of effective Cartier divisors with respect to generalized lc pairs $(Y,\Gamma, \boldsymbol{\rm L})/W$ such that ${\rm dim}Y=d$, the coefficients of $\Gamma$ belong to $\Phi$, and $p\boldsymbol{\rm L}$ is b-Cartier.  
Then we have $\mathcal{J}(d,\Phi,p) \subset \mathbb{Q}$ since $\Phi \subset \mathbb{Q}$, and $\mathcal{J}(d,\Phi,p)$ is an ACC set by \cite[Theorem 1.5]{bz}. 
Let $P$ be a prime divisor on $Z$. 
Then there is an open subset $U \subset Z$ containing the generic point of $P$ such that the generalized lc threshold of $\pi^{*}P$ with respect to $(X,\Delta,\boldsymbol{\rm M})$ over the generic point of $P$ is equal to the generalized lc threshold of $\pi^{*}P|_{\pi^{-1}(U)}$ with respect to $(\pi^{-1}(U),\Delta|_{\pi^{-1}(U)},\boldsymbol{\rm M}|_{\pi^{-1}(U)})/U$. 
From this and the construction of the discriminant part of generalized canonical bundle formula, we see that the set
$$\Omega:=\set{1-t|t\in \mathcal{J}(d,\Phi,p)}$$
is the desired DCC set. 

From now on, we prove the existence of $n$ and $q$ of Theorem \ref{thm--eff-vertical}. 
Replacing $(X,\Delta,\boldsymbol{\rm M})$ by a $\mathbb{Q}$-factorial generalized dlt model, we may assume that $(X,0)$ is a $\mathbb{Q}$-factorial klt pair.  
Let $\Delta^{h}$ be the horizontal part of $\Delta$. 
Restricting $(X,\Delta,\boldsymbol{\rm M})$ to the general fiber of $\pi$ and applying \cite[Theorem 1.6]{bz}, we can find a finite set $\Phi' \subset \Phi$, which depends only on $d$, $p$, and $\Phi$, such that the coefficients of $\Delta^{h}$ belong to $\Phi'$. 

Since $\boldsymbol{\rm M}$ is b-big$/Z$ and $(X,\Delta^{h},\boldsymbol{\rm M} )$ is generalized klt $(X,\Delta^{h},\boldsymbol{\rm M} )$ has a good minimal model $(X',\Delta'^{h},\boldsymbol{\rm M})$ over $Z$ by \cite{bchm}. 
Let $\pi' \colon X' \to Z'$ be the contraction over $Z$ induced by $K_{X'}+\Delta'^{h}+\boldsymbol{\rm M}_{X'}$. 
Then $Z'\to Z$ is birational. 
Since $\boldsymbol{\rm M}$ is b-big$/Z$ and $(X',\Delta'^{h},\boldsymbol{\rm M})$ is generalized klt, it follows that $X'$ is of Fano type over $Z'$. 
By applying  Lemma \ref{lem--klt-canbundleformula} to $(X',\Delta'^{h},\boldsymbol{\rm M})\to Z'$, we get a generalized lc pair $(Z',\Delta_{Z'},\boldsymbol{\rm N}')$ on $Z'$ such that 
$$n(K_{X'}+\Delta'^{h}+\boldsymbol{\rm M}_{X'})\sim n\pi'^{*}(K_{Z'}+\Delta_{Z'}+\boldsymbol{\rm N}'_{Z'})$$
 and $q \boldsymbol{\rm N}'$ is b-Cartier, where $n$ and $q$ depend only on $d$, $\Phi'$, and $p$. 
 
By construction, there are open subsets $V \subset Z$ and $V'\subset Z'$ such that the restriction of $(X,\Delta,\boldsymbol{\rm M})\to Z$ over $V$ coincides with the restriction of $(X',\Delta'^{h},\boldsymbol{\rm M})\to Z'$ over $V'$. 
By an argument similar to \cite[3.4 (2)]{birkar-compl}, we see that the moduli part of the generalized canonical bundle formula for $\pi \colon (X,\Delta,\boldsymbol{\rm M}) \to Z$ depends only on the generic fiber of $\pi$. 
This implies that $\boldsymbol{\rm N}'$ is the moduli part of the generalized canonical bundle formula for $(X,\Delta,\boldsymbol{\rm M}) \to Z$. 
Thus, the integer $q$ satisfy the condition of Theorem \ref{thm--eff-vertical}. 

We can easily check that $n$ also satisfies the condition of Theorem \ref{thm--eff-vertical}. 
Thus, $n$ and $q$ are the desired positive integers. 
\end{proof}

\begin{thm}\label{thm--eff-dlt-general}
Let $d$ and $p$ be positive integers, and let $\Phi \subset \mathbb{Q}$ be a DCC set.  
Then, there exist positive integers $n$ and $q$ and a DCC set $\Omega \subset \mathbb{Q}$, depending only on $d$, $p$, and $\Phi$, satisfying the following. 
Let $\pi \colon X \to Z$ be a contraction of normal projective varieties and let $(X,\Delta,\boldsymbol{\rm M})$ be a generalized dlt pair such that
\begin{itemize}
\item
${\rm dim}X=d$, 
\item
the coefficients of $\Delta$ belong to $\Phi$, 
\item
$p \boldsymbol{\rm M}$ is b-Cartier, 
\item
$K_{X}+\Delta+\boldsymbol{\rm M}_{X}\sim_{\mathbb{Q},Z}0$, and
\item
there is a log resolution $f\colon \tilde{X}\to X$ of $(X,\Delta)$ such that 
\begin{itemize}
\item
$f$ is an isomorphism over an open subset $U \subset X$ containing all the generic points of generalized lc centers of $(X,\Delta,\boldsymbol{\rm M})$, 
\item
$\boldsymbol{\rm M}$ descends to $\tilde{X}$, and 
\item
writing
$K_{\tilde{X}}+\tilde{\Delta}+\boldsymbol{\rm M}_{\tilde{X}}=f^{*}(K_{X}+\Delta+\boldsymbol{\rm M}_{X})+\tilde{E}$,
where $\tilde{\Delta}\geq0$ and $\tilde{E}\geq0$ have no common components, then $\boldsymbol{\rm M}_{\tilde{X}}$ is log big with respect to $(\tilde{X},\tilde{\Delta})$. 
\end{itemize}
\end{itemize}
Then $n(K_{X}+\Delta+\boldsymbol{\rm M}_{X})\sim \pi^{*}D$ for some $\mathbb{Q}$-Cartier divisor $D$ on $Z$, and there exists a generalized lc pair $(Z,\Delta_{Z},\boldsymbol{\rm N})$ such that 
\begin{itemize}
\item
$K_{Z}+\Delta_{Z}+\boldsymbol{\rm N}_{Z} = \frac{1}{n}D$, 
\item
the coefficients of $\Delta_{Z}$ belong to $\Omega$, and 
\item
$q \boldsymbol{\rm N}$ is b-Cartier. 
\end{itemize}
\end{thm}

\begin{proof}
We prove the theorem by induction on the dimension of $X$. 

If all generalized lc centers of $(X,\Delta,\boldsymbol{\rm M})$ are vertical over $Z$, then Theorem \ref{thm--eff-dlt-general} directly follows from Theorem \ref{thm--eff-vertical}. 
Hence we may assume that there is a component $S$ of $\lfloor \Delta \rfloor$ dominating $Z$. 

Applying Theorem \ref{thm--eff-dlt-trivial} to the general fiber of $\pi$, we can find a positive integer $n$, depending only on $d$, $p$, and $\Phi$, such that $n(K_{X}+\Delta+\boldsymbol{\rm M}_{X}) \sim \pi^{*}D$ for some $\mathbb{Q}$-Cartier divisor $D$ on $Z$. 

Let $m$ be a positive integer, which may depend on $(X,\Delta,\boldsymbol{\rm M})$, such that $m\Delta$ is a Weil divisor and both $m \boldsymbol{\rm M}_{\tilde{X}}$ and $mD$ are Cartier. 
Then Lemma \ref{lem--extension-lc-center} implies that the morphism
$\pi_{*}\mathcal{O}_{X}\otimes \mathcal{O}_{Z}(mD) \to \pi_{S*}\mathcal{O}_{S}\otimes \mathcal{O}_{Z}(mD)$ is surjective, where $\pi_{S}=\pi|_{S}$. 
Thus $\pi_{S}$ is a contraction. 

Let $(S,\Delta_{S},\boldsymbol{\rm L})$ be a generalized dlt pair constructed with divisorial adjunction for generalized pairs. 
Then we have 
$$n(K_{S}+\Delta_{S}+\boldsymbol{\rm L}_{S}) \sim \pi_{S}^{*}D.$$ 
By Remark \ref{rem--coeff-adj}, there is a DCC set $\Psi \subset \mathbb{Q}$,  depending only on $d$, $p$, and $\Phi$, such that the coefficients of $\Delta_{S}$ belong to $\Psi$. 
We apply the induction hypothesis of Theorem \ref{thm--eff-dlt-general} to $\pi_{S} \colon (S,\Delta_{S},\boldsymbol{\rm L}) \to Z$. We can find positive integers $n'$ and $q'$ and a DCC set $\Omega \subset \mathbb{Q}$, depending only on $d$, $p$, and $\Psi$, such that 
$$n'(K_{S}+\Delta_{S}+\boldsymbol{\rm L}_{S})\sim \pi_{S}^{*}D'$$
 for some $\mathbb{Q}$-Cartier divisor $D'$ on $Z$ and there exists a generalized lc pair $(Z,\Delta'_{Z},\boldsymbol{\rm N}')$ such that 
\begin{itemize}
\item
$K_{Z}+\Delta'_{Z}+\boldsymbol{\rm N}'_{Z} = \frac{1}{n'}D'$, 
\item
the coefficients of $\Delta'_{Z}$ belong to $\Omega$, and 
\item
$q' \boldsymbol{\rm N}'$ is b-Cartier. 
\end{itemize}

Let $\boldsymbol{\rm D}$ and $\boldsymbol{\rm D}'$ be the b-divisors defined by $\boldsymbol{\rm D}=\overline{D}$ and $\boldsymbol{\rm D}'=\overline{D'}$, respectively. 
We define a generalized lc pair $(Z,\Delta_{Z},\boldsymbol{\rm N})$ by putting 
$$\Delta_{Z}=\Delta'_{Z} \quad {\rm and} \quad \boldsymbol{\rm N}=\boldsymbol{\rm N}'-\frac{1}{n'}\boldsymbol{\rm D}'+\frac{1}{n}\boldsymbol{\rm D}.$$ 
We will prove that $nn'q' \boldsymbol{\rm N}'$ is b-Cartier. 
Let $\overline{Z} \to Z$ be a resolution of $Z$ such that $\boldsymbol{\rm N}'$  descends to $\overline{Z}$. 
Let $\overline{S} \to S$ be a resolution of $S$ such that the induced map $\pi_{\overline{S}}\colon \overline{S} \dashrightarrow \overline{Z}$ is a morphism. 
By the relation
$$nn' \pi_{S}^{*}\left(\frac{1}{n}D\right) \sim nn'(K_{S}+\Delta_{S}+\boldsymbol{\rm L}_{S}) \sim nn'\pi_{S}^{*}\left(\frac{1}{n'}D'\right)$$
and the definitions of $\boldsymbol{\rm N}'$ and $\boldsymbol{\rm N}$, it follows that 
$nn'q'\pi_{\overline{S}}^{*}\boldsymbol{\rm N}_{\overline{Z}} \sim nn'q'\pi_{\overline{S}}^{*}\boldsymbol{\rm N}'_{\overline{Z}}$. 
Since $nn'q' \boldsymbol{\rm N}'_{\overline{Z}}$ is Cartier, by \cite[Proposition 5.3]{filipazzimoraga} we see that $nn'q'\boldsymbol{\rm N}_{\overline{Z}}$ is Cartier. 

We put $q=nn'q'$. 
The generalized lc pair $(Z,\Delta_{Z},\boldsymbol{\rm N})$ satisfies $K_{Z}+\Delta_{Z}+\boldsymbol{\rm N}_{Z} = \frac{1}{n}D$, the coefficients of $\Delta_{Z}$ belong to $\Omega$, and $q \boldsymbol{\rm N}$ is b-Cartier. 
From the above discussion, the integers $n$ and $q$ and the DCC set $\Omega$ satisfy the condition of Theorem \ref{thm--eff-dlt-general}. 
\end{proof}

We are ready to prove Theorem \ref{thm--eff-iitaka-intro} and Theorem \ref{thm--base-iitaka-intro}.

\begin{proof}[Proof of Theorem \ref{thm--base-iitaka-intro}]
Let $\phi \colon X \dashrightarrow X'$ be finite steps of a $(K_{X}+\Delta+M)$-MMP such that $\phi_{*}(K_{X}+\Delta+M)$ is semi-ample, and let $\pi \colon X' \to Z$ be the contraction induced by $\phi_{*}(K_{X}+\Delta+M)$. 
We put $\Delta'=\phi_{*}\Delta$ and $M'=\phi_{*}M$. 
We consider the generalized dlt pair $(X',\Delta', \overline{M})$. 

Let $U' \subset X'$ be an open subset such that $(U',\Delta'|_{U'})$ is log smooth, $U'$ contains all the generic points of generalized lc centers of $(X',\Delta', \overline{M})$, and $X' \dashrightarrow X$ is an isomorphism on $U'$. 
Since $M$ is nef and log big with respect to $(X,\Delta)$, by taking an appropriate resolution of the graph of $\phi$, we get a common log resolution $f \colon \tilde{X} \to X$ and $f' \colon \tilde{X} \to X'$ of $(X,\Delta)\dashrightarrow (X',\Delta')$ such that $f'$ and $U'$ satisfy the final condition of Theorem \ref{thm--eff-dlt-general}. 

Since $K_{X}+\Delta+M$ is not big, we have ${\rm dim}Z < {\rm dim}X'$. 
Let $F$ be the general fiber of $\pi$, and put $\tilde{F}=f'^{-1}(F)$. 
We recall the hypothesis that $M=\sum_{i}\mu_{i}M_{i}$ and $M_{i}$ are log big with respect to $(X,\Delta)$. 
Since ${\rm dim}F>0$, we see that $(f^{*}M_{i})|_{\tilde{F}}$ is not numerically trivial for all $i$. 
Applying the global ACC \cite[Theorem 1.6]{bz} to the restriction of $(X',\Delta', \overline{M})$ to $F$, we see that all $\mu_{i}$ belong to a finite set $\Phi'\subset \Phi$ depending only on $d$ and $\Phi$. 
Therefore, there exists a positive integer $p$, depending only on $d$ and $\Phi$, such that $p\mu_{i}$ are integers. 
Then $p\overline{M}$ is b-Cartier. 

Obviously, the coefficients of $\Delta'$ belong to $\Phi$, hence we can apply Theorem \ref{thm--eff-dlt-general} to $(X',\Delta',\overline{M}) \to Z$. 
By Theorem \ref{thm--eff-dlt-general}, there exist positive integers $n$ and $q$ and a DCC set $\Omega \subset \mathbb{Q}$, depending only on $d$, $p$, and $\Phi$, such that 
$$n(K_{X'}+\Delta'+M')\sim \pi^{*}D$$
for some $\mathbb{Q}$-Cartier divisor $D$ on $Z$ and there is a generalized lc pair $(Z,\Delta_{Z},\boldsymbol{\rm N})$ such that 
\begin{itemize}
\item
$K_{Z}+\Delta_{Z}+\boldsymbol{\rm N}_{Z} = \frac{1}{n}D$, 
\item
the coefficients of $\Delta_{Z}$ belong to $\Omega$, and 
\item
$q \boldsymbol{\rm N}$ is b-Cartier. 
\end{itemize}
Then $n$, $q$, and $\Omega$ depend only on $d$ and $\Phi$ because $p$ depends only on $d$ and $\Phi$. 
Since we have 
\begin{equation*}
\begin{split}
H^{0}(X, \mathcal{O}_{X}(\lfloor ln (K_{X}+\Delta+M)\rfloor)) & \simeq H^{0}(X', \mathcal{O}_{X'}(\lfloor ln (K_{X'}+\Delta'+M')\rfloor)) \\
& \simeq H^{0}(X', \mathcal{O}_{X'}(\lfloor l \pi^{*}D\rfloor)) \\
& \simeq H^{0}(Z, \mathcal{O}_{Z}(\lfloor ln (K_{Z}+\Delta_{Z}+\boldsymbol{\rm N}_{Z})\rfloor)) 
\end{split}
\end{equation*}
for every positive integer $l$, we get the desired $n$, $p$, and $\Omega$. 
\end{proof}

\begin{proof}[Proof of Theorem \ref{thm--eff-iitaka-intro}]
The theorem follows from Theorem \ref{thm--base-iitaka-intro} and \cite[Theorem 1.3]{bz}. 
\end{proof}

\begin{exam}\label{exam--unbounded}
We fix positive integers $d$ and $p$. 
We consider the category $\mathcal{C}$ whose objects are generalized dlt pairs $(X,\Delta,\overline{M})$ such that ${\rm dim}X=d$, $\Delta$ is a Weil divisor, $pM$ is Cartier, $M$ is nef and log big with respect to $(X,\Delta)$, and $K_{X}+\Delta+M\sim_{\mathbb{Q}}0$. 
We show that even if $d=2$ and $p=1$, some set 
 $$\mathcal{D} \subset \{X\,|\,\text{$(X,\Delta,\overline{M})\in \mathcal{C}$ for some $\Delta$ and $M$}\}$$ 
is unbounded. 

We put $V=\mathbb{P}^{1}$. 
We fix a very ample Cartier divisor $H_{V}$ such that $\mathcal{O}_{V}(H_{V})=\mathcal{O}_{V}(1)$.  
We construct a $\mathbb{P}^{1}$-bundle
$$X_{n}:=\mathbb{P}_{V}(\mathcal{O}_{V}\oplus \mathcal{O}_{V}(-nH_{V}))\overset{f}{\longrightarrow} V.$$ 
Let $\Delta_{n}$ be the unique section corresponding to $\mathcal{O}_{X_{n}}(1)$.  Define $M_{n}:=\Delta_{n}+(n+2)f^{*}H_{V}$. 
By construction, we have
$K_{X_{n}}+\Delta_{n} +M_{n}\sim 0.$ 
We can easily check that $(X_{n},\Delta_{n})$ is a dlt pair and $M_{n}$ is an ample Cartier divisor. 
We also see that $(-K_{X_{n}}\cdot A_{n})\geq n$ for every $n$ and every ample Cartier divisor $A_{n}$ on $X_{n}$. 
Indeed, we may write $A_{n}\sim a_{n}\Delta_{n}+b_{n}f^{*}H_{V}$ for some $a_{n}\geq 1$ and $b_{n}\geq 1$ (cf.~\cite[III, Exercise 12.5]{hartshorne}). 
Then
$$(-K_{X_{n}}\cdot A_{n})= (\Delta_{n}+M_{n})\cdot A_{n}> (nf^{*}H_{V}\cdot A_{n}) \geq  (nf^{*}H_{V}\cdot a_{n}\Delta_{n})\geq n.$$
This fact shows that the set $\{X_{n}\}_{n\geq 1}$ is unbounded. 
\end{exam}

\section{Boundedness results}\label{sec5}

In this section, we prove some boundedness results. 
Before we start the main parts of this section, 
we recall the notion of the Iitaka volume for $\mathbb{Q}$-Cartier divisors and 
we prove Theorem \ref{thm--dcc-iitakavol-intro}. 

\begin{defn}[Iitaka volume, {\cite[Definition 1.1]{li-bounded}}]\label{defn--iitakavol}Let $X$ be a normal projective variety, and let $D$ be a $\mathbb{Q}$-Cartier divisor on $X$ such that the Iitaka dimension $\kappa(X,D)$ is nonnegative. Then the {\em Iitaka volume} of $D$, denoted by ${\rm Ivol(D)}$, is defined by $${\rm Ivol(D)}:=\underset{m\to \infty}{\rm lim\,sup}\frac{{\rm dim}H^{0}(X,\mathcal{O}_{X}(\lfloor mD \rfloor))}{m^{\kappa(X,D)} \slash \kappa(X,D)!}.$$\end{defn}  

\begin{proof}[Proof of Theorem \ref{thm--dcc-iitakavol-intro}]
The theorem follows from Theorem \ref{thm--base-iitaka-intro} and \cite[Theorem 1.3]{birkar-nefpart}. 
With notations as in Theorem \ref{thm--base-iitaka-intro}, the divisor $K_{Z}+\Delta_{Z}+\boldsymbol{\rm N}_{Z}$ is ample and $\mathbb{Q}$-Cartier. 
Therefore, 
${\rm Ivol}(K_{X}+\Delta+M)=(K_{Z}+\Delta_{Z}+\boldsymbol{\rm N}_{Z})^{{\rm dim}Z} \in \mathbb{Q}_{>0}$. 
\end{proof}

From now on, we prove some boundedness results, and we prove Theorems \ref{thm--compl-dlt-intro} and Theorem \ref{thm--base-boundedness-intro}.

\subsection{Boundedness of complements}\label{subsec5.1}

In this subsection, we prove Theorem \ref{thm--compl-dlt-intro}. 

\begin{lem}\label{lem--hyperstandardset}
Let $\mathfrak{R}\subset [0,1]$ be a finite set of rational numbers, and let $\Phi(\mathfrak{R})$ be the hyperstandard set associated to $\mathfrak{R}$. 
Let $q$ be a positive integer such that $qa \in \mathbb{Z}$ for every $a\in \mathfrak{R}$. 
Then $b-\{-qmb\}\geq 0$ for every $b\in \Phi(\mathfrak{R})$ and every positive integer $m$. 
\end{lem}

\begin{proof}
We may write $b=1-\frac{a}{r}$ for some $a\in \mathfrak{R}$ and $r \in \mathbb{Z}_{>0}$. 
Then 
$$b-\{-qmb\}=1-\frac{a}{r}-\left\{ -qm\left(1-\frac{a}{r}\right)\right\}=1-\frac{a}{r}-\left\{ \frac{qma}{r}\right\}.$$
Since $qa \in \mathbb{Z}$ and $m \in \mathbb{Z}$, we have $\{\frac{qma}{r}\} \leq \frac{r-1}{r}$. 
Since $a\leq 1$, we have
$$1-\frac{a}{r}-\left\{ \frac{qma}{r}\right\}\geq 1- \frac{1}{r}-\frac{r-1}{r}\geq 0.$$
Therefore $b-\{-qmb\}\geq 0$. 
\end{proof}

\begin{lem}\label{lem--gen-member}
Let $(X,\Delta)$ be a sub-pair such that $X$ is projective, and let $D$ be a Cartier divisor on $X$. 
Let $n$ be a positive integer, and let $V_{1},\cdots ,V_{l}$ be subvarieties of $X$. 
Suppose that  for each $1 \leq i \leq l$, there is $D_{i}\in |D|$ such that $(X,\Delta+\frac{1}{n}D_{i})$ is sub-lc on a neighborhood of $V_{i}$. 
Then, for every general member $D' \in |D|$ the sub-pair $(X,\Delta+\frac{1}{n}D')$ is sub-lc on a neighborhood of $\bigcup _{i=1}^{l}V_{i}$. 
\end{lem}

\begin{proof}
We take a resolution $f \colon Y \to X$ of the linear system $|D|$, i.e., $f^{*}|D|=|M|+F$ where $M$ is a base point free Cartier divisor and $F$ is the fixed divisor. 
Replacing $X$ by $Y$ and replacing $\Delta$, $D$, and $V_{1},\cdots ,V_{l}$ accordingly, we may assume that the movable part of $|D|$ is base point free. 
By the hypothesis, $F$ satisfies the property that $(X,\Delta+\frac{1}{n}F)$ is sub-lc on a neighborhood of $\bigcup _{i=1}^{l}V_{i}$. 
Thus, Lemma \ref{lem--gen-member} holds true. 
\end{proof}

The following theorem is the main result of this subsection.

\begin{thm}\label{thm--compl-dlt}
Let $d$ be a positive integer. 
Let $\mathfrak{R} \subset [0,1]$ be a finite set of rational numbers and $\Phi(\mathfrak{R})$ the hyperstandard set associated to $\mathfrak{R}$.  
Then there exists a positive integer $n$, depending only on $d$ and $\mathfrak{R}$, satisfying the following. 
Let $(X,\Delta)$ be a projective dlt pair such that 
\begin{itemize}
\item
${\rm dim}X=d$,
\item
the coefficients of $\Delta$ belong to $\Phi(\mathfrak{R})$, and
\item
$-(K_{X}+\Delta)$ is nef and log big.
\end{itemize}
Then there is an effective $\mathbb{Q}$-divisor $\Xi$ such that $(X,\Delta+\Xi)$ is lc and $n(K_{X}+\Delta+\Xi)\sim0$. 
\end{thm}

\begin{proof}
We prove the theorem by induction on the dimension of $X$. 

If $(X,\Delta)$ is klt, then the existence of $n$ is proved by Birkar \cite{birkar-compl}. 
So we may assume that $(X,\Delta)$ is not klt. 
Let $p$ be the minimum positive integer such that $pa \in \mathbb{Z}$ for all $a \in \mathfrak{R}$. 

Since $(X,\Delta)$ is dlt, there is a log resolution $f \colon \tilde{X} \to X$ of $(X,\Delta)$ such that $f$ is an isomorphism over an open subset $U \subset X$ containing all the generic points of lc centers of $(X,\Delta)$. 
We put $\tilde{D}=f^{*}(K_{X}+\Delta)$. We put $\tilde{S}=f_{*}^{-1}S$ for any component $S$ of $\lfloor \Delta \rfloor$. 
We prove Theorem \ref{thm--compl-dlt} in two steps. 

\begin{step2}\label{step1--compl-dlt}
In this step. we prove that the morphsim 
$$H^{0}(\tilde{X},\mathcal{O}_{\tilde{X}}(\lfloor -mp \tilde{D} \rfloor))\longrightarrow H^{0}(\tilde{S},\mathcal{O}_{\tilde{S}}(\lfloor -mp \tilde{D} \rfloor|_{\tilde{S}}))$$
defined with Definition \ref{defn--rest-mor} is surjective for every component $S$ of $\lfloor \Delta \rfloor$ and every positive integer $m$. 
We closely follow the proof of Lemma \ref{lem--extension-lc-center}.  

We can write
$$K_{\tilde{X}}+\tilde{\Delta}=f^{*}(K_{X}+\Delta)+\tilde{E},$$
where $\tilde{\Delta}\geq0$ and $\tilde{E}\geq0$ have no common components.  Then $-\tilde{D}$ is nef and log big with respect to $(\tilde{X},\tilde{\Delta})$, and we can write 
$$K_{\tilde{X}}+\tilde{\Delta}-(mp+1)\tilde{D}=-mp \tilde{D}+\tilde{E}.$$
We can write $-mp\tilde{D}=\lfloor -mp\tilde{D} \rfloor + \{-mp f_{*}^{-1}\Delta\} +\tilde{E}'$ with an effective $f$-exceptional $\mathbb{Q}$-divisor $\tilde{E}'$. 
We put 
$$\tilde{\Delta}'=\tilde{\Delta}-\{-mp f_{*}^{-1}\Delta\}-\tilde{S}.$$ 
By Lemma \ref{lem--hyperstandardset}, we have $\tilde{\Delta}'\geq0$ and
$$K_{\tilde{X}}+\tilde{S}+\tilde{\Delta}'-(mp+1)\tilde{D}=\lfloor -mp\tilde{D} \rfloor +\tilde{E}+ \tilde{E}'.$$

We define a Weil divisor $\tilde{F}$ on $\tilde{X}$ so that 
\begin{equation*}
{\rm coeff}_{P}(\tilde{F})=\left\{  \begin{array}{l}{0 \qquad ({\rm coeff}_{P}(\tilde{\Delta}'-\{(\tilde{E}+\tilde{E}')\})\geq 0)} \\ {1 \qquad  ({\rm coeff}_{P}(\tilde{\Delta}'-\{(\tilde{E}+\tilde{E}')\})< 0)} \end{array}\right. \end{equation*} 
for every prime divisor $P$. 
We put
$$\tilde{B}=\tilde{\Delta}'-\{(\tilde{E}+\tilde{E}')\}+\tilde{F} \qquad {\rm and} \qquad \tilde{G}=\lfloor( \tilde{E} + \tilde{E}')\rfloor +\tilde{F}.$$
Then we have
\begin{equation*}
K_{\tilde{X}}+\tilde{S}+\tilde{B} -(mp+1) \tilde{D}=\lfloor -mp \tilde{D} \rfloor  +\tilde{G}.
\end{equation*}
By the definition, we can check that $\tilde{B}$ is a boundary divisor, $\lfloor \tilde{B} \rfloor \leq \lfloor \tilde{\Delta}' \rfloor$, and $\tilde{G}$ is an effective $f$-exceptional Weil divisor. 
Then $-\tilde{D}$ is nef and log big with respect to $(\tilde{X},\tilde{S}+\tilde{B})$. 
By the Kodaira type vanishing theorem \cite[Lemma 1.5]{fujino-abund-logbig}, we see that 
$$H^{0}(\tilde{X},\mathcal{O}_{\tilde{X}}(\lfloor -mp \tilde{D} \rfloor+\tilde{G}))\longrightarrow H^{0}(\tilde{S},\mathcal{O}_{\tilde{S}}((\lfloor -mp \tilde{D} \rfloor +\tilde{G})|_{\tilde{S}}))$$
is surjective. 
Since $\tilde{G}$ is $f$-exceptional, there are natural isomorphisms
$$H^{0}(\tilde{X},\mathcal{O}_{\tilde{X}}(\lfloor -mp \tilde{D} \rfloor)) \simeq H^{0}(X,\mathcal{O}_{X}(\lfloor -mp (K_{X}+\Delta) \rfloor)) \simeq H^{0}(\tilde{X},\mathcal{O}_{\tilde{X}}(\lfloor -mp \tilde{D} \rfloor+\tilde{G})).$$
From the following diagram
 $$
\xymatrix{
H^{0}(\tilde{X},\mathcal{O}_{\tilde{X}}(\lfloor -mp \tilde{D} \rfloor+\tilde{G})) \ar@{->>}[r]&H^{0}(\tilde{S},\mathcal{O}_{\tilde{S}}((\lfloor -mp \tilde{D} \rfloor +\tilde{G})|_{\tilde{S}}))\\
H^{0}(\tilde{X},\mathcal{O}_{\tilde{X}}(\lfloor -mp \tilde{D} \rfloor)) \ar[u]^{\simeq}\ar[r]&H^{0}(\tilde{S},\mathcal{O}_{\tilde{S}}(\lfloor -mp \tilde{D} \rfloor|_{\tilde{S}})),\ar@{^{(}->}[u]
}
$$
the lower horizontal morphism is surjective. 
We finish this step. 
\end{step2}

For each component $S$ of $\lfloor \Delta \rfloor$, we define a dlt pair $(S,\Delta_{S})$ by divisorial adjunction $K_{S}+\Delta_{S}=(K_{X}+\Delta)|_{S}$. 
By Remark \ref{rem--coeff-adj}, there exists a finite set $\mathfrak{S}\subset [0,1]$ of rational numbers, depending only on $d$ and $\mathfrak{R}$, such that all the coefficients of $\Delta_{S}$ belong to the hyperstandard set associated to $\mathfrak{S}$. 
By the induction hypothesis of Theorem \ref{thm--compl-dlt}, there is an $n'$, depending only on $d-1$ and $\mathfrak{S}$, such that for every $(S,\Delta_{S})$ where $S$ is a component of $\lfloor \Delta \rfloor$, there is 
$$\Delta_{S}^{+} \geq \Delta_{S}$$
such that $(S,\Delta_{S}^{+})$ is lc and $n'(K_{S}+\Delta_{S}^{+})\sim 0$. 

\begin{step2}\label{step2--compl-dlt}
We put $n=n'p$. 
From now on, we prove that the $n$ satisfies the property of Theorem \ref{thm--compl-dlt}. 

Pick a component $S$ of $\lfloor \Delta \rfloor$, put $\tilde{S}=f_{*}^{-1}S$, and take a $\mathbb{Q}$-divisor $\Delta_{S}^{+} \geq \Delta_{S}$ on $S$ such that $(S,\Delta_{S}^{+})$ is lc and $n'(K_{S}+\Delta_{S}^{+})\sim 0$. 
We put 
$$C_{S}=n(\Delta^{+}_{S}-\Delta_{S}) \sim -n(K_{S}+\Delta_{S}).$$ 
We put $C_{\tilde{S}}=f_{\tilde{S}}^{*}C_{S}$, where $f_{\tilde{S}}=f|_{\tilde{S}} \colon \tilde{S} \to S$. 
Then $-n \tilde{D}|_{\tilde{S}}\sim C_{\tilde{S}}$. 
We put $\tilde{R}=\{ -n \tilde{D} \}$. 
Then 
$$-n \tilde{D}|_{\tilde{S}}=\lfloor -n \tilde{D} \rfloor|_{\tilde{S}}+\tilde{R}|_{\tilde{S}},$$ and $\tilde{R}|_{\tilde{S}}$ is well-defined as a $\mathbb{Q}$-divisor on $\tilde{S}$.  
We have $\lfloor \tilde{R}|_{\tilde{S}} \rfloor=0$ since ${\rm Supp}(\tilde{R}+\tilde{S})$ is snc. 
Therefore, we have 
$$\lfloor -n \tilde{D} \rfloor|_{\tilde{S}}=\lfloor -n \tilde{D}|_{\tilde{S}} \rfloor \quad  {\rm and} \quad \tilde{R}|_{\tilde{S}}=\{-n \tilde{D}|_{\tilde{S}}\}.$$ 
Therefore, we have
$$\lfloor -n \tilde{D} \rfloor|_{\tilde{S}} \sim \lfloor C_{\tilde{S}} \rfloor \qquad {\rm and} \qquad \tilde{R}|_{\tilde{S}}=\{C_{\tilde{S}}\}.$$
By Step \ref{step1--compl-dlt}, we can find an effective Weil divisor $\tilde{C}\sim \lfloor -n \tilde{D} \rfloor$ such that $\tilde{C}|_{\tilde{S}}=\lfloor C_{\tilde{S}} \rfloor$. 
We put $C=f_{*}(\tilde{C}+\tilde{R})$. 
Since $\tilde{C}+\tilde{R}\sim -n \tilde{D}$ and $\tilde{D}=f^{*}(K_{X}+\Delta)$, we can easily check $C\sim -n(K_{X}+\Delta)$ and $\tilde{C}+\tilde{R}=f^{*}C$. 
By definition of $C$, we see that 
$$C|_{S}=f_{\tilde{S}*}((\tilde{C}+\tilde{R})|_{\tilde{S}})=f_{\tilde{S}*}C_{\tilde{S}}=C_{S}=n(\Delta^{+}_{S}-\Delta_{S}).$$
From this and inversion of adjunction \cite{kawakita}, the pair $(X,\Delta+\frac{1}{n}C)$ is lc on a neighborhood of $S$. 
We define $\tilde{\Gamma}$ by $K_{\tilde{X}}+\tilde{\Gamma}=f^{*}(K_{X}+\Delta)$. 
Since $f^{*}C=\tilde{C}+\tilde{R}$, we see that the sub-pair $(\tilde{X},\tilde{\Gamma}+\frac{1}{n}(\tilde{C}+\tilde{R}))$ is sub-lc on a neighborhood of $f^{-1}(S)$. 

In this way,  for every component $S$ of $\lfloor \Delta \rfloor$, we may find $\tilde{C}\in|\lfloor -n \tilde{D} \rfloor|$ such that the sub-pair $(\tilde{X},\tilde{\Gamma}+\frac{1}{n}\tilde{R}+\frac{1}{n}\tilde{C})$ is sub-lc on a neighborhood of $f^{-1}(S)$.  
By Lemma \ref{lem--gen-member}, there is 
$$\tilde{M}\in|\lfloor -n \tilde{D} \rfloor|$$ such that $(\tilde{X},\tilde{\Gamma}+\frac{1}{n}\tilde{R}+\frac{1}{n}\tilde{M})$ is sub-lc on a neighborhood of $f^{-1}(\lfloor \Delta \rfloor)$. 
By construction, we have $\tilde{M}+\tilde{R}\sim-nf^{*}(K_{X}+\Delta)$. 
By putting 
$$M=f_{*}(\tilde{M}+\tilde{R}),$$ we have $M \sim -n(K_{X}+\Delta)$ and the pair $(X,\Delta+\frac{1}{n}M)$ is lc on a neighborhood of $\lfloor \Delta \rfloor$. 
\end{step2}

Finally, if $(X,\Delta+\frac{1}{n}M)$ is not lc, then there is $\epsilon \in (0, \frac{1}{n})$ such that the non-klt locus of $(X,\Delta+(\frac{1}{n}-\epsilon)M)$ has at least two connected components. 
Then we get a contradiction by the connectedness principle for non-klt locus \cite[17.4 Theorem]{kollar}. 
Thus, $(X,\Delta+\frac{1}{n}M)$ is lc, and therefore $\Xi:=\frac{1}{n}M$ is the desired divisor. 
\end{proof}

\begin{proof}[Proof of Theorem \ref{thm--compl-dlt-intro}]
We will freely use the notations in Theorem \ref{thm--compl-dlt-intro}. 
We may assume that $K_{X}+\Delta+M$ is not big. 
By the argument similar to the proof of Theorem \ref{thm--base-iitaka-intro}, the global ACC \cite[Theorem 1.6]{bz} implies the existence of a positive integer $m$, depending only on $d$ and $\Phi$, such that $m\Delta'_{F}$ is a Weil divisor and $mM$ is Cartier. 
By Theorem \ref{thm--eff-dlt-trivial}, we have 
$n_{1}(K_{F}+\Delta'_{F}+M'_{F})\sim 0$ for some positive integer $n_{1}$ which depends only on $d$ and $m$. 
If there exists $n_{2} \in \mathbb{Z}_{>0}$ depending only on $d$ and $m$ such that $(F,\Delta'_{F}+\frac{1}{n_{2}}B'_{F})$ is lc for some 
$B'_{F}\in |-n_{2}(K_{F}+\Delta'_{F})|,$
 then $n:=n_{1}n_{2}$ is the desired positive integer. 
In this way, we only need to prove the existence of the above $n_{2} \in \mathbb{Z}_{>0}$. 

We run a $(K_{X}+\Delta+3dmM)$-MMP with scaling of an ample divisor. 
By the length of extremal rays (\cite[Section 18]{fujino-fund} or \cite[Theorem 4.6.2]{fujino-book}) and the same argument as in the proof of Theorem \ref{thm--mmp-neflogbig}, we get a sequence of steps of a $(K_{X}+\Delta+3dmM)$-MMP to a good minimal model 
$$(X,\Delta,3dm\overline{M}) \dashrightarrow (\tilde{X},\tilde{\Delta},3dm\overline{M})$$ such that 
\begin{itemize}
\item
$X \dashrightarrow \tilde{X}$ is a sequence of steps of a $(K_{X}+\Delta+M)$-MMP, and
\item 
defining $\tilde{M}$ to be the birational transform of $M$ on $\tilde{X}$, then $m\tilde{M}$ is a nef Cartier divisor on $\tilde{X}$ that is log big with respect to $(\tilde{X},\tilde{\Delta})$. 
\end{itemize}
With the argument in \cite[Proof of Lemma 2.14]{has-mmp}, we may construct a sequence of steps of a $(K_{\tilde{X}}+\tilde{\Delta}+\tilde{M})$-MMP with scaling of $(3dm-1)\tilde{M}$ terminating with a good minimal model 
$$(\tilde{X},\tilde{\Delta},\overline{M}) \dashrightarrow (\tilde{X}',\tilde{\Delta}',\overline{M}).$$

Let $\tilde{M}'$ be the birational transform of $M$ on $\tilde{X}'$. 
By construction, $K_{\tilde{X}'}+\tilde{\Delta}'+(1+t)\tilde{M}'$ is nef for some $t>0$. 
Moreover, taking a common resolution of $\tilde{X} \dashrightarrow \tilde{X}'$ and using the negativity lemma, we see that $\tilde{M}'$ is log big with respect to $(\tilde{X}',\tilde{\Delta}')$. 
Let $\tilde{\pi}\colon \tilde{X}' \to Z$ be the contraction induced by $K_{\tilde{X}'}+\tilde{\Delta}'+\tilde{M}'$, and let $\tilde{F}$ be the general fiber of $\tilde{\pi}$. 
We put $\tilde{\Delta}'_{\tilde{F}}=\tilde{\Delta}'|_{\tilde{F}}$ and $\tilde{M}'_{\tilde{F}}=\tilde{M}'|_{\tilde{F}}$. 

Since $K_{\tilde{X}'}+\tilde{\Delta}'+\tilde{M}' \sim_{\mathbb{Q},Z}0$ and $K_{\tilde{X}'}+\tilde{\Delta}'+(1+t)\tilde{M}'$ is nef for some $t>0$, we see that the divisor 
$$\tilde{M}'\sim_{\mathbb{Q},Z}-(K_{\tilde{X}'}+\tilde{\Delta}')$$
 is nef and log big over $Z$ with respect to $(\tilde{X}',\tilde{\Delta}')$. 
By Theorem \ref{thm--compl-dlt}, there exists a positive integer $n_{2}$, depending only on $d$ and $m$, such that $(\tilde{F},\tilde{\Delta}'_{\tilde{F}}+\frac{1}{n_{2}}\tilde{B}'_{\tilde{F}})$ is an lc pair for some $\tilde{B}'_{\tilde{F}}\in |-n_{2}(K_{\tilde{F}}+\tilde{\Delta}'_{\tilde{F}})|$. 
Because $X\dashrightarrow X'$ and $X \dashrightarrow \tilde{X}'$ are sequences of steps of $(K_{X}+\Delta+M)$-MMP to good minimal models, the induced birational map $X'\dashrightarrow \tilde{X}'$ is isomorphic in codimension one, and therefore so is $F \dashrightarrow \tilde{F}$. 
Then $(F,\Delta'_{F}+\frac{1}{n_{2}}B'_{F})$ is lc for some $B'_{F}\in |-n_{2}(K_{F}+\Delta'_{F})|$. 
Hence $n:=n_{1}n_{2}$ satisfies the conditions of Theorem \ref{thm--compl-dlt-intro}. 
\end{proof}

The following example shows that the condition of the log bigness of the nef divisors in Theorem \ref{thm--eff-iitaka-intro} and Theorem \ref{thm--compl-dlt-intro} cannot be relaxed to the bigness.

\begin{exam}\label{exam-eff-nonvan-counter}Let $(X,\Delta)$ be a projective log smooth dlt pair such that $\Delta$ is a reduced divisor, $H^{1}(X,\mathcal{O}_{X})\neq 0$, and $-(K_{X}+\Delta)$ is nef and big. For example, take an elliptic curve $E$ with a very ample divisor $H$ and define $$X:=\mathbb{P}_{E}(\mathcal{O}_{E}\oplus \mathcal{O}_{E}(-H)).$$ For each $n \geq 1$, let $D_{n}$ be a Cartier divisor on $X$ such that $nD_{n} \sim 0$ and $lD_{n} \not\sim 0$ for all $0<l<n$. We define $L_{n}:=-(K_{X}+\Delta)+D_{n}$, and we consider the set $\{(X,\Delta,\overline{L_{n}})\}_{n\geq 1}$. If  the log bigness of $M_{i}$ in Theorem \ref{thm--eff-iitaka-intro} or Theorem \ref{thm--compl-dlt-intro} can be relaxed to the bigness, then there is a positive integer $m$ such that $mD_{n}=m(K_{X}+\Delta+L_{n})\sim 0$ for all $n$, which contradicts the definition of $D_{n}$. Therefore, the log bigness of $M_{i}$ in Theorem \ref{thm--eff-iitaka-intro} and Theorem \ref{thm--compl-dlt-intro} cannot be relaxed to the bigness. \end{exam}

\subsection{Effective finite generation}

In this subsection, we prove Theorem \ref{thm--base-boundedness-intro}. 

We will first prove the effective base point free theorem for special generalized dlt pairs (Theorem \ref{thm--eff-basepoint-free-gen-dlt}). 
For the proof, we need the base point free theorem for quasi-log schemes \cite[Theorem 6.5.1]{fujino-book} (see also \cite[Theorem 5.1]{ambro-quasi-log}). 

\begin{defn}[{Quasi-log scheme, \cite[Definition 6.2.2]{fujino-book}, see also \cite[Definition 4.1]{ambro-quasi-log}}]\label{x-def7.2}
A {\em{quasi-log scheme}} is a scheme $X$ endowed with an $\mathbb R$-Cartier divisor (or $\mathbb R$-line bundle) $\omega$ on $X$, a closed subscheme $X_{-\infty}\subsetneq X$, and a finite collection $\{C\}$ of reduced and irreducible subschemes of $X$ such that there is a proper morphism $f\colon (Y, B_Y)\to X$ from a globally embedded simple normal crossing pair satisfying the following properties: 
\begin{itemize}
\item $f^*\omega\sim_{\mathbb R}K_Y+B_Y$, 

\item the natural map $\mathcal O_X \to f_*\mathcal O_Y(\lceil -(B_Y^{<1})\rceil)$ induces an isomorphism 
$$
\mathcal I_{X_{-\infty}}\overset{\simeq}{\longrightarrow} f_*\mathcal O_Y(\lceil -(B_Y^{<1})\rceil-\lfloor B_Y^{>1}\rfloor),  
$$ 
where $\mathcal I_{X_{-\infty}}$ is the defining ideal sheaf of $X_{-\infty}$, and

\item the collection of reduced and irreducible subschemes 
$\{C\}$ coincides with the images 
of the strata of $(Y, B_Y)$ that are not included in $X_{-\infty}$. 
\end{itemize}

We simply write $[X, \omega]$ to denote the above data 
$$
\left(X, \omega, f\colon (Y, B_Y)\to X\right)
$$ 
if there is no risk of confusion. 
We sometimes use ${\rm Nqlc}(X, \omega)$ to denote $X_{-\infty}$. 
\end{defn}

\begin{thm}[{\cite[Theorem 6.5.1]{fujino-fund}}]\label{thm--quasi-log-trivial}
Let $[X, \omega]$ be a quasi-log scheme and $\pi \colon X \to S$ a projective morphism to a scheme $S$. 
Let $D$ be a $\pi$-nef Cartier divisor on $X$ such that $pD- \omega$ is $\pi$-ample for some positive integer $p$ and $\mathcal{O}_{X_{-\infty}}(mD|_{X_{-\infty}})$ is $\pi|_{X_{-\infty}}$-generated for all $m\gg 0$. 
Then there is a positive integer $m_{0}$ such that $\mathcal{O}_{X}(mD)$ is $\pi$-generated for every integer $m$ such that $m\geq  m_{0}$.
\end{thm}

\begin{thm}\label{thm--eff-basepoint-free-gen-dlt}
Let $d$ and $p$ be positive integers. 
Then there exists a positive integer $m$, depending only on $d$ and $p$, satisfying the following. 
Let $(X,\Delta,\boldsymbol{\rm M})$ be a generalized dlt pair such that
\begin{itemize}
\item
${\rm dim}X=d$, 
\item
$\boldsymbol{\rm M}$ is a $\mathbb{Q}$-b-divisor and $p(K_{X}+\Delta+\boldsymbol{\rm M}_{X})$ is nef and Cartier, and
\item
there is a log resolution $f\colon \tilde{X}\to X$ of $(X,\Delta)$ such that 
\begin{itemize}
\item
$f$ is an isomorphism over an open subset $U \subset X$ containing all the generic points of generalized lc centers of $(X,\Delta,\boldsymbol{\rm M})$, 
\item
$\boldsymbol{\rm M}$ descends to $\tilde{X}$, and 
\item
writing
$K_{\tilde{X}}+\tilde{\Delta}+\boldsymbol{\rm M}_{\tilde{X}}=f^{*}(K_{X}+\Delta+\boldsymbol{\rm M}_{X})+\tilde{E}$,
where $\tilde{\Delta}\geq0$ and $\tilde{E}\geq0$ have no common components, then $\boldsymbol{\rm M}_{\tilde{X}}$ is log big with respect to $(\tilde{X},\tilde{\Delta})$. 
\end{itemize}
\end{itemize}
Then $m(K_{X}+\Delta+\boldsymbol{\rm M}_{X})$ is base point free.
\end{thm}

\begin{proof}
We prove it by induction on the dimension of $X$. 

If $(X,\Delta,\boldsymbol{\rm M})$ is generalized klt, then we may find a $\mathbb{Q}$-divisor $E \geq 0$ on $X$ such that  $\boldsymbol{\rm M}_{X}-E$ is ample and $(X,\Delta+E)$ is klt. 
Then Theorem \ref{thm--eff-basepoint-free-gen-dlt} follows from the effective base point free theorem for klt pairs \cite{kollar-eff-basepoint-free}. 
Thus, we may assume that $(X,\Delta,\boldsymbol{\rm M})$ is not generalized klt. 

\begin{step3}\label{step1--eff-basepoint-free-gen-dlt}
In this step, we define the structure of a quasi-log scheme on $X$. 

Let $A$ be an ample $\mathbb{Q}$-divisor on $X$. 
Since $\boldsymbol{\rm M}_{\tilde{X}}$ is nef and big, rescaling $A$ if necessary, we can find 
$$\tilde{R} := \tilde{R}_{1}+\tilde{R}_{2}\in |\boldsymbol{\rm M}_{\tilde{X}}-f^{*}A|_{\mathbb{Q}}$$
such that $\tilde{R}_{1}$ is ample and $\tilde{R}_{2} \geq 0$. 
We take a log resolution $g \colon  Y \to \tilde{X}$ of $(\tilde{X}, \tilde{\Delta}+\tilde{E}+\tilde{R})$, where $\tilde{\Delta}$ and $\tilde{E}$ are as in Theorem \ref{thm--eff-basepoint-free-gen-dlt}, such that $-E_{Y}$ is $g$-ample for a $g$-exceptional  $\mathbb{Q}$-divisor $E_{Y}\geq0$. 
Rescaling $E_{Y}$, we may assume that $-E_{Y}+g^{*}\tilde{R}_{1}$ is ample. 
We put $h=f \circ g \colon Y \to X$. 
Then, for every $0< t \ll 1$ we have
\begin{equation*}
\begin{split}
g^{*}\boldsymbol{\rm M}_{\tilde{X}}-th^{*}A\sim_{\mathbb{Q}}&(1-t)g^{*}\boldsymbol{\rm M}_{\tilde{X}}+t g^{*}\tilde{R}\\
=&(1-t)g^{*}\boldsymbol{\rm M}_{\tilde{X}}+t(g^{*}\tilde{R}_{1}-E_{Y})+tE_{Y}+tg^{*}\tilde{R}_{2}.
\end{split}
\end{equation*}
We define 
$\Delta_{Y}$ by $K_{Y}+\Delta_{Y}=g^{*}(K_{\tilde{X}}+\tilde{\Delta}-\tilde{E}).$ 
By choosing $t>0$ sufficiently small, we can find 
$$R_{Y}\sim_{\mathbb{Q}}g^{*}\boldsymbol{\rm M}_{\tilde{X}}-th^{*}A$$ such that
\begin{itemize} 
\item
$R_{Y}$ is the sum of an effective ample $\mathbb{Q}$-divisor and an effective $\mathbb{Q}$-divisor, 
\item
$(Y,\Delta_{Y}+R_{Y})$ is log smooth, and 
\item
$\lfloor (\Delta_{Y}+R_{Y})\rfloor = \lfloor \Delta_{Y}\rfloor$. 
\end{itemize}
By perturbation of the coefficients of $\Delta_{Y}+R_{Y}$, we may assume that no coefficient of $\Delta_{Y}+R_{Y}$ is one. 
We put 
$$B_{Y}=\Delta_{Y}+R_{Y} \quad {\rm and} \quad B=h_{*}B_{Y}.$$ 
Then $K_{Y}+B_{Y}\sim_{\mathbb{Q}}(K_{Y}+\Delta_{Y})+g^{*}\boldsymbol{\rm M}_{\tilde{X}}-th^{*}A$, and therefore we have
$$K_{Y}+B_{Y}\sim_{\mathbb{Q}} g^{*}(K_{\tilde{X}}+\tilde{\Delta}-\tilde{E})+g^{*}\boldsymbol{\rm M}_{\tilde{X}}-h^{*}(tA)=h^{*}(K_{X}+\Delta+\boldsymbol{\rm M}_{X}-tA).$$
From this, we have 
\begin{equation*}\tag{I}\label{proof--thm--eff-basepoint-free-gen-dlt--(I)}
K_{Y}+B_{Y}=h^{*}(K_{X}+B)\quad  {\rm and} \quad K_{X}+\Delta+\boldsymbol{\rm M}_{X}\sim_{\mathbb{Q}}K_{X}+B+tA.
\end{equation*} 

Recalling that no coefficient of $B_{Y}$ is one, we have 
$$\lceil(-B_{Y}^{<1})\rceil-\lfloor B_{Y}^{>1}\rfloor=-\lfloor B_{Y}\rfloor=-\lfloor \Delta_{Y}\rfloor$$
and the right hand side is the sum of $-g^{*}\lfloor \tilde{\Delta} \rfloor$ and an effective $h$-exceptional divisor. 
From this, we can check that
\begin{equation*}\tag{II}\label{proof--thm--eff-basepoint-free-gen-dlt--(II)}
h_{*}\mathcal{O}_{Y}(\lceil(-B_{Y}^{<1})\rceil-\lfloor B_{Y}^{>1}\rfloor)=\mathcal{O}_{X}(-\lfloor \Delta \rfloor).
\end{equation*} 

As in (\ref{proof--thm--eff-basepoint-free-gen-dlt--(I)}) and (\ref{proof--thm--eff-basepoint-free-gen-dlt--(II)}), we defined an effective $\mathbb{Q}$-divisor $B$ on $X$, a log resolution $h \colon Y \to X$ of $(X,B)$, and an $\mathbb{Q}$-divisor $B_{Y}$ defined with $K_{Y}+B_{Y}=h^{*}(K_{X}+B)$ such that 
\begin{itemize}
\item
$(K_{X}+\Delta+\boldsymbol{\rm M}_{X})-(K_{X}+B)$ is ample, and
\item
no coefficient of $B_{Y}$ is one and $h_{*}\mathcal{O}_{Y}(\lceil(-B_{Y}^{<1})\rceil-\lfloor B_{Y}^{>1}\rfloor)=\mathcal{O}_{X}(-\lfloor \Delta \rfloor)$. 
\end{itemize} 
\end{step3}

From now on, we put $L=K_{X}+\Delta+\boldsymbol{\rm M}_{X}$. 

\begin{step3}\label{step2--eff-basepoint-free-gen-dlt}
In this step, we prove the existence of a positive integer $m'$, depending only on $d$ and $p$, such that $\mathcal{O}_{X}(m'L)$ is generated by global sections on a neighborhood of $\lfloor \Delta \rfloor$. 

Pick any component $S$ of $\lfloor \Delta \rfloor$ and put $\tilde{S}=f^{-1}_{*}S$. 
Let $(S,\Delta_{S},\boldsymbol{\rm M}')$ be the generalized dlt pair defined with divisorial adjunction. 
We have 
$$L|_{S}\sim_{\mathbb{Q}}K_{S}+\Delta_{S}+\boldsymbol{\rm M}'_{S}.$$ 
We define $\boldsymbol{\rm M}''$ to be a $\mathbb{Q}$-b-divisor on $S$ such that  $\boldsymbol{\rm M}''_{S}=L|_{S}-(K_{S}+\Delta_{S}+\boldsymbol{\rm M}'_{S})$ and $\boldsymbol{\rm M}''$ descends to $S$. 
Then the generalized dlt pair 
$$(S,\Delta_{S},\boldsymbol{\rm M}'+\boldsymbol{\rm M}'')$$
 satisfies the condition that $p(K_{S}+\Delta_{S}+\boldsymbol{\rm M}'_{S}+\boldsymbol{\rm M}''_{S})=pL|_{S}$ is Cartier. 
It is easy to see that we may apply the induction hypothesis of Theorem \ref{thm--eff-basepoint-free-gen-dlt} to $(S,\Delta_{S},\boldsymbol{\rm M}'+\boldsymbol{\rm M}'')$. 
Since $K_{S}+\Delta_{S}+\boldsymbol{\rm M}'_{S}+\boldsymbol{\rm M}''_{S}=pL|_{S}$, applying the induction hypothesis of Theorem \ref{thm--eff-basepoint-free-gen-dlt}  to $(S,\Delta_{S},\boldsymbol{\rm M}'+\boldsymbol{\rm M}'')$, there exists a positive integer $m''$, depending only on $d$ and $p$, such that $m''L|_{S}$ is base point free. 

By our assumption of $\tilde{\Delta}$ and $\tilde{E}$, we see that $\tilde{\Delta}+\lceil \tilde{E}\rceil-\tilde{E}$ is a boundary $\mathbb{Q}$-divisor such that $\lfloor \tilde{\Delta}\rfloor =\lfloor (\tilde{\Delta}+\lceil \tilde{E}\rceil-\tilde{E}) \rfloor$
and 
$$K_{\tilde{X}}+\tilde{\Delta}+\lceil \tilde{E}\rceil-\tilde{E}+\boldsymbol{\rm M}_{\tilde{X}}+(m''p-1)f^{*}L=m''pf^{*}L+ \lceil\tilde{E}\rceil.$$
Then $\boldsymbol{\rm M}_{\tilde{X}}+(m''p-1)f^{*}L$ is nef and log big with respect to $(\tilde{X},\tilde{\Delta}+\lceil \tilde{E}\rceil-\tilde{E})$. 
From the fact and the Kodaira type vanishing theorem \cite[Lemma 1.5]{fujino-abund-logbig}, we have 
$$H^{1}(\tilde{X},\mathcal{O}_{\tilde{X}}(m''pf^{*}L+ \lceil\tilde{E}\rceil-\tilde{S}))=0.$$
By the same argument as in \cite{fujino-abund-logbig} or the proof of Lemma \ref{lem--extension-lc-center}, we obtain the diagram.
 $$
\xymatrix{
H^{0}(\tilde{X},\mathcal{O}_{\tilde{X}}(m''pf^{*}L+ \lceil\tilde{E}\rceil)) \ar@{->>}[r]&H^{0}\bigl(\tilde{S},\mathcal{O}_{\tilde{S}}\bigl((m''pf^{*}L+ \lceil\tilde{E}\rceil)|_{\tilde{S}}\bigr)\bigr)\\
H^{0}(X,\mathcal{O}_{X}(m''pL)) \ar[u]^{\simeq}\ar[r]&H^{0}(S,\mathcal{O}_{S}(m''pL|_{S}))\ar@{^{(}->}[u]
}
$$
Therefore, the base locus of $m''pL$ does not intersect $S$. 
Since $S$ is a component of $\lfloor \Delta \rfloor$ and $m''p$ does not depend on $S$,  we see that $m':=m''p$ is the desired positive integer.   
\end{step3}

\begin{step3}\label{step3--eff-basepoint-free-gen-dlt}
In this step, we construct a contraction $\pi \colon X \to Z$ induced by $m'L$ such that $m' L \sim \pi^{*}D_{Z}$ for some ample Cartier divisor $D_{Z}$ on $Z$. 

By Step \ref{step1--eff-basepoint-free-gen-dlt}, we get a quasi-log scheme $(X,K_{X}+B,h\colon (Y,B_{Y})\to X)$ such that 
$$m'L-(K_{X}+B)=(m'-1)L+(K_{X}+\Delta+\boldsymbol{\rm M}_{X})-(K_{X}+B)$$ is ample. 
By Step \ref{step2--eff-basepoint-free-gen-dlt} and $\lfloor \Delta \rfloor={\rm Nqlc}(X,K_{X}+B)$, we see that 
$$\mathcal{O}_{{\rm Nqlc}(X,K_{X}+B)}(m'L|_{{\rm Nqlc}(X,K_{X}+B)})$$
 is globally generated. 
Since $X$ is projective, applying Theorem \ref{thm--quasi-log-trivial} to $[X,K_{X}+B]$ and $m'L$, we get a contraction $\pi \colon X \to Z$ induced by $m'L$ and a positive integer $m_{0}$ such that both $m_{0}m'L$ and $(m_{0}+1)m'L$ are the pullbacks of Cartier divisors on $Z$. 
Therefore, $m'L \sim \pi^{*}D_{Z}$ for some Cartier divisor $D_{Z}$ on $Z$. 
By construction, $D_{Z}$ is ample. 
\end{step3}

\begin{step3}\label{step4--eff-basepoint-free-gen-dlt}
From this step, we prove the existence of $m$ in Theorem \ref{thm--eff-basepoint-free-gen-dlt}. 
 In this step, we treat the case when there is a component $S'$ of $\lfloor \Delta \rfloor$ dominating $Z$. 

By Lemma \ref{lem--extension-lc-center}, the restriction $\pi_{S'}:=\pi|_{S'} \colon S' \to Z$ is a contraction. 
By Step \ref{step3--eff-basepoint-free-gen-dlt}, we have $\pi_{S'*}\mathcal{O}_{S'}(m'L|_{S'})\simeq \mathcal{O}_{Z}(D_{Z})$. 
By Step \ref{step2--eff-basepoint-free-gen-dlt}, it follows that $m'L|_{S'}$ is base point free. 
Then $D_{Z}$ is base point free, and therefore $m'L$ is base point free. 
\end{step3}

\begin{step3}\label{step5--eff-basepoint-free-gen-dlt}
Finally, we treat the case when $\lfloor \Delta \rfloor$ is vertical over $Z$. 

We put $L'=m'L$. 
For each positive integer $l$, let $Z_{l}\subset Z$ be the base locus of $|lD_{Z}|$. 
By the relation $L'\sim \pi^{*}D_{Z}$, we see that  $\pi^{-1}(Z_{l})$ is disjoint from $\lfloor \Delta \rfloor$ for all $l$. 

In the rest of the proof, we closely follow \cite[Proof of Theorem 1.1]{fujino-eff-basepointfree}. 
We recall the property that the non-klt locus of $(X,B)$ and the non-lc locus of $(X,B)$ are the same and they coincide with $\lfloor B \rfloor=\lfloor \Delta \rfloor$ (see Step \ref{step1--eff-basepoint-free-gen-dlt}). 
We put 
$$A'_{l}=\frac{3}{2}lL'-(K_{X}+B)$$
 for each positive integer $l$.
By Step \ref{step1--eff-basepoint-free-gen-dlt}, it follows that $A'_{l}$ is ample for all $l$. 
Fix an arbitrary positive integer $l$, and suppose $Z_{l}\neq \emptyset$. 
Let $C_{0},C_{1},\cdots, C_{d}$ be general members of $|lD_{Z}|$, and we put 
$$C=\pi^{*}(\frac{1}{2}C_{0}+C_{1}+\cdots+C_{d}).$$ 
Then $C\sim_{\mathbb{Q}}(d+\frac{1}{2})l\pi^{*}D_{Z}$, and the non-lc locus of $(X,B+C)$ is the disjoint union of $\pi^{-1}(Z_{l})$ and $\lfloor B \rfloor$. 

Pick an arbitrary irreducible component $V$ of $\pi^{-1}(Z_{l})$ such that $\pi(V)$ has the maximal dimension. 
Then $V$ is disjoint from $\lfloor B \rfloor$. 
Let $\widetilde{h} \colon \widetilde{Y} \to X$ be a log resolution of $(X,B+C)$ such that $\widetilde{h}^{-1}(\pi^{-1}(\pi(V))$ is an snc divisor. 
We may write
$$K_{\widetilde{Y}}=\widetilde{h}^{*}(K_{X}+B)+\sum_{i}e_{i}\widetilde{E}_{i}$$
for some $e_{i} \in \mathbb{Q}$. 
Let $\gamma$ be the lc threshold of $C$ with respect to $(X,B)$ over the generic point of $\pi(V)$. 
It follows that $\gamma \in (0,1)\cap \mathbb{Q}$ because $(X,B)$ is klt on $X\setminus \lfloor B \rfloor$ and $\pi(V)$ is disjoint from $\pi(\lfloor B \rfloor)$. 
We may write $\gamma \widetilde{h}^{*}C=\sum_{i}\gamma_{i}\widetilde{E}_{i}$ for some $\gamma_{i}\in \mathbb{Q}_{\geq 0}$. 
Then we have 
\begin{equation*}
\begin{split}
\widetilde{h}^{*}\pi^{*}((d+2)lD_{Z})\sim_{\mathbb{Q}}&\frac{3}{2}\widetilde{h}^{*}(lL')+\widetilde{h}^{*}\pi^{*}((d+\frac{1}{2})lD_{Z}) \\
\sim_{\mathbb{Q}}&\widetilde{h}^{*}(K_{X}+B+A'_{l})+\gamma \widetilde{h}^{*}C+(1-\gamma)\widetilde{h}^{*}C\\
\sim_{\mathbb{Q}}&K_{\widetilde{Y}}+\sum_{i}(\gamma_{i}-e_{i})\widetilde{E}_{i}+\widetilde{h}^{*}A'_{l}+(1-\gamma)\widetilde{h}^{*}C.
\end{split}
\end{equation*}
We can write 
$$\sum_{i}\lfloor (\gamma_{i}-e_{i})\rfloor \widetilde{E}_{i}=\widetilde{F}+\widetilde{G}_{1}+\widetilde{G}_{2}-\widetilde{H}$$
where $\widetilde{F}$, $\widetilde{G}_{1}$, $\widetilde{G}_{2}$, and $\widetilde{H}$ are all effective and they have no common components with each other such that
\begin{itemize}
\item
the image of every component of $\widetilde{F}$ by $\pi \circ \widetilde{h}$ is $\pi(V)$, 
\item
the image of every component of $\widetilde{G}_{1}$ by $\pi \circ \widetilde{h}$ does not contain $\pi(V)$,
\item
the image of every component of $\widetilde{G}_{2}$ by $\pi \circ \widetilde{h}$ contains $\pi(V)$ but does not coincide with $\pi(V)$, and 
\item
$\widetilde{H}$ is $\widetilde{h}$-exceptional. 
\end{itemize} 
Note that $\widetilde{h}^{-1}_{*}\lfloor B \rfloor \subset {\rm Supp}\widetilde{G}_{1}$, $\widetilde{F}$ is a non-zero reduced divisor by the definition of $\gamma$, and 
$\widetilde{G}_{2}$ is reduced because every irreducible component of $(\pi \circ \widetilde{h})(\widetilde{G}_{2})$ intersects $\pi(V)$ and $(\widetilde{Y}, \sum_{i} (\gamma_{i}-e_{i}) \widetilde{E}_{i})$ is sub-lc over a neighborhood of the generic point of $\pi(V)$. 
By taking blow-ups along lc centers of $(\widetilde{Y},\widetilde{G}_{2})$ if necessary, we may assume that any lc center $T$ of $(\widetilde{Y},\widetilde{G}_{2})$ does not satisfy $(\pi \circ \widetilde{h})(T) \subset \pi(V)$. 

For each $j \in \mathbb{Z}_{\geq 0}$, we define 
$$\widetilde{N}_{j}:=\widetilde{h}^{*}\pi^{*}((j+d+2)lD_{Z})+\widetilde{H}-\widetilde{G}_{1},$$
and consider the exact sequence 
$$0 \longrightarrow \mathcal{O}_{\widetilde{Y}}(\widetilde{N}_{j}-\widetilde{F}) \longrightarrow \mathcal{O}_{\widetilde{Y}}(\widetilde{N}_{j})\longrightarrow \mathcal{O}_{\widetilde{F}}(\widetilde{N}_{j}|_{\widetilde{F}})\longrightarrow 0.$$
By construction, we have
\begin{equation*}\tag{$\spadesuit$}\label{proof--spade}
\widetilde{N}_{j}-\widetilde{F}\sim_{\mathbb{Q}}K_{\widetilde{Y}}+\sum_{i}\{\gamma_{i}-e_{i}\}\widetilde{E}_{i}+\widetilde{G}_{2}+\widetilde{h}^{*}A'_{l}+(1-\gamma)\widetilde{h}^{*}C+jl\widetilde{h}^{*}\pi^{*}D_{Z},
\end{equation*} and no lc center $T$ of the lc pair $\bigl(\widetilde{Y},\sum_{i}\{\gamma_{i}-e_{i}\}\widetilde{E}_{i}+\widetilde{G}_{2}\bigr)$ satisfies $(\pi \circ \widetilde{h})(T) \subset \pi(V)$. 
Here, the log canonicity of the pair $\bigl(\widetilde{Y},\sum_{i}\{\gamma_{i}-e_{i}\}\widetilde{E}_{i}+\widetilde{G}_{2}\bigr)$ follows from the facts that $\bigl(\widetilde{Y},{\rm Supp}(\sum_{i}\widetilde{E}_{i}+\widetilde{G}_{2})\bigr)$ is log smooth, $\widetilde{G}_{2}$ is reduced, and $\widetilde{G}_{2}$ and $\sum_{i}\{\gamma_{i}-e_{i}\}\widetilde{E}_{i}$ have no common components. 
By \cite[Theorem 6.3 (i)]{fujino-fund}, the connecting morphism
$$(\pi \circ \widetilde{h})_{*}\mathcal{O}_{\widetilde{F}}(\widetilde{N}_{j}|_{\widetilde{F}}) \longrightarrow R^{1}(\pi \circ \widetilde{h})_{*}\mathcal{O}_{\widetilde{Y}}(\widetilde{N}_{j}-\widetilde{F})$$
is a zero map. 
Thus, we get the following exact sequence
$$0 \longrightarrow (\pi \circ \widetilde{h})_{*}\mathcal{O}_{\widetilde{Y}}(\widetilde{N}_{j}-\widetilde{F}) \longrightarrow (\pi \circ \widetilde{h})_{*}\mathcal{O}_{\widetilde{Y}}(\widetilde{N}_{j})\longrightarrow (\pi \circ \widetilde{h})_{*}\mathcal{O}_{\widetilde{F}}(\widetilde{N}_{j}|_{\widetilde{F}})\longrightarrow 0.$$
Since $A'_{l}$ is ample, there is an $\widetilde{A}\sim_{\mathbb{Q}}\widetilde{h}^{*}A'_{l}$ such that $\bigl(\widetilde{Y}, \sum_{i}\{\gamma_{i}-e_{i}\}\widetilde{E}_{i}+\widetilde{G}_{2}+\widetilde{A}\bigr)$ is a log smooth lc pair. 
Because $(1-\gamma)\widetilde{h}^{*}C+jl\widetilde{h}^{*}\pi^{*}D_{Z}$ is the pullback of an ample $\mathbb{Q}$-divisor on $Z$, we have 
$$H^{1}(Z, (\pi \circ \widetilde{h})_{*}\mathcal{O}_{\widetilde{Y}}(\widetilde{N}_{j}-\widetilde{F}))=0$$ by (\ref{proof--spade}) and \cite[Theorem 6.3 (ii)]{fujino-fund}. 
From this, we get a surjective morphism
\begin{equation*}\tag{$\clubsuit$}\label{proof--club}
H^{0}(\widetilde{Y},\mathcal{O}_{\widetilde{Y}}(\widetilde{N}_{j}))\longrightarrow H^{0}(\widetilde{F},\mathcal{O}_{\widetilde{F}}(\widetilde{N}_{j}|_{\widetilde{F}})).
\end{equation*}
From the definitions of $\gamma$ and $\widetilde{F}$, it follows that no component of $\widetilde{F}$ is a component of $\sum_{i}\{\gamma_{i}-e_{i}\}\widetilde{E}_{i}$. 
By (\ref{proof--spade}) and the construction of $\widetilde{A}$, we have
$$\widetilde{N}_{j}|_{\widetilde{F}} \sim_{\mathbb{Q}} K_{\widetilde{F}}+\Bigl(\sum_{i}\{\gamma_{i}-e_{i}\}\widetilde{E}_{i}+\widetilde{G}_{2}+\widetilde{A}\Bigr)\Big{|}_{\widetilde{F}}+\Bigl((1-\gamma)\widetilde{h}^{*}C+jl\widetilde{h}^{*}\pi^{*}D_{Z}\Bigr)\Big{|}_{\widetilde{F}}$$
and $\bigl(\widetilde{F}, (\sum_{i}\{\gamma_{i}-e_{i}\}\widetilde{E}_{i}+\widetilde{G}_{2}+\widetilde{A})|_{\widetilde{F}}\bigr)$ is a simple normal crossing pair.  
By the vanishing theorem for simple normal crossing pairs \cite[Theorem 3.2 (b)]{fujino-eff-basepointfree}, we see that 
$$H^{q}(\pi(V),(\pi \circ \widetilde{h})_{*}\mathcal{O}_{\widetilde{F}}(\widetilde{N}_{j}|_{\widetilde{F}}))=0$$
 for all $q>0$. 
Thus, we have 
\begin{equation*}
\begin{split}{\rm dim}H^{0}(\widetilde{F},\mathcal{O}_{\widetilde{F}}(\widetilde{N}_{j}|_{\widetilde{F}}))=&{\rm dim}H^{0}\bigl(\pi(V),(\pi \circ \widetilde{h})_{*}\mathcal{O}_{\widetilde{F}}(\widetilde{N}_{j}|_{\widetilde{F}})\bigr)\\
=&\chi\bigl(\pi(V),(\pi \circ \widetilde{h})_{*}\mathcal{O}_{\widetilde{F}}(\widetilde{N}_{j}|_{\widetilde{F}})\bigr).\end{split}
\end{equation*} 
By definition of $\widetilde{N}_{j}$, we have 
\begin{equation*}
\begin{split}
(\pi \circ \widetilde{h})_{*}\mathcal{O}_{\widetilde{F}}(\widetilde{N}_{j}|_{\widetilde{F}})=(\pi \circ \widetilde{h})_{*}\mathcal{O}_{\widetilde{F}}(\widetilde{H}|_{\widetilde{F}}-\widetilde{G}_{1}|_{\widetilde{F}})\otimes \mathcal{O}_{\pi(V)}\bigl((j+d+2)lD_{Z}|_{\pi(V)}\bigr),
\end{split}
\end{equation*} 
and $(\pi \circ \widetilde{h})_{*}\mathcal{O}_{\widetilde{F}}(\widetilde{H}|_{\widetilde{F}}-\widetilde{G}_{1}|_{\widetilde{F}})$ is a non-zero sheaf because $\widetilde{G}_{1}|_{\widetilde{F}}$ is zero at the generic point of $\pi(V)$ (see the definition of $\widetilde{G}_{1}$). 
Therefore, ${\rm dim}H^{0}(\widetilde{F},\mathcal{O}_{\widetilde{F}}(\widetilde{N}_{j}|_{\widetilde{F}}))$ is a non-zero polynomial of degree at most $d$ with one variable $j$. 
Then there is $0\leq j' \leq d$ such that 
$$H^{0}(\widetilde{F},\mathcal{O}_{\widetilde{F}}(\widetilde{N}_{j'}|_{\widetilde{F}}))\neq0,$$
 so $H^{0}(\widetilde{Y},\mathcal{O}_{\widetilde{Y}}(\widetilde{N}_{j'}))$ has an element not vanishing on $\widetilde{F}$ by (\ref{proof--club}). 
 
Since $\widetilde{N}_{j'}=\widetilde{h}^{*}\pi^{*}((j'+d+2)lD_{Z})+\widetilde{H}-\widetilde{G}_{1}$ and ${\rm Supp}\widetilde{G}_{1} \not\supset \widetilde{F}$, we see that 
$$H^{0}(\widetilde{Y},\mathcal{O}_{\widetilde{Y}}(\widetilde{h}^{*}\pi^{*}((j'+d+2)lD_{Z})+\widetilde{H}))$$
 has an element not vanishing on $\widetilde{F}$. 
Since $\widetilde{H}$ is $\widetilde{h}$-exceptional, it follows that
$$H^{0}(X,\mathcal{O}_{X}(\pi^{*}((j'+d+2)lD_{Z})))$$ has an element not vanishing on $\widetilde{h}(\widetilde{F})$. 
This implies that the base locus of $(j'+d+2)lD_{Z}$ does not contain $\pi(V)$. 

Recall that $V$ is an arbitrary irreducible component of $\pi^{-1}(Z_{l})$ such that $\pi(V)$ has the maximal dimension. 
From the above argument, it follows that $Z_{(2d+2)!\cdot l}$ does not contain any irreducible component of $Z_{l}$ of the maximal dimension.
Thus, we have
$${\rm dim}Z_{(2d+2)!\cdot l}<{\rm dim}Z_{l}$$
for all $l \in \mathbb{Z}_{>0}$. 
Therefore, $mL$ is base point free where $m$ is defined to be $((2d+2)!)^{d}\cdot m'$. 
\end{step3}
We finish the proof of Theorem \ref{thm--eff-basepoint-free-gen-dlt}. 
\end{proof}

\begin{rem}
By arguing inductively, for fixed integers $d$ and $p$ and every generalized dlt pair $(X,\Delta,\boldsymbol{\rm M})$ as in Theorem \ref{thm--eff-basepoint-free-gen-dlt}, we can check that $((2d+2)!)^{d^2}p(K_{X}+\Delta+\boldsymbol{\rm M}_{X})$ is base point free. 
However, compared to \cite[Remark 1.2]{fujino-eff-basepointfree}, it looks far from optimal. 
\end{rem}

\begin{lem}\label{lem--bound-veryampleindex}
Let $d$ and $p$ be positive integers, and let $v$ be a positive real number. 
Then there exists a positive integer $m$, depending only on $d$, $p$, and $v$, satisfying the following. 
Let $(X,\Delta,\boldsymbol{\rm M})$ be a generalized lc pair such that
\begin{itemize}
\item
${\rm dim}X=d$, 
\item
$p\boldsymbol{\rm M}$ is b-Cartier,
\item
$p(K_{X}+\Delta+\boldsymbol{\rm M}_{X})$ is an ample and base point free Cartier divisor, and 
\item
${\rm vol}(X,K_{X}+\Delta+\boldsymbol{\rm M}_{X}) \leq v$. 
\end{itemize}
Then $m(K_{X}+\Delta+\boldsymbol{\rm M}_{X})$ is very ample, and putting $R_{l}=H^{0}(X,\mathcal{O}_{X}(\lfloor l(K_{X}+\Delta+\boldsymbol{\rm M}_{X})\rfloor))$ for every $l \in \mathbb{Z}_{\geq 0}$ then $\bigoplus_{l \in \mathbb{Z}_{\geq 0}} R_{lm}$ is generated by $R_{m}$ as a graded $\mathbb{C}$-algebra. 
\end{lem}

\begin{proof}
If Lemma \ref{lem--bound-veryampleindex} does not hold, then there is a sequence of generalized lc pairs $\{(X_{i},\Delta_{i},\boldsymbol{\rm M}^{i})\}_{i\geq 1}$ as in Lemma \ref{lem--bound-veryampleindex} such that if $m_{i}$ is the smallest positive integer such that $m_{i}(K_{X_{i}}+\Delta_{i}+\boldsymbol{\rm M}^{i}_{X_{i}})$ is very ample and $H^{0}(X_{i},\mathcal{O}_{X_{i}}( m_{i}(K_{X_{i}}+\Delta_{i}+\boldsymbol{\rm M}^{i}_{X_{i}})))$ generates $\bigoplus_{l \in \mathbb{Z}_{\geq 0}}H^{0}(X_{i},\mathcal{O}_{X_{i}}(lm_{i}(K_{X_{i}}+\Delta_{i}+\boldsymbol{\rm M}^{i}_{X_{i}})))$, then ${\rm lim}_{i \to \infty}m_{i}=\infty$. 

Since $p\boldsymbol{\rm M}^{i}$ is b-Cartier and $p(K_{X_{i}}+\Delta_{i}+\boldsymbol{\rm M}^{i}_{X_{i}})$ is Cartier, $p\Delta_{i}$ is a Weil divisor for every $i$. 
By the effective birationality for generalized pairs \cite[Theorem 1.3]{bz}, there exists a positive integer $p'$, depending only on $d$ and $p$, such that $|pp'(K_{X_{i}}+\Delta_{i}+\boldsymbol{\rm M}^{i}_{X_{i}})|$ defines a birational morphism. 
Then we may find $G_{i} \in |pp'(K_{X_{i}}+\Delta_{i}+\boldsymbol{\rm M}^{i}_{X_{i}})|$. 
Note that 
\begin{itemize}
\item
$G_{i}$ is Cartier and base point free, 
\item
${\rm vol}(X_{i},G_{i})\leq p^{d}p'^{d}v$, and 
\item
$G_{i}-(K_{X_{i}}+\Delta_{i})$ is the pushdown of a pseudo-effective $\mathbb{Q}$-divisor. 
\end{itemize}
By \cite[Proposition 4.4]{birkar-compl} and replacing $\{(X_{i},\Delta_{i},\boldsymbol{\rm M}^{i})\}_{i\geq 1}$ with a subsequence, we can find a projective morphism $h\colon V \to T$ of varieties, a reduced divisor $D$ on $V$, and a positive real number $c$, depending only on $d$, $p$, and $v$, such that for every $i$, there is a closed point $t_{i} \in T$ with the fibers $V_{t_{i}}$ and $D_{t_{i}}$ satisfying the following. 
\begin{itemize}
\item
$(V_{t_{i}}, D_{t_{i}})$ is log smooth, 
\item
there is a birational map $\phi_{i} \colon X_{i} \dashrightarrow V_{t_{i}}$ such that if $E_{i}$ is the sum of $\phi_{i*}G_{i}$ and the reduced $\phi_{i}^{-1}$-exceptional divisor, then ${\rm Supp}E_{i}\subset D_{t_{i}}$, and
\item
if $f_{i} \colon Y \to X_{i}$ and $f'_{i} \colon Y \to V_{t_{i}}$ is a common resolution of $\phi_{i}$, then $f_{i}^{*}G_{i} \sim_{\mathbb{Q},V_{t_{i}}}0$ and all the coefficients of $f'_{i*}f_{i}^{*}G_{i}$ are at most $c$. 
\end{itemize}
Shrinking $T$ and replacing $\{(X_{i},\Delta_{i},\boldsymbol{\rm M}^{i})\}_{i\geq 1}$ with a subsequence if necessary, we may assume that the set $\{t_{i}\}_{i\geq 1}$ is dense in $T$. 

Put $D'_{i}=f'_{i*}f_{i}^{*}G_{i}$ for each $i$. 
Then $f_{i}^{*}G_{i}=f'^{*}_{i}D'_{i}$ by the third property. 
Since $G_{i}$ is ample, the birational map $\varphi_{i} := \phi_{i}^{-1}\colon V_{t_{i}} \dashrightarrow X_{i}$ is a morphism. 
Taking a log resolution of $(V,D)$ and shrinking $T$, we may assume that the morphism $h\colon (V,D) \to T$ is log smooth and $T$ is affine. 
Moreover, taking an appropriate \'etale base change of $h$, we may assume that the restriction of every component of $D$ to ${V_{t}}$ is irreducible for every closed point $t \in T$ with the fiber $V_{t}$. 

Since $\varphi^{*}_{i}G_{i}$ is an effective Cartier divisor on $V_{t_{i}}$ such that ${\rm Supp}\varphi^{*}_{i}G_{i}\subset {\rm Supp}D|_{V_{t_{i}}}$ and all the coefficients of $\varphi^{*}_{i}G_{i}$ is at most $c$, the coefficients of $\varphi^{*}_{i}G_{i}$ belong to a finite set depending only on $d$, $p$, and $v$. 
Therefore, replacing $\{(X_{i},\Delta_{i},\boldsymbol{\rm M}^{i})\}_{i\geq 1}$ by a subsequence, we may find a Cartier divisor $\tilde{D}$ on $V$ such that for every $i$, we have $\tilde{D}|_{V_{t_{i}}}=\varphi^{*}_{i}G_{i}$. 

We consider the morphism 
$$h^{*}h_{*}\mathcal{O}_{V}(\tilde{D}) \to \mathcal{O}_{V}(\tilde{D}).$$
Then there is an open set $U\subset T$ such that for any $t' \in U$, the morphism
$$h_{*}\mathcal{O}_{V}(\tilde{D})\otimes k(t')\to H^{0}(V_{t'}, \mathcal{O}_{V_{t'}}(\tilde{D}|_{V_{t'}}))$$ 
is an isomorphism, where $k(t')$ is the quotient field of $\mathcal{O}_{T,t'}$ (see  \cite[III, Theorem 12.8]{hartshorne} and \cite[III, Corollary 12.9]{hartshorne}). 
Since $\mathcal{O}_{V_{t_{i}}}(\varphi^{*}_{i}G_{i})$ is globally generated, the morphism 
$$h^{*}h_{*}\mathcal{O}_{V}(\tilde{D})\otimes k(t') \to  \mathcal{O}_{V}(\tilde{D})\otimes k(t')$$
is surjective for all $t' \in U \cap \{t_{i}\}_{i \geq 1}$. 
Since $\{t_{i}\}_{i \geq 1}$ is dense in $T$, we see that $U \cap \{t_{i}\}_{i \geq 1}$ is not empty. 
From this fact, shrinking $T$ and replacing $\{(X_{i},\Delta_{i},\boldsymbol{\rm M}^{i})\}_{i\geq 1}$, we may assume that $\tilde{D}$ is base point free over $T$. 

Let $\tau \colon V \to W$ be the contraction over $T$ induced by $\tilde{D}$. 
Then there is a Cartier divisor $H$, which is ample over $T$, such that $\tilde{D} \sim \tau^{*}H$. 
By shrinking $T$ and replacing $\{(X_{i},\Delta_{i},\boldsymbol{\rm M}^{i})\}_{i\geq 1}$ by a subsequence, we may assume that for every closed point $t \in T$ the fiber $W_{t}$ is normal. 
We consider the morphisms $\tau_{i}:=\tau|_{V_{t_{i}}}\colon V_{t_{i}} \to W_{t_{i}}$ and $\varphi_{i}\colon V_{t_{i}} \to X_{i}$ for each $i$. 
Since $\tau^{*}_{i}H|_{W_{t_{i}}}\sim \tilde{D}|_{V_{t_{i}}}=\varphi^{*}_{i}G_{i}$ and both $H|_{W_{t_{i}}}$ and $G_{i}$ are ample, we see that ${\rm dim}W_{t_{i}} ={\rm dim}X_{i}$. 
This shows that $\tau_{i} \colon V_{t_{i}} \to W_{t_{i}}$ is birational, hence $\tau_{i}$ is a contraction. 
Since $\tau^{*}_{i}H|_{W_{t_{i}}}\sim \varphi^{*}_{i}G_{i}$ and both $H|_{W_{t_{i}}}$ and $G_{i}$ are ample, we see that $W_{t_{i}} \simeq X_{i}$. 

By this discussion, we get 
\begin{itemize}
\item
a projective morphism $W \to T$ from a normal variety $W$ to an affine variety $T$, and 
\item
an ample Cartier divisor $H$ 
\end{itemize}
such that for every $i$, there exist a closed point $t_{i} \in T$ and an isomorphism $\psi_{i}\colon W_{t_{i}}\to  X_{i}$ such that $\psi^{*}_{i}G_{i} \sim H|_{W_{t_{i}}}$. 
Shrinking $T$ and replacing $\{(X_{i},\Delta_{i},\boldsymbol{\rm M}^{i})\}_{i\geq 1}$, we may assume that the morphism $W \to T$ is flat.

Let $m$ be a positive integer such that 
\begin{itemize}
\item
$mH$ is very ample over $T$, 
\item
$H^{0}(W, \mathcal{O}_{W}(mH))$ generates $\bigoplus_{l \in \mathbb{Z}_{\geq 0}}H^{0}(W, \mathcal{O}_{W}(lmH))$, and 
\item
$H^{q}(W, \mathcal{O}_{W}(lH))=0$ for all $l \geq m$ and $q>0$. 
\end{itemize}
Then $H^{0}(W, \mathcal{O}_{W}(lmH))\otimes k(t) \to H^{0}(W_{t}, \mathcal{O}_{W_{t}}(lmH|_{W_{t}}))$
 is surjective for every closed point $t \in T$. 
Therefore, for every closed point $t \in T$ and every $l \in \mathbb{Z}_{\geq0}$, we see that
$$H^{0}(W_{t}, \mathcal{O}_{W_{t}}(mH|_{W_{t}}))\otimes H^{0}(W_{t}, \mathcal{O}_{W_{t}}(l mH|_{W_{t}})) \longrightarrow H^{0}(W_{t}, \mathcal{O}_{W_{t}}((1+l)mH|_{W_{t}}))$$
is surjective. 
Therefore, we see that $mG_{i}$ is very ample and $H^{0}(X_{i},\mathcal{O}_{X_{i}}(mG_{i}))$ generates the graded ring $\bigoplus_{l \in \mathbb{Z}_{\geq 0}}H^{0}(X_{i},\mathcal{O}_{X_{i}}(mlG_{i}))$ for every $i$. 
This is a contradiction since we have $mG_{i}\sim mpp'(K_{X_{i}}+\Delta_{i}+\boldsymbol{\rm M}^{i}_{X_{i}})$ for all $i$. 
Thus, Lemma \ref{lem--bound-veryampleindex} holds. 
\end{proof}

\begin{proof}[Proof of Theorem \ref{thm--base-boundedness-intro}]
We will freely use the notations as in Theorem \ref{thm--base-boundedness-intro}, and we put $\Delta'=\phi_{*}\Delta$ and $M'=\phi_{*}M$. 
Since $M$ is nef and log big with respect to $(X,\Delta)$, we may apply Theorem \ref{thm--eff-basepoint-free-gen-dlt} to the generalized dlt pair $(X',\Delta',\overline{M})$. 
By Theorem \ref{thm--eff-basepoint-free-gen-dlt}, there exists $m'$, depending only on $d$ and $p$, such that $m'(K_{X'}+\Delta'+M')$ is base point free. 
Let $\pi \colon X' \to Z$ be the Stein factorization of the morphism defined with $|m'(K_{X'}+\Delta'+M')|$.
Then $Z$ is nothing but ${\boldsymbol{\rm Proj}}\bigoplus_{l \in \mathbb{Z}_{\geq 0}}R_{l}$ in Theorem \ref{thm--base-boundedness-intro}. 
There is an ample and base point free Cartier divisor $H_{Z}$ on $Z$ such that $$m'(K_{X'}+\Delta'+M') \sim \pi^{*}H_{Z}.$$ 

Since $pM'$ and $p(K_{X'}+\Delta'+M')$ are Weil divisors, $p\Delta'$ is a Weil divisor. 
Applying Theorem \ref{thm--eff-dlt-general} to $\pi\colon (X',\Delta',\overline{M})\to Z$, we can find positive integers $n$ and $q$, depending only on $d$ and $p$, such that $Z$ has the structure of a generalized lc pair $(Z,\Delta_{Z},\boldsymbol{\rm N})$ with
\begin{itemize}
\item
$n(K_{X'}+\Delta'+M') \sim n \pi^{*}(K_{Z}+\Delta_{Z}+\boldsymbol{\rm N}_{Z})$, and 
\item
$q \boldsymbol{\rm N}$ is b-Cartier.
\end{itemize}
By the first property, we have 
$$nm'\pi^{*}(K_{Z}+\Delta_{Z}+\boldsymbol{\rm N}_{Z})\sim nm'(K_{X'}+\Delta'+M') \sim n\pi^{*}H_{Z}.$$
By \cite[Proposition 5.3]{filipazzimoraga}, we see that $nm'(K_{Z}+\Delta_{Z}+\boldsymbol{\rm N}_{Z}) \sim nH_{Z}$. 
In particular, we see that $nm'(K_{Z}+\Delta_{Z}+\boldsymbol{\rm N}_{Z})$ is ample and base point free.  

Since we have 
$${\rm vol}(Z, K_{Z}+\Delta_{Z}+\boldsymbol{\rm N}_{Z})={\rm Ivol}(K_{X}+\Delta+M)\leq v,$$
 we may apply Lemma \ref{lem--bound-veryampleindex} to $(Z,\Delta_{Z},\boldsymbol{\rm N})$. 
There is a positive integer $m''$, depending only on $d$, $nm'$, and $v$, such that
\begin{itemize}
\item
$m''(K_{Z}+\Delta_{Z}+\boldsymbol{\rm N}_{Z})$ is very ample, and 
\item
putting $R'_{l}=H^{0}(Z, \mathcal{O}_{Z}(lm''(K_{Z}+\Delta_{Z}+\boldsymbol{\rm N}_{Z})))$, then $R'_{1}$ generates $\bigoplus_{l \in \mathbb{Z}_{\geq 0}}R'_{l}$. 
\end{itemize}

We define $m=nm'm''$. 
By construction, we can check that 
\begin{itemize}
\item
$m$ depends only on $d$, $p$, and $v$, 
\item
$m(K_{Z}+\Delta_{Z}+\boldsymbol{\rm N}_{Z})$ is very ample, and 
\item
$H^{0}(Z, \mathcal{O}_{Z}(m(K_{Z}+\Delta_{Z}+\boldsymbol{\rm N}_{Z})))$ generates $\bigoplus_{l \in \mathbb{Z}_{\geq 0}}H^{0}(Z, \mathcal{O}_{Z}(lm(K_{Z}+\Delta_{Z}+\boldsymbol{\rm N}_{Z})))$ as a graded $\mathbb{C}$-algebra. 
\end{itemize}
Now we have
\begin{equation*}
\begin{split}
H^{0}(X,\mathcal{O}_{X}(\lfloor lm(K_{X}+\Delta+M)\rfloor))& \simeq H^{0}(X',\mathcal{O}_{X'}( lm(K_{X'}+\Delta'+M')))\\
& \simeq H^{0}(Z, \mathcal{O}_{Z}(lm(K_{Z}+\Delta_{Z}+\boldsymbol{\rm N}_{Z})))
\end{split}
\end{equation*}
for every positive integer $l$. 
In this way, we see that $\bigoplus_{l \in \mathbb{Z}_{\geq 0}} R_{lm}$ 
 is generated by $R_{m}$ as a graded $\mathbb{C}$-algebra, where $R_{lm}=H^{0}(X,\mathcal{O}_{X}(\lfloor lm(K_{X}+\Delta+M)\rfloor))$. 
Furthermore, the inequality 
${\rm vol}(Z,m(K_{Z}+\Delta_{Z}+\boldsymbol{\rm N}_{Z})) \leq m^{d}v$ 
shows that $Z={\boldsymbol{\rm Proj}}\bigoplus_{l \in \mathbb{Z}_{\geq 0}}R_{l}$ belongs to a bounded family $\mathfrak{F}$ that depends only on $d$, $p$, and $v$. 
We finish the proof. 
\end{proof}

Finally, we introduce an application of Theorem \ref{thm--eff-basepoint-free-gen-dlt}. 

\begin{thm}\label{thm--eff-basepoint-free-gen-dlt-general}
Let $d$ and $p$ be positive integers. 
Then there exists a positive integer $m$, depending only on $d$ and $p$, satisfying the following. 
Let $(X,\Delta,\boldsymbol{\rm M})$ be a generalized dlt pair such that
\begin{itemize}
\item
${\rm dim}X=d$, 
\item
$\boldsymbol{\rm M}$ is a $\mathbb{Q}$-b-divisor and $p(K_{X}+\Delta+\boldsymbol{\rm M}_{X})$ is nef and Cartier, and
\item
$K_{X}+\Delta+\boldsymbol{\rm M}_{X}$ is log big with respect to $(X,\Delta,\boldsymbol{\rm M})$. 
\end{itemize}
Then $m(K_{X}+\Delta+\boldsymbol{\rm M}_{X})$ is base point free.
\end{thm}

\begin{proof}
By Theorem \ref{thm--equiv-gendlt}, there is a log resolution $f\colon \tilde{X} \to X$ of $(X,\Delta)$ such that $\boldsymbol{\rm M}$ descends to $\tilde{X}$ and $f$ is an isomorphism over an open subset $U$ containing all the generic points of generalized lc centers of $(X,\Delta,\boldsymbol{\rm M})$. 
Set 
$D=K_{X}+\Delta+\boldsymbol{\rm M}_{X}$
 and consider the generalized dlt pair $(X,\Delta, \boldsymbol{\rm M}+\overline{D})$. 
We can write
$$K_{\tilde{X}}+\tilde{\Delta}+\boldsymbol{\rm M}_{\tilde{X}}+f^{*}D=f^{*}(K_{X}+\Delta+\boldsymbol{\rm M}_{X}+D)+\tilde{E}$$
with $\tilde{\Delta}\geq 0$ and $\tilde{E}\geq 0$ having no common components. 
Since $D$ is nef and log big with respect to $(X,\Delta,\boldsymbol{\rm M})$ and $f$ is an isomorphism over the genetic points of all generalized lc centers of $(X,\Delta,\boldsymbol{\rm M})$, we see that $\boldsymbol{\rm M}_{\tilde{X}}+f^{*}D$ is nef and log big with respect to $(\tilde{X},\tilde{\Delta})$. 
Thus we may apply Theorem \ref{thm--eff-basepoint-free-gen-dlt}. 
There exists $m'$, depending only on $d$ and $p$, such that $m'(K_{X}+\Delta+\boldsymbol{\rm M}_{X}+D)$ is base point free. 
Then $2m'(K_{X}+\Delta+\boldsymbol{\rm M}_{X})$ is base point free. 
Hence $m:=2m'$ is the desired positive integer.
\end{proof}

\begin{cor}\label{cor--gendlt-bounded}
Let $d$ and $p$ be positive integers, and let $v$ be a positive real number. 
Then there exists a positive integer $m$, depending only on $d$, $p$, and $v$, satisfying the following. 
Let $(X,\Delta,\boldsymbol{\rm M})$ be a generalized dlt pair such that
\begin{itemize}
\item
${\rm dim}X=d$, 
\item
$p\boldsymbol{\rm M}$ is b-Cartier and $p(K_{X}+\Delta+\boldsymbol{\rm M}_{X})$ is  Cartier, 
\item
$K_{X}+\Delta+\boldsymbol{\rm M}_{X}$ is ample, and
\item
${\rm vol}(X,K_{X}+\Delta+\boldsymbol{\rm M}_{X}) \leq v$. 
\end{itemize}
Then $m(K_{X}+\Delta+\boldsymbol{\rm M}_{X})$ is very ample. 
In particular, $(X,\Delta,\boldsymbol{\rm M})$ belongs to a bounded family $\mathfrak{F}$ depending only on $d$, $p$, and $v$ in the sense of Birkar \cite{birkar-nefpart}. 
\end{cor}

\begin{proof}
By Theorem \ref{thm--eff-basepoint-free-gen-dlt-general}, there exists a positive integer $n$, depending only on $d$ and $p$, such that $n(K_{X}+\Delta+\boldsymbol{\rm M}_{X})$ is base point free for every generalized dlt pair $(X,\Delta,\boldsymbol{\rm M})$ as in Corollary \ref{cor--gendlt-bounded}. 
Then Corollary \ref{cor--gendlt-bounded} follows from Lemma \ref{lem--bound-veryampleindex}. 
\end{proof}

\section{Appendix: On definition of generalized dlt pairs}\label{sec6}

The goal of this appendix is to prove that the definition of generalized dlt pairs by Birkar \cite{birkar-compl} and that by Han--Li \cite{hanli} are equivalent when the nef part of a generalized pair is a finite $\mathbb{R}_{>0}$-linear combination of b-nef $\mathbb{Q}$-b-Cartier $\mathbb{Q}$-b-divisors. 

\begin{thm}\label{thm--equiv-gendlt}
The following three conditions are equivalent  
for every generalized lc pair $(X,\Delta,\boldsymbol{\rm M})/Z$ such that $\boldsymbol{\rm M}$ is a finite $\mathbb{R}_{>0}$-linear combination of b-nef$/Z$ $\mathbb{Q}$-b-Cartier $\mathbb{Q}$-b-divisors. 
\begin{enumerate}
\item\label{1}
For any generic point $\eta$ of any generalized lc center of $(X,\Delta,\boldsymbol{\rm M})/Z$, $(X,\Delta)$ is log smooth near $\eta$ and $\boldsymbol{\rm M}$ descends to $X$ on a neighborhood of $\eta$ (\cite[2.13]{birkar-compl}). 
\item\label{2}
There is an open subset $U \subset X$ such that $(U,\Delta|_{U})$ is log smooth, $U$ contains the generic point of any generalized lc center of $(X,\Delta,\boldsymbol{\rm M})/Z$, and the generic point of any generalized lc center of $(X,\Delta,\boldsymbol{\rm M})/Z$ is the generic point of an lc center of $(U,\Delta|_{U})$ (\cite[Definition 2.3]{hanli}).
\item\label{3}
There is a log resolution $f\colon \tilde{X} \to X$ of $(X,\Delta)$ and an open subset $V \subset X$ such that $\boldsymbol{\rm M}$ descends to $\tilde{X}$, $f$ is an isomorphism over $V$, and $V$ contains the generic point of any generalized lc center of $(X,\Delta,\boldsymbol{\rm M})/Z$. 
\end{enumerate} 
\end{thm}

We note that we do not need to assume that $X$ or $Z$ is quasi-projective. 

The following result is the key ingredient for the proof. 

\begin{lem}[{cf.~\cite[Lemma 5.18]{hacon-liu}}]\label{lem--genklt-descend}
Let $(X,\Delta,\boldsymbol{\rm M})/Z$ be a generalized klt pair such that $\boldsymbol{\rm M}_{X}$ is $\mathbb{R}$-Cartier and $\boldsymbol{\rm M}$ is a finite $\mathbb{R}_{>0}$-linear combination of b-nef$/Z$ $\mathbb{Q}$-b-Cartier $\mathbb{Q}$-b-divisors. 
Then there is a projective birational morphism $g \colon Y \to X$ such that $\boldsymbol{\rm M}$ descends to $Y$ and the divisor $F:=g^{*}\boldsymbol{\rm M}_{X}-\boldsymbol{\rm M}_{Y}$ satisfies ${\rm Supp}F={\rm Ex}(g)$.  
\end{lem}

\begin{proof}
Since $\boldsymbol{\rm M}_{X}$ is $\mathbb{R}$-Cartier, $(X,\Delta)$ is a klt pair. 
Let $f\colon \bar{X} \to X$ be a log resolution of $(X,\Delta)$ such that $\boldsymbol{\rm M}$ descends to $\bar{X}$. 
We can write 
$$K_{\bar{X}}+\bar{\Delta}=f^{*}(K_{X}+\Delta)+\bar{E}$$
with $\bar{\Delta}\geq 0$ and $\bar{E}\geq0$ having no common components. 
We may write $\boldsymbol{\rm M}_{\bar{X}}=\sum_{i}\mu_{i}M_{i}$ where $\mu_{i} \in \mathbb{R}_{>0}$ and $M_{i}$ are Cartier divisors that are nef over $Z$. 
We fix $\alpha$ such that $\alpha \mu_{i}>2\cdot {\rm dim}X$ for all $i$. 

We fix an open affine covering $\{X_{j}\}_{j}$ of $X$. 
In particular, all $X_{j}$ are quasi-projective. 
We put $\bar{X}_{j}=f^{-1}(X_{j})$, $\bar{\Delta}_{j}=\bar{\Delta}|_{\bar{X}_{j}}$, and $\boldsymbol{\rm M}^{j}=\boldsymbol{\rm M}|_{X_{j}}$. 
In the rest of the proof, for any projective birational morphism $W \to X$ and any open subset $V \subset W$, we will denote $\boldsymbol{\rm M}_{W}|_{V}$ by $\boldsymbol{\rm M}_{V}$. 

Since $(\bar{X}_{j},\bar{\Delta}_{j},\alpha \boldsymbol{\rm M}^{j})/X_{j}$ is generalized klt, by applying \cite{bchm} we get a sequence of steps of a $(K_{\bar{X}_{j}}+\bar{\Delta}_{j}+\alpha \boldsymbol{\rm M}_{\bar{X}_{j}})$-MMP over $X_{j}$ that terminates with a good minimal model $(\bar{X}'_{j},\bar{\Delta}'_{j},\alpha \boldsymbol{\rm M}^{j})/X_{j}$ of $(\bar{X}_{j},\bar{\Delta}_{j},\alpha \boldsymbol{\rm M}^{j})/X_{j}$. 
By the length of extremal rays (\cite[Section 18]{fujino-fund} or \cite[Theorem 4.6.2]{fujino-book}), it follows that the birational transform of $\boldsymbol{\rm M}_{\bar{X}_{j}}$ is numerically trivial with respect to the extremal contraction in each step of the MMP. 
Therefore, we see that 
\begin{itemize}
\item
$\boldsymbol{\rm M}^{j}$ descends to $\bar{X}'_{j}$, and
\item
$\bar{X}_{j}\dashrightarrow \bar{X}'_{j}$ is a sequence of steps of a $(K_{\bar{X}_{j}}+\bar{\Delta}_{j}+(\alpha+t) \boldsymbol{\rm M}_{\bar{X}_{j}})$-MMP over $X_{j}$ to a minimal model $(\bar{X}'_{j},\bar{\Delta}'_{j}, (\alpha+t) \boldsymbol{\rm M}^{j})/X_{j}$ for all $t\geq 0$. 
\end{itemize}
Then $K_{\bar{X}'_{j}}+\bar{\Delta}'_{j}+(\alpha+t) \boldsymbol{\rm M}_{\bar{X}'_{j}}$ is semi-ample over $X_{j}$.  

Let $\phi \colon \bar{X}'_{j} \to Y_{j}$ be the contraction over $X_{j}$ induced by $K_{\bar{X}'_{j}}+\bar{\Delta}'_{j}+2 \alpha\boldsymbol{\rm M}_{\bar{X}'_{j}}$, and let 
$\bar{X}'_{j} \to Y'_{j}$ be the contraction over $X_{j}$ induced by $K_{\bar{X}'_{j}}+\bar{\Delta}'_{j}+ \alpha\boldsymbol{\rm M}_{\bar{X}'_{j}}$. 
We pick any curve $\xi \subset \bar{X}'_{j}$ that is contracted by $\bar{X}'_{j} \to X_{j}$. 
Since $K_{\bar{X}'_{j}}+\bar{\Delta}'_{j}+ \alpha\boldsymbol{\rm M}_{\bar{X}'_{j}}$ and $\alpha\boldsymbol{\rm M}_{\bar{X}'_{j}}$ are nef over $X_{j}$, if $\xi$ is contracted by $\phi \colon \bar{X}'_{j} \to Y_{j}$ then
$$0=\xi \cdot (K_{\bar{X}'_{j}}+\bar{\Delta}'_{j}+2 \alpha\boldsymbol{\rm M}_{\bar{X}'_{j}}) \geq \xi \cdot (K_{\bar{X}'_{j}}+\bar{\Delta}'_{j}+ \alpha\boldsymbol{\rm M}_{\bar{X}'_{j}})\geq 0.$$
Thus, if $\xi$ is contracted by $\phi \colon \bar{X}'_{j} \to Y_{j}$ then $\xi$ is contracted by $\bar{X}'_{j} \to Y'_{j}$. 
In particular, the induced birational map $Y_{j} \dashrightarrow Y'_{j}$ is a morphism. 
Then we have
$$\alpha \boldsymbol{\rm M}_{\bar{X}'_{j}}=(K_{\bar{X}'_{j}}+\bar{\Delta}'_{j}+2 \alpha\boldsymbol{\rm M}_{\bar{X}'_{j}})-(K_{\bar{X}'_{j}}+\bar{\Delta}'_{j}+\alpha\boldsymbol{\rm M}_{\bar{X}'_{j}})\sim_{\mathbb{R},Y_{j}}0,$$
so $\boldsymbol{\rm M}_{Y_{j}}$ is $\mathbb{R}$-Cartier and $\boldsymbol{\rm M}_{\bar{X}'_{j}}=\phi^{*}\boldsymbol{\rm M}_{Y_{j}}$. 
From this, we see that $\boldsymbol{\rm M}^{j}$ descends to $Y_{j}$. 

Put $\Delta_{Y_{j}}=\phi_{*}\bar{\Delta}'_{j}$. 
By construction, the generalized pair $(Y_{j}, \Delta_{Y_{j}},2 \alpha\boldsymbol{\rm M}^{j})/X_{j}$ is a weak generalized lc model of $(\bar{X}_{j},\bar{\Delta}_{j},2\alpha \boldsymbol{\rm M}^{j})/X_{j}$ such that $K_{Y_{j}}+\Delta_{Y_{j}}+2 \alpha\boldsymbol{\rm M}_{Y_{j}}$ is ample over $X_{j}$. 
This is a generalized pair analogue of the relative log canonical models for lc pairs. 
Hence we may glue $Y_{j}$, and we get a projective birational morphism $g\colon Y \to X$ and a birational contraction $\psi \colon \bar{X}\dashrightarrow Y$ over $X$ such that
\begin{itemize}
\item
$\boldsymbol{\rm M}$ descends to $Y$, and 
\item
$K_{Y}+\psi_{*}\bar{\Delta}+ 2\alpha \boldsymbol{\rm M}_{Y}$ is ample over $X$.  
\end{itemize} 
Since $\boldsymbol{\rm M}_{X}$ is $\mathbb{R}$-Cartier, we see that $K_{X}+\Delta+ 2\alpha \boldsymbol{\rm M}_{X}$ is $\mathbb{R}$-Cartier. 
Then there is an effective $g$-exceptional $\mathbb{R}$-divisor $E_{Y}$ on $Y$ such that 
$$K_{Y}+\psi_{*}\bar{\Delta}+ 2\alpha \boldsymbol{\rm M}_{Y}=g^{*}(K_{X}+\Delta+ 2\alpha \boldsymbol{\rm M}_{X})-E_{Y}$$
and ${\rm Supp}E_{Y}={\rm Ex}(g)$. 
We recall the relation $K_{\bar{X}}+\bar{\Delta}=f^{*}(K_{X}+\Delta)+\bar{E}$, so we have $K_{Y}+\psi_{*}\bar{\Delta}=g^{*}(K_{X}+\Delta)+\psi_{*}\bar{E}$. 
We put 
$$F=\frac{1}{2\alpha}(E_{Y}+\psi_{*}\bar{E}).$$ 
Then $F$ is effective and $g$-exceptional, ${\rm Supp}F \supset {\rm Ex}(g)$, and $2\alpha \boldsymbol{\rm M}_{Y}=g^{*}(2\alpha \boldsymbol{\rm M}_{X})-2\alpha F$. 
From these facts, we see that $g \colon Y \to X$ is the desired morphism.  
\end{proof}

From now on, we prove Theorem \ref{thm--equiv-gendlt}. 
\begin{proof}[Proof of Theorem \ref{thm--equiv-gendlt}]
Clearly, (\ref{3}) implies (\ref{1}), and (\ref{1}) implies (\ref{2}). 
Hence we only need to prove that (\ref{2}) implies (\ref{3}). 
Let $U \subset X$ be the open subset as in (\ref{2}). 
We will prove the existence of a projective birational morphism $g \colon Y \to X$ and an open subset $V \subset U$ such that $\boldsymbol{\rm M}$ descends to $Y$, $g$ is an isomorphism over $V$, and $V$ contains the generic point of any generalized lc center of $(X,\Delta,\boldsymbol{\rm M})/Z$. 
Assuming the existence of such $g$, then we can construct the desired log resolution $f\colon \tilde{X} \to X$ and the open subset $V\subset X$ by considering an appropriate log resolution of $(X,\Delta)$ factoring through $g$. 
Thus, we only need to prove the existence of $g \colon Y \to X$ and $V \subset U$ stated above. 

In this paragraph, we reduce the problem to the case where $K_{X}$ is $\mathbb{Q}$-Cartier and $\boldsymbol{\rm M}_{X}$ is $\mathbb{R}$-Cartier. 
By \cite[Proposition 4.12]{has-class}, with notations as in \cite{has-class}, the pair $\langle X,\Delta \rangle$ of $X$ and $\Delta$ is pseudo-lc in the sense of \cite[Definition 4.2]{has-class}. 
By \cite[Theorem 1.2]{has-class}, there exists a projective small birational morphism $h \colon W \to X$ from a normal variety $W$ such that putting $\Delta_{W}=h^{-1}_{*}\Delta$ then $K_{W}+\Delta_{W}$ is an $h$-ample $\mathbb{R}$-Cartier divisor on $W$. 
Note that \cite[Theorem 1.2]{has-class} can be applied without the quasi-projectivity of the variety. 
Since $h$ is small, $(K_{W}+\Delta_{W})|_{h^{-1}(U)}$ is the pullback of $(K_{X}+\Delta)|_{U}$, hence $h$ is an isomorphism over $U$. 
Since $h$ is small, we can check that $(W,\Delta_{W},\boldsymbol{\rm M})/Z$ is a generalized lc pair and $h^{-1}(U)$ contains the generic point of any generalized lc center of $(W,\Delta_{W},\boldsymbol{\rm M})/Z$. 
Then it is easy to see that $h^{-1}(U)$ satisfies the conditions of (\ref{2}). 
Thus, replacing $(X,\Delta,\boldsymbol{\rm M})/Z$ with $(W,\Delta_{W},\boldsymbol{\rm M})/Z$, we may assume that $K_{X}+\Delta$ is $\mathbb{R}$-Cartier. 
Then $\boldsymbol{\rm M}_{X}$ is $\mathbb{R}$-Cartier. 
With notations as in \cite{has-class}, the log canonicity of $(X,\Delta)$ and \cite[Lemma 4.3]{has-class} imply that the pair $\langle X,0 \rangle$ is pseudo-lc in the sense of \cite[Definition 4.2]{has-class}. 
By \cite[Theorem 1.2]{has-class}, we get a projective small birational morphism $h' \colon W' \to X$ from a normal variety $W'$ such that $K_{W'}$ is an $h'$-ample $\mathbb{R}$-Cartier divisor on $W'$. 
Then $K_{W'}|_{h'^{-1}(U)}$ is the pullback of $K_{X}|_{U}$, hence $h'$ is an isomorphism over $U$. 
It is easy to see that $(W',h'^{-1}_{*}\Delta,\boldsymbol{\rm M})/Z$ is a generalized lc pair and $h'^{-1}(U)$ satisfies the conditions of (\ref{2}). 
Replacing $(X,\Delta,\boldsymbol{\rm M})/Z$ with $(W',h'^{-1}_{*}\Delta,\boldsymbol{\rm M})/Z$, we may assume that $K_{X}$ is $\mathbb{Q}$-Cartier. 
In this way, we may assume that $K_{X}$ is $\mathbb{Q}$-Cartier and $\boldsymbol{\rm M}_{X}$ is $\mathbb{R}$-Cartier. 

By the definition of $U$, it follows that $(X,\Delta)$ is dlt. 
Hence $(X,0)$ is klt. 
By applying Lemma \ref{lem--genklt-descend} to the generalized klt pair $(X,\tfrac{1}{2}\Delta, \tfrac{1}{2}\boldsymbol{\rm M})/Z$, we get a projective birational morphism $g \colon Y \to X$ such that $\tfrac{1}{2}\boldsymbol{\rm M}$ descends to $Y$ and the divisor $F=\tfrac{1}{2}g^{*}\boldsymbol{\rm M}_{X}-\tfrac{1}{2}\boldsymbol{\rm M}_{Y}$ satisfies ${\rm Supp}F={\rm Ex}(g)$. 
Let $V_{0}\subset X$ be the largest open subset over which $g$ is an isomorphism. 
Let $\eta$ be the generic point of a generalized lc center of $(X,\Delta, \boldsymbol{\rm M})/Z$.
Then $\eta$ is the generic point of an lc center of $(U,\Delta|_{U})$, and $\eta$ is also the generic point of a generalized lc center of $(U,\Delta|_{U}, \boldsymbol{\rm M}|_{U})/U$ with the identity $U \to U$.  
Pick a prime divisor $P$ over $U$ whose center is $\overline{\{\eta\}}\cap U$ and the log discrepancy $a(P,U, \Delta|_{U})$ is zero. 
Then 
$$0 \leq a(P,U,\Delta|_{U}+ \boldsymbol{\rm M}_{U})\leq a(P,U, \Delta|_{U}) = 0.$$
In particular, we have $a(P,U,\Delta|_{U}+ \boldsymbol{\rm M}_{U}) = a(P,U, \Delta|_{U})$. 
Since ${\rm Supp}F={\rm Ex}(g)$, the center $c_{g^{-1}(U)}(P)$ of $P$ on $g^{-1}(U)$ is not contained in ${\rm Ex}(g)$. 
Then $c_{g^{-1}(U)}(P)$ intersects $g^{-1}(V_{0})$, hence we see that $\eta \in V_{0}$. 
 
We set $V=V_{0}\cap U$. 
Then $g$ is an isomorphism over $V$. 
By the above discussion, $V$ contains the generic point of any generalized lc center of $(X,\Delta, \boldsymbol{\rm M})/Z$. 
By the property of $U$ in (\ref{2}), we see that $(V, \Delta|_{V})$ is log smooth. 
In this way, we get a projective birational morphism $g \colon Y \to X$ and an open subset $V \subset U$ such that $\boldsymbol{\rm M}$ descends to $Y$, $g$ is an isomorphism over $V$, and $V$ contains the generic point of any generalized lc center of $(X,\Delta,\boldsymbol{\rm M})/Z$. 
From this fact, (\ref{2}) implies (\ref{3}). 
We complete the proof. 
\end{proof}


\end{document}